\documentclass[10pt,english,reqno]{amsart}

\newcommand{\da}{{\hat{\alpha}}}
\newcommand{\db}{{\hat{\beta}}}
\newcommand{\wh}{\widehat}

\newcommand{\BA}{\mathbb{A}} 
\newcommand{\BP}{\mathbb{P}} 
\newcommand{\wpair}[1]{\left\{{#1}\right\}}
\newcommand{\ov}{\overline} 
\newcommand{\wt}{\widetilde} 
\newcommand{\bpm}{\begin{pmatrix}}
\newcommand{\epm}{\end{pmatrix}}

\usepackage[T1]{fontenc}

\usepackage{amssymb,url,xspace,amsthm,adjustbox}

\usepackage[alphabetic]{amsrefs}

\usepackage{mathtools}
\usepackage{graphicx}

\usepackage{float}

\author[C. Cunningham]{Clifton Cunningham}
\address{University of Calgary}
\email{cunning@math.ucalgary.ca}
\thanks{Clifton Cunningham gratefully acknowledges the support of NSERC Discovery Grant 696185 and additional support from the Pacific Institute for the Mathematical Sciences (PIMS)}

\author[A. Fiori]{Andrew Fiori}
\address{University of Lethbridge}
\email{andrew.fiori@uleth.ca}
\thanks{Andrew Fiori thanks and acknowledges the University of Lethbridge for their financial support.}

\author[Q. Zhang]{Qing Zhang}
\address{University of Calgary}
\email{qing.zhang1@ucalgary.ca}
\thanks{Qing Zhang would like to thank the support NSFC grant 11801577, a postdoc fellowship from the Pacific Institute for the Mathematical Sciences (PIMS) and financial support from University of Calgary.}

\thanks{All of the authors would like to thank the Banff International Research Station (BIRS) for hosting the 2019 Focused Research Group 19frg277 on The Voganish Project}


\usepackage{datetime}
\usepackage[T1]{fontenc}
\usepackage{chngcntr}

\usepackage{amssymb}
\usepackage{mathrsfs,stmaryrd}
\usepackage{yfonts, bbm}
\usepackage{enumitem}

\usepackage{hyperref}
\hypersetup{
  colorlinks   = true, 
  urlcolor     = blue, 
  linkcolor    = blue, 
  citecolor   = green 
}

\usepackage{tikz}
\usetikzlibrary{shapes,arrows,calc,matrix}
\usepackage{tikz-cd}
\usepackage{todonotes}
\usepackage{color}

\usepackage{amsmath}
\usepackage{lipsum}
\usepackage{setspace}

\theoremstyle{plain}
      \newtheorem{theorem}{Theorem}[section]
      \newtheorem{proposition}[theorem]{Proposition}
      \newtheorem{lemma}[theorem]{Lemma}

      \theoremstyle{definition}

      \newtheorem{remark}[theorem]{Remark}

      \newtheorem*{conjecture*}{Conjecture}

\counterwithin{table}{subsection}

\newcommand{\pair}[1]{\langle {#1} \rangle}

\newcommand{\FF}{{\mathbb{F}}}
\newcommand{\ZZ}{{\mathbb{Z}}}

\newcommand{\CC}{{\mathbb{C}}}

\newcommand{\RR}{{\mathbb{R}}}

\DeclareMathOperator{\GL}{GL}
\DeclareMathOperator{\Sym}{Sym}
\newcommand{\G}{\textbf{G}}
\DeclareMathOperator{\SL}{SL}

\DeclareMathOperator{\PGL}{PGL}
\DeclareMathOperator{\Sp}{Sp}

\DeclareMathOperator{\PGSpin}{PGSpin}
\DeclareMathOperator{\SO}{SO}
\DeclareMathOperator{\PSO}{PSO}

\newcommand{\Lgroup}[1]{{\hskip-2 pt \,^L\hskip-1pt{#1}}}
\newcommand{\dualgroup}[1]{{\widehat{#1}}}

\newcommand{\Frob}{{\operatorname{Fr}}}

\DeclareMathOperator{\id}{id}

\DeclareMathOperator{\trace}{trace}

\newcommand{\Spec}[1]{{\operatorname{Spec}(#1)}}
\DeclareMathOperator{\Ad}{Ad}

\newcommand{\abs}[1]{{\vert #1 \vert}}
\newcommand{\ceq}{{\, :=\, }}
\newcommand{\tq}{{\ \vert\ }}
\newcommand{\iso}{{\ \cong\ }}

\DeclareMathOperator{\diag}{diag}

\newcommand{\Perv}{\mathsf{Per}}

\newcommand{\Loc}{\mathsf{Loc}}
\newcommand{\Rep}{\mathsf{Rep}}

\newcommand{\Ev}{\operatorname{\mathsf{E}\hskip-1pt\mathsf{v}}}
\newcommand{\Evs}{\operatorname{\mathsf{E}\hskip-1pt\mathsf{v}\hskip-1pt\mathsf{s}}}

\newcommand{\NEvs}{\operatorname{\mathsf{N}\hskip-1pt\mathsf{E}\hskip-1pt\mathsf{v}\hskip-1pt\mathsf{s}}}

\newcommand{\IC}{{\mathcal{I\hskip-1pt C}}}

\newcommand{\RPhi}{{\mathsf{R}\hskip-0.5pt\Phi}}

\newcommand{\Ft}{\operatorname{\mathsf{F\hskip-1pt t}}}

\makeatletter
\newcommand{\labitem}[2]{
\def\@itemlabel{\textbf{#1}}
\item
\def\@currentlabel{#1}\label{#2}}
\makeatother

\newcommand{\1}{{\mathbbm{1}}}

\newcommand{\Ind}{\text{Ind}}

\newcommand{\C}{\mathbb{C}}

\newcommand{\Aut}{\text{Aut}}

\renewcommand{\Im}{\operatorname{Im}}

\newcommand{\s}{{s}}

\newcommand{\KPair}[2]{( \, #1\, \vert\, #2\, )}

\newcommand{\Lie}{\operatorname{Lie}}

\title[Arthur packets for $G_2$ and perverse sheaves on cubics]
{Arthur packets for $G_2$ and perverse sheaves on cubics}

\date{\today}                                           

\begin{document}

\begin{abstract}
This paper begins the project of defining Arthur packets of all unipotent representations for the $p$-adic exceptional group $G_2$. Here we treat the most interesting case by defining and computing Arthur packets with component group $S_3$. We also show that the distributions attached to these packets are stable, subject to a hypothesis. This is done using a self-contained microlocal analysis of simple equivariant perverse sheaves on the moduli space of homogeneous cubics in two variables. In forthcoming work we will treat the remaining unipotent representations and their endoscopic classification and strengthen our result on stability.
\end{abstract}

\maketitle



\section{Introduction}

Appearing first in \cite{Arthur:Conjectures}, Arthur packets play the lead role in Arthur's endoscopic classification \cite{Arthur:book} of admissible representations of the symplectic and orthogonal groups $\Sp_{2n}$, $\mathrm{O}_{2n}$ and $\SO_{2n+1}$ over local fields.
When combined with the work of many others, especially on the stable trace formula, Arthur's endoscopic classification is the keystone in the proof of the local Langlands correspondence for these groups.
While this classification has been extended to other classical groups in \cite{Mok:Unitary} and \cite{KMSW:Unitary}, work remains to be done in several directions, including extending the endoscopic classification to all forms of classical groups, spin groups, exceptional groups and generalising the notion of Arthur packets to include admissible representations that are not of Arthur type.
We are interested in extending Arthur's endoscopic classification of admissible representations to the split exceptional group $G_2$ and also extending the classification to include admissible representations that are not of Arthur type.

In this paper we begin the project of finding the endoscopic classification of unipotent representations of $G_2(F)$, where $F$ is a $p$-adic field; see \cite{Lusztig:Classification1}*{Section 0.3} for the definition of unipotent representations.
These representations correspond to unramified Langlands parameters $G_2$.
Since a local Langlands correspondence is available for unipotent representations of $G_2$ as a particular case of \cite{Lusztig:Classification1}, what we seek, first, is an answer to the following question.
\begin{quotation}{\it
If $\phi: W_F' \to \Lgroup{G}_2$ is an unramified Langlands parameter of Arthur type $\psi$, how can we define a packet $\Pi_\psi(G_2(F))$ of irreducible representations that includes the L-packet $\Pi_\phi(G_2(F))$ and a function 
\[
\begin{array}{rcl}
\langle \ , \ \rangle_\psi : \Pi_{\psi}(G_2(F)) &\to& {\widehat A_\psi} \\
\pi &\mapsto& \langle \ , \pi\rangle_\psi,
\end{array}
\]
that has the properties predicted by Arthur's main local result \cite{Arthur:book}*{Theorem 1.5.1} and also such that
\[
\Theta_{\psi} \ceq \sum_{\pi \in\Pi_{\psi}(G_2(F))} \langle a_\psi , \pi\rangle_\psi \ \Theta_{\pi}
\]
is stable, for some $a_\psi \in A_\psi$?
}\end{quotation}

Arthur's work shows that the function $\langle \ , \ \rangle_\psi$ is canonically determined by endoscopic character relations, derived mainly from the stable trace formula.   
We approach this problem by using \cite{CFMMX} to construct candidate packets from which stable distributions can be formed; in forthcoming work we show that these packets are compatible, in a precise sense, with endoscopy and twisted endoscopy.
We believe that this justifies calling these packets Arthur packets.
In the case of $G_2$, we will show that these candidate packets satisfy conditions coming from the fact that $G_2$ is a twisted endoscopic group of $\PSO_8$; likewise, we will show that these candidate packets satisfy conditions related to the endoscopic subgroups of $G_2$, which are $\PGL_3$, split $\SO_4$, $\GL_2$ and $\GL_1\times\GL_1$. 
Our approach to finding these candidate packets relies on ideas that can be traced back to Vogan \cite{Vogan:Langlands} and are developed further in \cite{CFMMX}.
The candidate packets we construct are called ABV-packets, defined in \cite{CFMMX}.

Our goal in this paper is to construct ABV-packets for $p$-adic $G_2$ in the most complicated case, namely, the case when these candidate packets contain representations for which the component  group of the corresponding Langlands parameter is the symmetric group $S_3$.
In subsequent work we construct ABV-packets for the remaining unipotent representations of $G_2(F)$ and show that they satisfy the conditions imposed by twisted endoscopic restriction from $\PSO_8(F)$ and by endoscopic restriction to $\PGL_3(F)$ and $\SO_4(F)$.  
Even though this paper does not complete that process, we prejudicially refer to ABV-packets as Arthur packets in the title of this paper, begging the reader's indulgence until forthcoming work is finished.

\subsection{Langlands parameter of $G_2$ with component group $S_3$}  Let $F$ be a $p$-adic field.
Up to $\dualgroup{G}_2$-conjugacy, the $p$-adic exceptional group $G_2(F)$ admits exactly one Langlands parameter  $\phi_{3} : W_F \times \SL_2(\CC) \to \Lgroup{G_2}$ with component group $S_3$; this parameter arises from the subregular nilpotent orbit of $\Lie\widehat{G}_2$. 
 The L-packet for $\phi_3$ is
\[
\Pi_{\phi_3}(G_2(F)) = \{ \pi_3, \pi_3^\varrho, \pi_3^\varepsilon \},
\]
where $\pi_3$ is the admissible representation which corresponds to the trivial representation of $S_3$, $\pi_3^{\varrho}$ corresponds to the unique 2-dimensional irreducible representation $\varrho$ of $S_3$ and $\pi_3^{\varepsilon}$ corresponds to the sign character $\varepsilon$ of $S_3$.
It is expected that the distribution
\[
\Theta_{\phi_3} \ceq \Theta_{\pi_3} + 2\Theta_{\pi_3^\varrho} + \Theta_{\pi_3^\varepsilon},
\]
attached to the L-packet $\Pi_{\phi_3}(G_2(F))$ is stable, where $\Theta_{\pi}$ denotes the Harish-Chandra distribution character determined by the admissible representation $\pi$.
Note that the representations in this L-packet are all discrete series representations (so, in particular, they are tempered) and depth-zero, but only $\pi_3^\varepsilon$ is supercuspidal. 
Note the appearance of $2 = \dim \varrho$ as a coefficient in this linear combination.
The exceptional group $G_2$ is the smallest rank group for which a multiplicity of greater than $1$ appears in the expected stable distribution attached to an Arthur packet.

These three unipotent representations are tempered and $\pi_3$ and $\pi_3^\varrho$ have nonzero Iwahori-fixed vectors while $\pi_3^{\varepsilon}$ is supercuspidal. 
Using a Chevalley group scheme for $G_2$, $\pi_3^\varepsilon$ is given by compact induction
\[
\pi_3^\varepsilon \ceq \operatorname{cInd}_{G_2(\mathcal{O}_F)}^{G_2(F)} G_2[1];
\]
here, $G_2[1]$ is the cuspidal unipotent representation of $G_2(\FF_q)$ appearing in \cite{Carter}*{page 460} and $G_2(\mathcal{O}_F)$ is the maximal compact subgroup of $G_2(F)$.
%
This supercuspidal representation is of great interest for a variety of reasons. 
For instance, while there are four irreducible supercuspidal unipotent representations of $G_2(F)$, only $\pi_3^\varepsilon$ has the property that its corresponding Langlands parameter is trivial on $W_F$. 
However, our interest in the representation $\pi_3^{\varepsilon}$ lies in the collection of A-packets that contain it, and the main results of this paper concern those representations. 

\subsection{Main results}

In order to state our main result, recall that, for any Langlands parameter $\phi : W_F'\to \Lgroup{G}$, the infinitesimal parameter $\lambda_\phi : W_F'\to \Lgroup{G}$ is defined by
\[
\lambda_{\phi}(w)=\phi\left(w, \bpm |w|^{1/2} & 0 \\ 0 &|w|^{-1/2} \epm \right).
\]
Let $\lambda_{\mathrm{sub}}:=\lambda_{\phi_3}:W_F\to \Lgroup{G_2}$ be the infinitesimal parameter of $\phi_3$.
It is natural to consider all Langlands parameters $\phi$ such that $\lambda_\phi=\lambda_{\mathrm{sub}}$.
In this paper we use the Bernstein decomposition to find an intrinsic description of category $\Rep(G_2(F))_{\mathrm{sub}}$ of all representations of $G_2(F)$ with infinitesimal parameter $\lambda_{\mathrm{sub}}$.
It turns out that every Langlands parameter $\phi$ with $\lambda_\phi=\lambda_{\mathrm{sub}}$ is of Arthur type, {\it i.e.}, there exists an Arthur parameter $\psi:W_F\times \SL_2(\CC)\times \SL_2(\CC)\to \Lgroup{G_2}$ such that $\phi=\phi_\psi$, where 
\[
\phi_\psi(w,x):=\psi\left(w,x, \bpm |w|^{1/2} & 0 \\ 0 &|w|^{-1/2} \epm \right).
\]
For an Arthur parameter $\psi$, denote by $A_{\psi}:=\pi_0(Z_{\wh{G_2}}(\Im(\psi)))$ the component group of $\psi$. 
For an Arthur parameter $\psi$, we refer to $\lambda_{\phi_\psi}$ as its infinitesimal parameter.

We write ${\wh A_\psi}$ for the set of equivalence classes of irreducible representations of $A_\psi$; below, we will identify a representation of $A_\psi$ with its character.

The two main theorems of this paper are the following.

\begin{theorem}[See also Theorem~\ref{thm:main}]\label{intromaintheorem}
For each Arthur parameter $\psi$ with infinitesimal parameter $\lambda_{\mathrm{sub}}$, there exists a finite set $\Pi_{\psi}(G_2(F))$ of irreducible unipotent representations and a function 
\[
\begin{array}{rcl}
\langle \ , \ \rangle_\psi : \Pi_{\psi}(G_2(F)) &\to& {\widehat A_\psi} \\
\pi &\mapsto& \langle \ , \pi\rangle_\psi,
\end{array}
\]
such that
\begin{enumerate} 
\labitem{(a)}{intromain:tempered} if $\psi$ is trivial on $\SL_2(\C)$ then all the representations in $\Pi_{\psi}(G_2(F))$ are tempered and $\langle \ , \ \rangle_\psi$ is bijective;
\labitem{(b)}{intromain:nontempered} if $\psi$ is not trivial on $\SL_2(\C)$ then $\Pi_\psi(G_2(F))$ contains non-tempered representations and $\langle \ , \ \rangle_\psi$ is not necessarily bijective;
\labitem{(c)}{intromain:spherical} if $\pi$ is spherical then $\langle \ , \pi \rangle_\psi = \1$, the trivial representation of $A_\psi$.
\end{enumerate}
\end{theorem}

\begin{remark}\label{remark:spherical}
An admissible representation $\pi$ of $G_2(F)$ is spherical iff it has a non-trivial vector which is fixed by the maximal compact $G_2(\mathcal{O}_F)$. 
Such representations are sometimes called ``unramified'' representations.
For every spherical $\pi$, there is a unique Arthur parameter $\psi$ such that $\pi \in \Pi_\psi(G_2(F))$.
\end{remark}

The packets $\Pi_\psi(G_2(F))$ are displayed in Table~\ref{table:ALpackets}; note that each one contains $\pi_3^\varepsilon$! 
As further justification for referring to $\Pi_\psi(G_2(F))$ as an Arthur packet, we study the stability of these distributions $\Theta_{\psi}$.
 At present, our result in this direction, Theorem~\ref{thm:stable} is conditional.
To state the hypothesis of this theorem, we make a definition:
let us say that an Arthur parameter $\psi : W'_F \times \SL_2(\C)\to \Lgroup{G}_2$ is tempered if the Langlands parameter $\phi_\psi$ is tempered, which means $\phi_\psi$ is bounded upon restriction to $W_F$.

\begin{theorem}[See also Theorem~\ref{thm:stable}]\label{introstabletheorem}
For every Arthur parameter $\psi$ with subregular infinitesimal parameter, consider the invariant distribution 
\begin{equation}
\Theta_{\psi} \ceq \sum_{\pi \in\Pi_{\psi}(G_2(F))} \langle a_\psi , \pi\rangle_\psi \ \Theta_{\pi},
\end{equation}
where $a_\psi$ is the image of $s_\psi \ceq \psi(1,-1)$ in $A_\psi$.
Suppose $\Theta_{\psi}$ is stable when $\psi$ is tempered.
Then the distributions $\Theta_\psi$ are stable for all Arthur parameters $\psi$.
\end{theorem}

We expect to prove that $\Theta_{\psi}$ is stable when $\psi$ is tempered and therefore remove 
this hypothesis from Theorem~\ref{thm:stable}, making that result unconditional.

In forthcoming work we fully justify calling $\Pi_\psi(G_2(F))$ an Arthur packet by showing that $\Pi_{\psi}(G_2(F))$ also satisfies properties relating to (twisted) endoscopy for $G_2$; see also Section~\ref{ssec:forthcoming}. 
As mentioned above, we believe that these two properties of the packets $\Pi_\psi(G_2(F))$ -- stability of the associated distribution and compatibility with endoscopy and twisted endoscopy -- justify calling these Arthur packets.  

\subsection{Strategy of the proof}
The Langlands correspondence for unipotent representations of $G_2(F)$ determines a bijection between isomorphism classes of simple objects in two categories:
\begin{equation}\label{eqn:introLVC}
\begin{array}{rcl}
\mathcal{P} : \Big( \Rep(G_2(F))_\mathrm{sub} \Big)^\mathrm{simple}_{/\mathrm{iso}}
&\to&
\left( \Perv_{\dualgroup{G_2}}(X_\mathrm{sub})  \right)^\mathrm{simple}_{/\mathrm{iso}}\\
\pi &\mapsto& \mathcal{P}(\pi) , 
\end{array}
\end{equation}
where $X_\mathrm{sub}$ is the moduli space of Langlands parameters $\phi: W'_F \to \dualgroup{G_2}$ for which $\lambda_\phi$ is $\dualgroup{G_2}$-conjugate to $\lambda_\mathrm{sub}\ceq \lambda_{\phi_3}$ and where $\Perv_{\dualgroup{G_2}}(X_\mathrm{sub}) $ is the category of equivariant perverse sheaves on $X_\mathrm{sub}$; see Proposition~\ref{prop:LVC}.
We use the ``microlocal vanishing cycles'' functor introduced in \cite{CFMMX} to define
\[
\Evs_{\psi} : \Perv_{\dualgroup{G_2}}(X_\mathrm{sub}) \to {\widehat A_\psi}.
\]

After passing to isomorphism classes of simple objects, the packet $\Pi_\psi(G_2(F))$ is defined as support of $\Evs_{\psi}$:
\[
\Pi_\psi(G_2(F)) =\{ \pi\in \left( \Rep(G_2(F))_\mathrm{sub} \right)^\mathrm{simple}_{/\mathrm{iso}}:  \Evs_{\psi} \mathcal{P}(\pi)\ne 0 \}.
\]
 By restriction, $\Evs_{\psi}$ gives the function $\langle \ , \ \rangle_\psi : \Pi_{\psi}(G_2(F)) \to {\widehat A_\psi}$ appearing in Theorem~\ref{intromaintheorem}.
Finally we use $\Evs_{\psi}$ to define the stable distributions appearing in Theorem~\ref{introstabletheorem}.

It is important to observe that even when a Langlands parameter, $\phi$, is not of Arthur type one can still define a packet $\Pi_{\phi}^{\mathrm{ABV}}$ using the functor $\Evs$ from \cite{CFMMX}. 
This will be important in our forthcoming work since not all unipotent representations of $G_2(F)$ are of Arthur type.

In the proof of Theorem \ref{intromaintheorem} we use \eqref{eqn:introLVC} and also establish a tighter connection between these categories, known to us as the Kazhdan-Lusztig conjecture. 
Recall that any irreducible representation $\pi$ can be written as a Langlands quotient of $M_\pi$, where  $M_\pi$ is  the standard module of $\pi$.

\begin{proposition}[See also Proposition~\ref{prop:KL}]\label{prop:KazhdanLusztig}
For any irreducible admissible $\pi$ with infinitesimal parameter $\lambda_\mathrm{sub}$, let $\mathcal{P}(\pi) = j_{!*}\mathcal{L}_\pi[\dim \mathfrak{C}_\pi]$ be the corresponding simple perverse sheaf under \eqref{eqn:introLVC}, where $j : \mathfrak{C}_\pi \hookrightarrow X_\mathrm{sub}$ is inclusion. Let $\mathcal{L}_\pi^!\ceq j_!\mathcal{L}_\pi$, be the standard sheaf defined by $\mathcal{L}_\pi$.
Then, the multiplicity of $\pi'$ in $M_\pi$ is equal to the multiplicity of $\mathcal{L}_\pi^!$ in $\mathcal{P}(\pi')$.
\end{proposition}
See Proposition~\ref{prop:KL} for unexplained notations and a more precise statement.

In Section~\ref{ssec:Ft} we also define a Fourier transform $\Ft$ functor on $\Perv_{\dualgroup{G_2}}(X_\mathrm{sub})$ and show that it corresponds with the Aubert involution $D_{G_2}$ defined in  \cite{Aubert:Dualite} on the Grothendieck group $K\Rep(G_2(F))_{\mathrm{sub}}$ of $\Rep(G_2(F))_{\mathrm{sub}}$.

\begin{theorem}[See also Theorem~\ref{thm:Ft}]\label{thm:aubertinvolution}
The Aubert involution on $K \Rep(G_2(F))_\mathrm{sub} $ commutes with the Fourier transform on $K\Perv_{\dualgroup{G_2}}(X_\mathrm{sub}) $ under the correspondence \eqref{eqn:introLVC}, {\it i.e.}, the diagram
\[
\begin{tikzcd}
K \Rep(G_2(F))_\mathrm{sub}\arrow{rr}{D_{G_2}} \arrow{d}{\mathcal{P}}& & \arrow{d}{\mathcal{P}}  K \Rep(G_2(F))_\mathrm{sub} \\
K\Perv_{\dualgroup{G_2}}(X_\mathrm{sub})\arrow{rr}{\Ft} && K\Perv_{\dualgroup{G_2}}(X_\mathrm{sub})
\end{tikzcd}
\]
commutes.
\end{theorem}

The two theorems above suggest that there is a stronger connection between these two categories, $\Rep(G_2(F))_{\mathrm{sub}}$ and $\Perv_{\wh{G_2}}(X_{\mathrm{sub}})$ than \eqref{eqn:introLVC} might suggest on its own. See also Remark~\ref{remark:CLLC}.

\subsection{Geometric calculations}

We find an equivalence of categories
\[
\Perv_{\dualgroup{G_2}}(X_\mathrm{sub}) \iso \Perv_{\GL_2}(\det\!\,^{-1}\otimes \Sym^3),
\]
where $\det^{-1}\otimes \Sym^3$ is the twist by $\det^{-1}$ of the canonical four-dimensional representation of $\GL_2$ on the vector space $P_3[x,y]$ of homogeneous cubic polynomials in two variables. The proof of our main theorems relies on a microlocal study of the simple $\GL_2$-equivariant perverse sheaves on $\det\!\,^{-1}\otimes \Sym^3$. Since we believe such a purely geometric problem might be of some independent interest, we present it in a self-contained way so that it can be read without reference to our application. As such, Section~\ref{sec:cubics} is precisely a study of the microlocal geometry of homogeneous cubics in two variables.

More precisely, using \cite{CFMMX}, in Section~\ref{sec:cubics} we compute the functor 
\[
\Evs : \Perv_{\GL_2}(\det\!\,^{-1}\otimes \Sym^3) \to \Loc_{\GL_2}(\Lambda^\mathrm{reg})
\]
on simple objects, where $\Lambda^\mathrm{reg}$ is the regular part of the conormal bundle 
\[
\Lambda = T^*_{\GL_2}(\det\!\,^{-1}\otimes \Sym^3).
\]
The functor $\Evs$ is full and faithful and so it actually converts equivariant perverse sheaves on $\det\!\,^{-1}\otimes \Sym^3$ into local systems on the regular part of the conormal bundle of $\det\!\,^{-1}\otimes \Sym^3$. 
Not only is this a great tool for understanding $\Perv_{\GL_2}(\det\!\,^{-1}\otimes \Sym^3)$ itself but, as explained in \cite{CFMMX}, it is the key to understanding ABV-packets of admissible representations.
We refer to $\Evs$ as a microlocal vanishing cycles functor because, for every $(r,s)\in \Lambda^\mathrm{reg}$,
\[
\left(\Evs \mathcal{F}\right)_{(r,s)}
=
\left(\RPhi_s[-1]\mathcal{F}\right)_r[\dim C_r^*][-\dim (\det\!\,^{-1}\otimes \Sym^3)],
\]
where $C_r^*$ is the dual orbit to the $\GL_2$-orbit $C_r$ of the cubic $r$ in $\det\!\,^{-1}\otimes \Sym^3$ and where $\RPhi_s$ is the vanishing cycles functor.

In Section~\ref{sec:cubics} we also study the Fourier transform on $\Perv_{\GL_2}(\det\!\,^{-1}\otimes \Sym^3)$, which interacts nicely with $\Evs$. 

\subsection{Structure of the paper}
 
The proofs in Section~\ref{sec:main} rely on results from Section~\ref{sec:cubics}.
Geometers who feel that a microlocal study of $\Perv_{\GL_2}(\det\!\,^{-1}\otimes \Sym^3)$ requires no motivation may enjoy reading Section~\ref{sec:cubics} first and then moving to Section~\ref{sec:main} as an application of the geometric results.

\subsection{Forthcoming work}\label{ssec:forthcoming}

In a sequel to this paper, \cite{CFZ:unipotent}, we consider the remaining unipotent representations of $p$-adic $G_2(F)$ and calculate the ABV-packets for these representations. 
Then, in \cite{CFZ:endoscopic} we show that ABV-packets for all unipotent representations of $G_2$ are indeed Arthur packets by showing that they satisfy the conditions imposed on them by the theory of endoscopy and twisted endoscopy. More precisely, \cite{CFZ:endoscopic} uses \cite{Huang-Magaard-Savin} and twisted endoscopic transfer from $\PSO(8)$, one for each unramified cubic etale algebra $E/F$, to $G_2$ over $F$; it also uses endoscopic transfer from $G_2$ to $\PGL_3$ and split $\SO_4$;
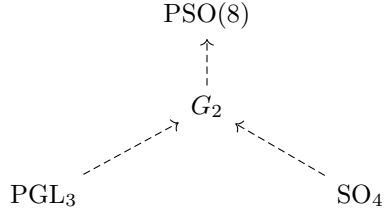
\[
\begin{tikzcd}
& \PSO(8) & \\
& G_2 \arrow[dashed]{u} & \\
\PGL_3 \arrow[dashed]{ur} && \arrow[dashed]{ul}  \SO_4 
\end{tikzcd}
\]

\subsection{Other exceptional groups}\label{ssec:other}

We have already remarked that $\pi_3^\varepsilon$ is the only irreducible supercuspidal unipotent representation of $G_2(F)$ for which the corresponding Langlands parameter is trivial on $W_F$. 
This "superunipotent" representation $\pi_3^\varepsilon$ is actually the first of a triple of such representations.
We have seen that the exceptional group $G_2$ admits a unique supercuspidal unipotent representation $\pi_3^{\varepsilon}$ with component $S_3$ and enhanced Langlands parameter $(\phi_3,\varepsilon)$ where $\varepsilon$ is the sign character of $S_3$. 
For large residual characteristic, the exceptional group $F_4$ admits a unique supercuspidal unipotent representation $\pi_4^{\varepsilon}$ with component group $S_4$ and enhanced Langlands parameter $(\phi_4,\varepsilon)$ where $\varepsilon$ is the sign character of $S_4$.
Likewise, for large residual characteristic, $E_8$ admits a unique supercuspidal unipotent representation $\pi_5^{\varepsilon}$ with component group $S_5$ and enhanced Langlands parameter $(\phi_5,\varepsilon)$ where $\varepsilon$ is the sign character of $S_5$.
In each case, the Langlands parameter is of Arthur type and therefore determined by a unipotent orbit in the appropriate split exceptional group; these are given in Table~\ref{table:triple}. 
The $L$-packets containing these representations have cardinality $3$, $5$ and $7$, respectively, since there are $3$, $5$ and $7$ irreducible representations of the symmetric groups $S_3$, $S_4$ and $S_5$, respectively.

\begin{table}[htp]
\caption{A triple of ``superunipotent" representations}
\begin{center}
$
\begin{array}{cccccc}
\mathrm{split} & \mathrm{unipotent}  & \mathrm{fund'l} &\mathrm{irred} & \mathrm{superunipotent} & \mathrm{Lusztig's} \\ 
\mathrm{group} & \mathrm{orbit}  & \mathrm{group} & \mathrm{rep'n} & \mathrm{representation} & \mathrm{classification} \\ 
 & \text{\cite{Carter}} &  &  &  & \text{\cite{Lusztig:Classification1}} \\
\hline
G_2 & G_2(a_1) & S_3 & \varepsilon \mathrm{\ (sign)} &\pi_3^\varepsilon = \operatorname{cInd}_{G_2(\mathcal{O}_F)}^{G_2(F)} G_2[1]& 7.33 \\
F_4 &  F_4(a_3) & S_4 & \varepsilon \mathrm{\ (sign)} & \pi_4^\varepsilon = \operatorname{cInd}_{F_4(\mathcal{O}_F)}^{F_4(F)} F^{\mathrm{II}}_4[1] & 7.26 \\
E_8 & E_8(a_7) & S_5 & \varepsilon \mathrm{\ (sign)} & \pi_5^\varepsilon  = \operatorname{cInd}_{E_8(\mathcal{O}_F)}^{E_8(F)} E^{\mathrm{II}}_8[1] & 7.1 
\end{array}
$
\end{center}
\label{table:triple}
\end{table}%

We remark on one final common feature of these examples: the groups involved, $G_2$, $F_4$, and $E_8$ are the only three semisimple algebraic groups which are simply connected, adjoint and have no outer automorphisms. It is a direct consequence that over a $p$-adic field there is a unique form, which must be split.
For $G_2$ this is explained in \cite{FioriG2}*{Section 2--4}, but follows also from \cite{Involutions}*{Theorem 26.19, Proposition 33.24} and \cite{PR}*{Theorem 6.4}.
The proofs for the other cases are analogous.

We have begun a study of the supercuspidal unipotent representation $\pi_4^{\varepsilon}$, parallel to the study of $\pi_3^\varepsilon$ in this paper. 
While Langlands parameters with the same infinitesimal parameter as $\pi_3^{\varepsilon}$ are classified by orbits of homogeneous cubics in two variables, 
Langlands parameters with the same infinitesimal parameter as $\pi_4^{\varepsilon}$ are classified by orbits of pencils of quadratics. 
We hope some day to make a similar study of $\pi_5^{\varepsilon}$.

\subsection{Acknowledgements}

The authors gratefully acknowledge and thank Ahmed Moussaoui for patiently answering our many questions about Lusztig's classification of unipotent representations and many related matters.
The authors would also like to thank all the \href{http://automorphic.ca}{Voganists}, especially Bin Xu, Sarah Dijols, Kam-Fai Tam, Jerrod Smith, Geoff Vooys and Nicole Kitt for participating in numerous talks as these ideas and calculations were worked out. We thank Thomas Bitoun for helpful conversations.
We thank George Lusztig for pointing out to us that Proposition~\ref{prop:KazhdanLusztig} follows from \cite{Lusztig:Cuspidal2}*{Corollary 10.7}.
Finally C.C. would like to thank James Arthur for many helpful conversations over the years and for drawing our attention to the triple of "superunipotent" representations of the exceptional groups $G_2$, $F_4$ and $E_8$.

\section{Main results on admissible representations of $p$-adic $G_2$}\label{sec:main}

Let $F$ be a $p$-adic field; let $q$ be the cardinality of the residue field of $F$ and let $\mathcal{O}_F$ be the ring of integers in $F$.

\subsection{Roots and coroots}\label{ssec:roots}

Let $G_2$ be the unique group of type $G_2$ over a $p$-adic group $F$.
Recall that $G_2$ is split over $F$ as discussed in Section~\ref{ssec:other}.
Fix a maximal split torus $T$ of $G_2$ and denote the corresponding set of roots by $\Phi(G_2)$.
Let $\alpha,\beta : T \rightarrow F^\times$ be a choice simple roots of $G_2$ with $\alpha$ short and $\beta$ long. 
Then the positive roots of $\Phi(G_2)$ are
\[
\alpha, \beta, \alpha+\beta, 2\alpha+\beta, 3\alpha+\beta, 3\alpha+2\beta,
\]
and the Cartan matrix is 
\[
\begin{pmatrix}
 \pair{\alpha,\alpha} &\pair{\alpha,\beta} \\ 
 \pair{\beta,\alpha} &\pair{\beta,\beta} 
 \end{pmatrix}
=
 \begin{pmatrix}
  2&-1\\ 
 -3&2
\end{pmatrix}.
\]
Denote $\Phi^\vee(G_2)$ the set of coroots of $G_2$ and by $\alpha^\vee,\beta^\vee : F^\times \to T$ the coroots of $\alpha$ and $\beta$.
Following \cite{Muic}, we fix an isomorphism $T \rightarrow F^\times \times F^\times$ by 
 \[ t \mapsto ((2\alpha+\beta)(t),(\alpha+\beta)(t)). \]
  We use the notation $m: F^\times \times F^\times \to T$ for the inverse of this isomorphism.
 Under this identification we have
 \[ \alpha^\vee(a)=m(a,a^{-1}),\qquad\text{\ and\ }\qquad \beta^\vee(a)=m(1,a) . \]

Let $\wh G_2$ be the dual group of $G_2$ over the complex field obtained via an identification of the roots $\Phi(\wh G_2) = \Phi^\vee(G_2)$ and $\Phi^\vee(\wh G_2) = \Phi(G_2)$.
Denote by $\da\in \Phi(\wh G_2)$ (resp. $\db\in \Phi(G_2)$) the image of $\alpha^\vee$ (resp. $\beta^\vee$) under this identification. 
Then $\wh G_2$ is a complex reductive group of type $G_2$, with simple roots $\da,\db$.
Notice that $\da$ is the long root of $\wh G_2$ and $\db$ is the short root of $\wh G_2$. 
We will denote by $\wh T$ the maximal torus of $\wh G_2$.
The Cartan matrix of $\wh G_2$ is given by
\[
\begin{pmatrix}
 \pair{\da,\da} &\pair{\da,\db} \\ 
 \pair{\db,\da} &\pair{\db,\db} 
 \end{pmatrix}
=
 \begin{pmatrix}
  2&-3\\ 
 -1&2
\end{pmatrix}.
\]
As above, we fix an isomorphism $\wh T\rightarrow \mathbb{C}^\times\times\mathbb{C}^\times $ by 
 \[ t\mapsto ((\da+2\db)(t),(\da+\db)(t)). \]
Note that we have 
\begin{equation}\label{eqn:rootactionontorus}
\da (\wh m(x,y))=x^{-1}y^2, \qquad \mathrm{\ and\ }\qquad  \db(\wh m(x,y))=xy^{-1}.
\end{equation}

 We write $\wh m : \mathbb{C}^\times\times\mathbb{C}^\times \to \wh T$ for the inverse of this isomorphism.
We will denote by $\da^\vee, \db^\vee :\mathbb{C}^\times \rightarrow \wh T$ the coroots of $\da$ and $\db$. 
We have again that
\[ \da^\vee(a)={\wh m}(1,a),\quad \db^\vee(a)={\wh m}(a,a^{-1}) . \] 

Note that under the identification $\Phi(G_2)= \Phi^\vee(\wh G_2)$, we have $\alpha = \da^\vee$ and $\beta = \db^\vee$. Moreover, we have
$(\da + \db)^\vee = 3\da^\vee + \db^\vee$, 
$(\da + 2\db)^\vee = 3\da^\vee + 2\db^\vee$, 
$(\da + 3\db)^\vee = \da^\vee + \db^\vee$ and 
$(2\da + 3\db)^\vee = 2\da^\vee + \db^\vee$. These relations can be checked using the explicit formula for $\hat \gamma^\vee(a),a\in \C^\times$, as given in \cite{Muic}*{p.467}.

\subsection{Chevalley basis}\label{ssec:Chevalley}

Denote by $\mathfrak{g}_2$ the Lie algebra of $G_2$ and $\hat {\mathfrak{g}}_2:=\mathfrak{g}_{2,\mathbb{C}}$ the Lie algebra of $\wh G_2$. For each root $\gamma\in \Phi(G_2)$, denote by $\mathfrak{g}_{\gamma}\subset \mathfrak{g}_2$ the root space of $\gamma$. We fix an element $X_\gamma\in \mathfrak{g}_{\gamma}$ as in \cite{BumpJoyner}. Note that in \cite{BumpJoyner}, the root $\alpha$ is denoted by $\alpha_1$ and $\beta$ is denoted by $\alpha_2$. For each root $ \gamma\in \Phi( G_2)$, denote by $U_\gamma\subset G_2$ the corresponding root space with Lie algebra $\mathfrak{g}_{\gamma}$. We fix an isomorphism $x_\gamma: F\rightarrow U_\gamma$ defined by $x_\gamma(t)=\exp(t X_\gamma)$. Let $U$ be the subgroup of $G_2=G_2(F)$ generated by $U_\gamma$ when $\gamma$ runs over all positive roots of $G_2$.
For each positive root $\gamma\in \Phi(G_2)$, denote $\iota_{\gamma}: \SL_2(F)\rightarrow G_2$ the embedding determined by 
$$\iota_\gamma\left(\bpm 1&t\\ 0 &1 \epm \right)=x_\gamma(t), 
\qquad
\iota_\gamma\left( \bpm1& 0 \\ t&1 \epm \right)=x_{-\gamma}(t).$$
For a positive root $\gamma$, denote $h_\gamma(t)=\iota_\gamma(\diag(t,t^{-1}))$, $t\in F^\times$ and $w_\gamma=\iota_\gamma\left(\left(\begin{smallmatrix} &1\\-1& \end{smallmatrix}\right) \right)$. Note that $h_\gamma(t)=\gamma^\vee(t)$ and $w_\gamma$ is a representative of the Weyl element associated with $\gamma$. Moreover, denote by $\GL_2(\gamma)$ the subgroup of $G_2$ generated by $\iota_{\gamma}(F)$ and $T$. Note that $\GL_2(\gamma)$ is isomorphic to $\GL_2(F)$. Indeed, one can extend $\iota_{\gamma}$ to an isomorphism $\GL_2(F)\cong \GL_2(\gamma)$.

For the dual group $\wh G_2$, we will use similar notation as above. For example, for a root $\hat \gamma\in \Phi(\wh G_2)$, we will fix element $X_{\hat\gamma}\in \mathfrak{g}_{\hat \gamma}$ as in \cite{BumpJoyner}, where our $\da$ (resp. $\db$) is $\alpha_2$ (resp. $\alpha_1$) in \cite{BumpJoyner}. We define an embedding  $x_{\hat\gamma}: \mathbb{C}\rightarrow \wh G_2$ by $x_{\hat\gamma}(t)=\exp(t X_{\hat\gamma})$. Similarly, let $\iota_{\hat \gamma}: \SL_2(\mathbb{C})\rightarrow \wh G_2$ be the embedding determined by $\hat \gamma$ and $h_{\hat{\gamma}}(t)=\iota_{\hat{\gamma}}(\diag(t,t^{-1})),t\in \C^\times$. Moreover, denote $H_{\hat{\gamma}}=[X_{\hat{\gamma}},X_{-\hat\gamma}]\in \hat{\mathfrak{g}}_2$ for a positive root $\hat\gamma$. Note that the matrix realization of $H_{\hat \gamma}$ is also given in \cite{BumpJoyner}. For a positive root $\hat \gamma\in \{\hat \alpha, \hat\beta\} $, denote by $\GL_2(\hat \gamma)$ the subgroup of $\wh G_2$ generated $\iota_{\hat \gamma}(\SL_2(\mathbb C))$ and $\wh T$. Then $\GL_2(\hat \gamma)$ is isomorphic to $\GL_2(\mathbb C)$, and $\iota_{\hat \gamma}$ extends to a group homomorphism $\GL_2(\C)\to \GL_2(\hat\gamma)\rightarrow \wh G_2$. We fix embeddings $ \iota_{\db}$ and $\iota_{\da+2\db}$ such that 
\begin{equation}\label{eqn:embeddingoftorus}
 \iota_{\db} \left(\bpm x& 0 \\ 0 &y \epm \right)=\wh m(x,y), \quad \iota_{\da+2\db}\left(\bpm x& 0 \\ 0 &y \epm \right)=\wh m(xy^{-1},x) .
\end{equation}

\subsection{Representations of $G_2(F)$}\label{ss:representation}
Let $F$ be a $p$-adic field. We recall some notations from \cite{Muic} on representations of $G_2=G_2(F)$. Write $\nu$ as the character of $F^\times$ defined by the absolute value of $F$, i.e., $\nu(t)=|t|_F$. Let $B=TU$ be the standard Borel subgroup of $G_2$ with maximal torus $T$ and maximal unipotent $U$ as fixed in previous sections. For a simple root $\gamma\in \{\alpha,\beta\}$, let $P_\gamma$ be the standard parabolic subgroup of $G_2$ with Levi subgroup $\GL_2(\gamma)$.

  For a pair of characters $(\chi_1,\chi_2)$ of $F^\times$, we view $\chi_1\otimes \chi_2$ as a character of $T$ under the isomorphism $m^{-1}:T\rightarrow F^\times\times F^\times$. For a pair of complex numbers of $(s_1,s_2)$, denote 
$$I(\nu^{s_1}\chi_1,\nu^{s_2}\chi_2):=\Ind_{B}^{G_2}((\nu^{s_1}\chi_1)\otimes (\nu^{s_2}\chi_2)).$$
When $\Theta_1,\Theta_2$ are unitary and $s_1>s_2>0$, the representation $I(s_1,s_2,\chi_1,\chi_2)$ has a unique quotient, which is denoted by $J(s_1,s_2,\chi_1,\chi_2)$.

For a simple root $\gamma\in \{\alpha,\beta\}$, a representation $\tau$ of $\GL_2(\gamma)\cong \GL_2(F)$ and a complex number $s\in \C$, denote 
$$I_\gamma(s,\pi)=\Ind_{P_\gamma}^{G_2}((\nu^s\circ \det)\otimes \pi).$$
If $\pi$ is tempered and $s>0$, denote by $J_\gamma(s,\pi)$ the unique Langlands quotient of $I_\gamma(s,\pi) $.

For a pair of characters $\chi_1,\chi_2$ of $F^\times$, denote by $\pi(\chi_1,\chi_2)$ the induced representation $\Ind_{B_{\GL_2}}^{\GL_2(F)}(\chi_1,\chi_2)$ of $\GL_2(F)$, where $B_{\GL_2}$ is the upper triangular unipotent subgroup of $\GL_2(F)$. Denote by $\delta(1)$ the Steinberg representation of $\GL_2(F)$.

\subsection{Subregular unipotent orbit for $G_2$}\label{ssec:sub}



Every unipotent conjugacy class $\mathcal{O}$ in $\dualgroup{G_2}$ determines a $\dualgroup{G_2}$-conjugacy class of algebraic group homomorphisms $\SL_2(\CC) \to \dualgroup{G_2}$ such that $\varphi\left(\begin{smallmatrix} 1 & 1 \\ 0 & 1\end{smallmatrix}\right)\in \mathcal{O}$.
Let $\varphi_\mathrm{sub} : \SL_2(\CC) \to \dualgroup{G_2}$ be one such algebraic group homomorphism determined by the subregular unipotent orbit of $\dualgroup{G_2}$.
Specifically, let  $\varphi_\mathrm{sub} : \SL_2(\CC) \to \dualgroup{G}$ be the homomorphism of algebraic groups corresponding to the  $\mathfrak{sl}(2)$-triple
\[e = X_{\da} - X_{\da+2\db},\quad f = X_{-\da} - X_{-\da-2\db},\quad h=  H_{\da} + H_{\da+2\db}. \] 
Then $\varphi_\mathrm{sub}$ is defined by 
\[
\begin{array}{l}
\varphi_\mathrm{sub}\begin{pmatrix} 1 & x \\ 0 & 1\end{pmatrix} = \exp(x e) 
= \exp(x X_{\da})\exp(-x X_{\da+2\db}), \\
\varphi_\mathrm{sub} \begin{pmatrix} 1 & 0  \\ y & 1\end{pmatrix} = \exp(y f) 
= \exp(y X_{-\da})\exp(-y X_{-\da-2\db}),
\end{array}
\]
and
\[
\begin{array}{l}
\varphi_\mathrm{sub} \begin{pmatrix} t & 0 \\ 0 & t^{-1} \end{pmatrix} = \exp(t h) = \exp(t (H_{\da} + H_{\da+2\db})).
\end{array}
\]
Now let $\phi_3 : W'_F \to \dualgroup{G_2}$ be the Langlands parameter defined by $\phi_3(w,x) = \varphi_\mathrm{sub}(x)$; this Langlands parameter is of Arthur type for  $\psi_3 : W'_F \times \SL_2(\CC) \to \dualgroup{G_2}$ defined by $\psi_3(w,x,y) = \varphi_\mathrm{sub}(x)$.
The infinitesimal parameter $\lambda_\mathrm{sub} :  W_F \to \dualgroup{G_2}$ for $\phi_3$ is 
\begin{equation}\label{eqn:lambdasub}
\lambda_\mathrm{sub}(w) \ceq \phi_3\left(w, \left(\begin{smallmatrix} \abs{w}^{1/2} & 0 \\ 0 & \abs{w}^{-1/2} \end{smallmatrix}\right)\right) = \varphi_\mathrm{sub}\left(\begin{smallmatrix} \abs{w}^{1/2} & 0 \\ 0 & \abs{w}^{-1/2} \end{smallmatrix}\right).
\end{equation}
This unramified infinitesimal parameter is thus completely determined by
\[ 
\lambda_\mathrm{sub}(\Frob) =  \exp(q^{1/2} (H_{\da} + H_{\da+2\db})) = (2\da^\vee+\db^\vee)\otimes q \in X^*(T)\otimes_\ZZ \CC^\times = \dualgroup{T} .
\]
This follows from $(\da+2\db)^\vee=2\db^\vee+3\da^\vee$ and $\frac{1}{2}(\da^\vee+(\da+2\db)^\vee)=2\da^\vee+\db^\vee$; see Section~\ref{ssec:roots}.
We observe that 
\[
\lambda_\mathrm{sub}(\Frob)= h_{\da}(q) h_{\db}(q^2).
\]
\subsection{Moduli space of Langlands parameters}\label{sec:Vsub}

In the remainder of this paper, we replace the L-group $\Lgroup{G}_2$ with the dual group $\dualgroup{G_2}$; this should cause no confusion since $G_2$ is split over $F$.

For any Langlands parameter $\phi : W'_F \to \dualgroup{G_2}$, we refer to  
\[
\begin{array}{rcl}
\lambda_\phi : W_F &\to& \dualgroup{G_2}\\
w &\mapsto& \phi\left(w,\left(\begin{smallmatrix} \abs{w}^{1/2} & 0 \\ 0 & \abs{w}^{-1/2} \end{smallmatrix}\right)\right)
\end{array}
\]
as the {\it infinitesimal parameter for $\phi$}.
As explained in \cite{CFMMX}*{Proposition 4.2.2}, the moduli space of Langlands parameter $\phi$ for which $\lambda_\phi=\lambda_\mathrm{sub}$ is  
\begin{equation}
V_\mathrm{sub} = \{ r_0 X_{\da} + r_1 X_{\da+\db} + r_2 X_{\da+2\db} + r_3 X_{\da+3\db} \tq r_0, r_1, r_2, r_3 \} \iso \mathbb{A}^4.
\end{equation}
Using \cite{CFMMX}*{Section 4.2}, it follows that
\[
V_\mathrm{sub} = \dualgroup{\mathfrak{g}}_2(q),
\]
for the weight filtration of $\dualgroup{\mathfrak{g}}_2$ determined by the action $\Ad(\lambda_\mathrm{sub}(\Frob))$.
The centralizer of $\lambda_\mathrm{sub}$ in $\dualgroup{G}$ is 
$H_\mathrm{sub} = \GL_2(\db)$, as defined in Section~\ref{ssec:Chevalley}. We will identify $V_{\mathrm{sub}}$ with $\BA^4$ (column vectors) by $ r_0 X_{\da} + r_1 X_{\da+\db} + r_2 X_{\da+2\db} + r_3 X_{\da+3\db}\mapsto (r_0,r_1,r_2,r_3)^t$.
The action $H_\mathrm{sub}\times V_\mathrm{sub}  \to V_\mathrm{sub}$ is given by
\[
h.r =
\frac{1}{ad-bc} 
\begin{pmatrix} 
 d^3 &-3cd^2 & -3c^2 d&-c^3\\ 
 -bd^2 &d(ad+2bc) &c(2ad+bc)& ac^2 \\ 
 -b^2d &b(2ad+bc) & a(ad+2bc) & a^2c \\ 
 -b^3 & 3ab^2 & 3a^2b & a^3 
\end{pmatrix}.\bpm r_0 \\ r_1\\ r_2\\ r_3 \epm
\]
where $h=\left(\begin{smallmatrix}  a&b\\ c&d \end{smallmatrix}\right)\in H_{\mathrm{sub}}$ and $r = (r_0,r_1, r_2, r_3)^t$. The above action can be computed using the matrix realization of $G_2$ given in \cite{BumpJoyner}, where the negative signs follows from the fixed choices of Chevalley basis of \cite{BumpJoyner}.  See \cite[(3.5)]{Ch} for a matrix realization of the above action in the form $h^{-1}vh, h\in H_{\mathrm{sub}},v\in V_{\mathrm{sub}}$ based on the commutator relation \cite{Ree}*{p.443, (3.10)}

Recall that the equivariant fundamental group of an $H_\mathrm{sub}$-orbit $C\subset V_\mathrm{sub}$ is defined by
\[
A_C \ceq \pi_0(Z_{H_\mathrm{sub}}(c)) = \pi_1(C,c)_{Z_{H_\mathrm{sub}}(c)},
\]
for the choice of a base point $c\in C$.
Consider the following 4 points in $V_{\mathrm{sub}}$: $$c_0=(0,0,0,0)^t, \qquad c_1=(1,0,0,0)^t, $$ 
$$c_2=(0,1,0,0)^t, \qquad c_3=(1,0,1,0)^t.$$

\begin{lemma}\label{lem:fg}
The action of $ H_{\mathrm{sub}}$ on $V_{\mathrm{sub}}$ has four orbits $C_i$, $i=0,1,2,3$ with representatives $c_i\in C_i$. Moreover, we have $A_{C_3}\cong S_3,$ and $ A_{C_2}=A_{C_1}=A_{C_0}=1,$ where $A_{C_i}$ is the component group of $C_i$ for $i=0,1,2,3$. Here $S_3$ is the symmetric group of $3$ letters.
\end{lemma}

\begin{proof}
The orbits can be computed directly as in \cite{Ch}*{p.197}. 
See \ref{ssec:cubicorbits} for a different description of the orbits.

Consider the element $c_3=(1,0,1,0)^t\in C_3$. We have that $A_3\cong Z_{H_{\mathrm{sub}}}(c_3)^0$, i.e., the component group of the centralizer of $c_3$ in $H_{\mathrm{sub}}$. Suppose that $h=\left(\begin{smallmatrix} a&b\\ c&d \end{smallmatrix}\right) \in H_{\mathrm{sub}}$ is in the centralizer of $c_3$. By the action of $H_{\mathrm{sub}}$ on $V_{\mathrm{sub}}$, we have
\[
h.x_3 = \frac{1}{ad-bc}\bpm d^3-3c^2d\\ -bd^2+c(2ad+bc)\\ -b^2d+a(ad+2bc)\\ -b^3+3a^2d \epm =\bpm 1\\ 0 \\ 1\\ 0\epm.
\]
Then we can solve that $h.c_3 = c_3$ is in the subgroup of $\GL_2(\mathbb C)$ generated by 
$$\bpm -\frac{1}{2} & -\frac{\sqrt{3}}{2}\\ \frac{\sqrt 3}{2} & -\frac{1}{2} \epm \qquad\mathrm{ and }\qquad \bpm 1 & 0 \\ 0 &-1 \epm.$$
This group is discrete and isomorphic to $S_3$. Thus we get $A_3\cong S_3$. We omit for now the details of the calculation of $A_0,A_1$ and $A_2$, these can be infered from Lemma \ref{lem:stabilizers}.
\end{proof}

The moduli space of Langlands parameters $\phi$ for which $\lambda_\phi$ is $\dualgroup{G_2}$-conjugate to $\lambda_\mathrm{sub}$ is
\begin{equation}\label{eqn:Xsub}
X_\mathrm{sub} = \dualgroup{G_2}\times_{H_\mathrm{sub}} V_\mathrm{sub}
\end{equation}
The equivalence of categories
\begin{equation}\label{eqn:VX}
\Perv_{\dualgroup{G}}(X_\mathrm{sub}) \iso \Perv_{H_\mathrm{sub}}(V_\mathrm{sub})
\end{equation}
matches $\dualgroup{G}$-orbits $\mathfrak{C}_i$ in $X_\mathrm{sub}$ with $H_\mathrm{sub}$-orbits $C_i$ in $V_\mathrm{sub}$.
Under this matching, Lemma~\ref{lem:fg} gives us the equivariant fundamental groups of the $\dualgroup{G}$-orbits $S_i$ in $X_\mathrm{sub}$:
\begin{equation}\label{eqn:AmfC}
\begin{array}{ll} 
A_{\mathfrak{C}_0} =1 &\qquad A_{\mathfrak{C}_1} = 1\\
A_{\mathfrak{C}_2} = 1 &\qquad A_{\mathfrak{C}_3} = S_3.
\end{array}
\end{equation}

\begin{proposition}\label{prop:langlandsparameters}
There are exactly four $\dualgroup{G_2}$-conjugacy classes of Langlands parameters for $G_2$ with infinitesimal parameter  $\dualgroup{G_2}$-conjugate to $\lambda_\mathrm{sub}$; these conjugacy classes are classified by the four $\dualgroup{G}_2$-orbits $\mathfrak{C}_i$ in $X_\mathrm{sub}$, $i=0,1,2,3$. 
Likewise, there are exactly four $H_\mathrm{sub}$-conjugacy classes of Langlands parameters for $G_2$ with infinitesimal parameter equal to $\lambda_\mathrm{sub}$; these conjugacy classes are classified by the four $H_\mathrm{sub}$-orbits $C_i$ in $V_\mathrm{sub}$, $i=0,1,2,3$. 
The four orbit representatives chosen above, $c_0$, $c_1$, $c_2$ and $c_3$ determine, respectively, the following four Langlands parameters for $G_2$:
\[
\begin{array}{rcl}
\phi_{0}(w,x) &=& \lambda_\mathrm{sub}(w) = {\wh m}(|w|,|w|) = (2\da^\vee+\db^\vee)(|w|) , \\
\phi_{1}(w,x) &=& \iota_{\da}(\diag(|w|^{1/2},|w|^{1/2}))\ \iota_{\da}(x), \\
\phi_{2}(w,x) &=& \iota_{\da+2\db}(\diag(|w|^{1/2},|w|^{1/2}))\ \iota_{\da+2\db}(x) , \\
\phi_{3}(w,x) &=&\varphi_{\mathrm{sub}}(x).
\end{array}
\]
 Moreover, the component group $A_{\phi_i}$ is equal to the equivariant fundamental group $A_{C_i}$, where $A_{\phi_i}=\pi_0(Z_{\dualgroup{G}_2}(\phi_i))$. 
\end{proposition}

Note that in the above Proposition, we have 
\[ \iota_{\da+2\db}(\diag(|w|^{1/2},|w|^{1/2}))=\wh m(1,|w|^{1/2})\] and 
\[\iota_{\da}(\diag(|w|^{1/2},|w|^{1/2}))=\wh m(|w|,|w|^{1/2}).\]
\begin{proof}
It is known that the orbits of the $H_{\mathrm{sub}}$ action on $V_{\mathrm{sub}}$ classifies the conjugacy classes of Langlands parameters with infinitetesimal parameter $\lambda_{\mathrm{sub}}$.
The recipe is given as follows. For $c_i\in C_i $, the corresponding Langlands parameter $\phi_i:W_F\times \SL_2(\C)\to \wh G_2(\C)$ is determined by the condition 
$$\phi_i\left(1,\begin{pmatrix}1 & r\\ 0 &1 \end{pmatrix}\right)=\exp(rc_i), r\in \C,$$
and $$\phi_i(w,d_w)=\lambda(w),w\in W_F.$$
Note that the first condition will fix the value of $\phi_i$ on $1\times \SL_2(\C)$, which together with the second condition determines the value of $\phi_i$ on $W_F\times 1$. We can then check that $\phi_i$ determined by $c_i$ has the form as stated in the Proposition. The ``moreover'' part is a standard fact; see for example \cite{CFMMX}*{Lemma~4.6.1}.
\end{proof}

\subsection{Dual moduli space of Langlands parameters}\label{sec:Vsubdual}

Consider the ``dual'' infinitesimal parameter
\begin{equation}\label{eqn:lambdasubdual}
\lambda_\mathrm{sub}^*(w) \ceq \varphi_\mathrm{sub}\left(\begin{smallmatrix} \abs{w}^{-1/2} & 0 \\ 0 & \abs{w}^{1/2} \end{smallmatrix}\right).
\end{equation}
The moduli space of Langlands parameters $\phi : W'_F \to \dualgroup{G}$ for which $\lambda_\phi = \lambda_\mathrm{sub}^*$ is
\begin{equation}\label{eqn:Vsub}
V^*_\mathrm{sub} = \{ s_0 X_{-\da} + s_1 X_{-\da-\db} + s_2 X_{-\da-2\db} + s_3 X_{-\da-3\db} \tq s_0, s_1, s_2, s_3 \}.
\end{equation}
Then, arguing as in Proposition~\ref{prop:langlandsparameters},
\[
V^*_\mathrm{sub} 
=
\dualgroup{\mathfrak{g}}_2(q^{-1}),
\]
for the weight filtration of $\dualgroup{\mathfrak{g}}_2$ determined by the action $\Ad(\lambda_\mathrm{sub}(\Frob))$.
The centralizer of $\lambda_\mathrm{sub}^*$ in $\dualgroup{G}$ is again
$H_\mathrm{sub} = \GL_2(\da)$. We will identify $V_{\mathrm{sub}}^*$ with $\BA^4$ (column vectors) by $ s_0 X_{-\da} + s_1 X_{-\da-\db} + s_2 X_{-\da-2\db} + s_3 X_{-\da-3\db}\mapsto (s_0,s_1,s_2,s_3)^t$.
The action $H_\mathrm{sub}\times V_\mathrm{sub}^*  \to V_\mathrm{sub}^*$ is given by
\[
h.s =
{(ad-bc)} 
\prescript{t}{}{\begin{pmatrix} 
 d^3 &-3cd^2 & -3c^2 d&-c^3\\ 
 -bd^2 &d(ad+2bc) &c(2ad+bc)& ac^2 \\ 
 -b^2d &b(2ad+bc) & a(ad+2bc) & a^2c \\ 
 -b^3 & 3ab^2 & 3a^2b & a^3 
\end{pmatrix}^{-1}}.
\bpm s_0 \\ s_1\\ s_2\\ s_3 \epm
\]
where $h=\left(\begin{smallmatrix}  a&b\\ c&d \end{smallmatrix}\right)\in H_{\mathrm{sub}}$ and $s = (s_0,s_1, s_2, s_3)^t$.

The Killing form for $\dualgroup{\mathfrak{g}}_2$ establishes the duality between $\dualgroup{\mathfrak{g}}_2(q) = V_\mathrm{sub}$ and $\dualgroup{\mathfrak{g}}_2(q^{-1}) = V_\mathrm{sub}^*$:
\[
\KPair{}{} : \dualgroup{\mathfrak{g}}_2(q) \times \dualgroup{\mathfrak{g}}_2(q^{-1}) \to \C.
\]
In the coordinates chosen above, this is given by
\[
\KPair{r}{s} = r_0 s_0 +3 r_1 s_1 + 3r_2 s_2 + r_3 s_3.
\]
We will also need the Lie bracket
\begin{equation}\label{eqn:LieV}
[\ ,\ ] : \dualgroup{\mathfrak{g}}_2(q) \times \dualgroup{\mathfrak{g}}_2(q^{-1}) \to  \Lie H_\mathrm{sub} \ceq \mathfrak{h}_\mathrm{sub}.
\end{equation}

\begin{lemma}\label{lemma:momentmap}
 For $r=(r_0,r_1,r_2,r_3)^t=\sum r_i X_{\da+i\db}\in V$ and $s=\sum s_i X_{-(\da+i\db)}\in V^*$,
 the Lie bracket \eqref{eqn:LieV} is given by
\begin{align*}
[r,s] &=(r_0s_0+2r_1s_1+r_2s_2) H_{\da}+(r_1s_1+2r_2s_2+r_3s_3)(H_{\da}+H_{\db})\\
&~~+(-r_1s_0+2r_2s_1+r_3s_2)X_{\db}+(-r_0s_1+2r_1s_2+r_2s_3)X_{-\db}.
\end{align*}
\end{lemma}

\begin{proof}
Write $r=(r_0,r_1,r_2,r_3)^t=\sum r_i X_{\da+i\db}\in V$ and $s=\sum s_i X_{-(\da+i\db)}\in V^*$, we compute $[r,s]$. 
Note that 
$$H_{\da+\db}=3H_{\da}+H_{\db}, H_{\da+2\db}=2H_{\da}+3H_{\db}, H_{\da+3\db}=H_\da+H_{\db}.$$
The above relations can be checked from the matrix realizations of $H_{\hat \gamma}$ given in \cite{BumpJoyner} and they reflect the facts that $(\da+\db)^\vee=2\da^\vee+\db^\vee, (\da+2\db)^\vee=3\da^\vee+2\db^\vee,$ and $(\da+3\db)^\vee=\da^\vee+\db^\vee.$
Using the Chevalley basis in \cite{BumpJoyner}, we have 
\[
\begin{array}{rl rl} 
[ X_\da,X_{-(\da+\db)} ] &=-X_{-\db}, & [ X_{\da+\db},X_{-\da}] &=-X_{\db}, \\
\ [ X_{\da+\db},X_{-(\da+2\db)} ] &=2X_{-\db}, & [ X_{\da+2\db},X_{-(\da+\db)} ] &=2X_{\db}, \\
\ [ X_{\da+2\db},X_{-(\da+3\db)} ] &=X_{-\db}, & [ X_{\da+3\db}, X_{-(\da+2\db)} ] &=X_{\db}.
\end{array}
\]
Accordingly,
\begin{align*}
[r,s]&=r_0s_0H_\da +r_1s_1 H_{\da+\db}+r_2s_2H_{\da+2\db}+r_3s_3H_{\da+3\db}\\
 &~~+(-r_1s_0+2r_2s_1+r_3s_2)X_{\db}+(-r_0s_1+2r_1s_2+r_2s_3)X_{-\db},
\end{align*}
which gives the formula of $[r,s]$.
\end{proof}

\subsection{Langlands classification}\label{ssec:SM}

Recall that the Langlands parameter $\phi_0$ is given by $\phi_0(w,x)=(2\da^\vee+\db^\vee)(|w|)$ and it has image in $\wh T$. Under the local Langlands correspondence for tori, $\phi_0$ corresponds to the character $\nu\otimes 1$ of $T$, see Section \ref{ss:representation} for the notations.  Let $\pi_0$ be the irreducible smooth representation of $G_2(F)$ corresponding to $\phi_0$ under the Langlands correspondence for unramified principle series representations. 
 Then $\pi_0$ is the spherical component of the unramified principal series $I(\nu\otimes 1)$, where a representation is called spherical if there is a nonzero vector fixed by the maximal open compact subgroup $G_2(\mathcal{O}_F)$. Note that we have $w_{\alpha}(\nu\otimes 1)=1\otimes \nu$. By \cite[Proposition 4.3]{Muic}, the representation $I(1\otimes\nu)$ contains exactly two irreducible non-equivalent irreducible subrepresentations $\pi(1)$ and $\pi(1)'$, which are square-integral. Moreover, in the category of Grothendieck groups of representations of $G_2(F)$, we have
\begin{align}\label{eqn:muic1}
\begin{split}
I(1\otimes \nu)&=I(\nu\otimes 1)\\
&= I_{\beta}(1,\pi(1,1))\\
&=I_{\alpha}(1/2,\delta(1))+I_\alpha(1/2,1_{\GL_2})\\
&=I_{\beta}(1/2,\delta(1))+I_\beta(1/2,1_{\GL_2}),
\end{split}
\end{align}
and
\begin{align}\label{eqn:muic}
\begin{split}
I_\alpha(1/2,\delta(1))&=\pi(1)'+J_\alpha(1/2,\delta(1))+J_\beta(1/2,\delta(1)),\\
I_\beta(1/2,\delta(1))&=\pi(1)+\pi(1)'+J_\beta(1/2,\delta(1)),\\
I_\alpha(1/2,1_{\GL_2})&=\pi(1)+J_\beta(1,\pi(1,1))+J_{\beta}(1/2,\delta(1)),\\
I_\beta(1/2,1_{\GL_2})&=J_{\beta}(1,\pi(1,1))+J_{\beta}(1/2,\delta(1))+J_\alpha(1/2,\delta(1)).
\end{split}
\end{align}

\begin{lemma}\label{lem:generic}
With reference to \cite{Muic}, the representation $\pi(1)'$ is generic and the representation $J_\beta(1,\pi(1,1))$ is spherical. 
\end{lemma}
\begin{proof}
It is well-known that for an irreducible representation $\tau$ of a Levi subgroup $M$ of a $p$-adic group $G$, the induced representation $\Ind_{P}^G(\tau)$ is generic if and only if $\tau$ is generic; moreover, if $\tau$ is generic, then $\Ind_{P}^G(\tau)$ contains a unique generic component. Here $P$ is a parabolic subgroup of $G$ with Levi $M$. This fact together with the above decomposition \eqref{eqn:muic} of various induced representation implies that $\pi(1)'$ is generic. In fact, since $1_{\GL_2}$ is not generic, we know that $I_{\gamma}(1/2,1_{\GL_2})$ is not generic for $\gamma=\alpha,\beta$. Thus neither $J_\alpha(1/2,\delta(1))$ nor $J_{\beta}(1/2,\delta(1))$ is generic from the last two rows of \eqref{eqn:muic}. Since $\delta(1)$ is an irreducible generic representation of $\GL_2(F)$, we see that $I_\alpha(1/2,\delta(1))$ has a unique generic component, which must be $\pi(1)'$ from the decomposition of $I_\alpha(1/2,\delta(1))$. This proves the first part. 

To see that $J_\beta(1,\pi(1,1))$ is spherical, we only have to notice that neither $I_{\alpha}(1/2,\delta(1))$ nor $I_{\beta}(1/2,\delta(1))$ is spherical, and thus any component of $ I_{\gamma}(1/2,\delta)$ is non-spherical. 
Thus, $J_{\beta}(1,\pi(1,1))$ is the unique spherical component of $I(1\otimes \nu)$. 
\end{proof}


Note that $\pi_0$ has infinitesimal parameter $\lambda_{\mathrm{sub}}$ and thus every irreducible component of $I(1\otimes \nu)$ has infinitesimal parameter $\lambda_{\mathrm{sub}}$.
Using Lemma~\ref{lem:generic} we now have
\[
\begin{array}{rcl}
\pi_0 &=& J_{\beta}(1,\pi(1,1)).
\end{array}
\]
For simplicity of later use, we introduce the following notations for the other four irreducible components of the induced representation $I(1\otimes \nu)$:
\[
\begin{array}{rcl}
\pi_1 &\ceq& J_\alpha(1/2,\delta(1)), \\
\pi_2 &\ceq& J_{\beta}(1/2,\delta(1)), \\
\pi_3 &\ceq& \pi(1)', \\
\pi_3^{\varrho} &\ceq& \pi(1).
\end{array}
\]
Let $G_2[1]$ be the unipotent supercuspidal representation of $G_2(\mathbb{F}_q)$ appearing in \cite[p.460]{Carter}. To the list above, we add 
\[
\begin{array}{rcl}
\pi_3^{\varepsilon} &\ceq& \mathrm{cInd}_{G_2(\mathcal{O}_F)}^{G_2(F)} G_2[1],
\end{array}
\]
as defined in the Introduction. 
From \cite{Muic}, we know that  all of the 6 irreducible representations appearing in this section are unitary. Moreover, they are also unipotent in the sense of Lusztig \cite{Lusztig:Classification1}.
The notation above in this section will be justified in Section \ref{ssec:LLC}.


\subsection{Langlands correspondence}\label{ssec:LLC}

In \cite{Kazhdan-Lusztig} Kazhdan-Lusztig proved a local Langlands correspondence for representations with Iwahori-fixed vectors. 
This correspondence was extended to all unipotent representations of unramified adjoint simple algebraic groups over $F$ in \cite{Lusztig:Classification1}. 
More precisely, for such groups $G$, Lusztig proved that the set of isomorphism classes of unipotent representations is in one-to-one correspondence with a pair $(\phi,r)$, where $\phi:W_F\times \SL_2(\C)\to \widehat{G}$ is an unramified Langlands parameter, and $r$ is a representation of the component group $A_{\phi}$. 
Such pairs $(\phi,r)$ are called enhanced Langlands parameters. 
This classification of unipotent representations was extended to more general groups in \cite{Solleveld} recently. 

In our special case when $G=G_2$, all enhanced Langlands parameter with infinitesimal parameter $\lambda_{\mathrm{sub}}$ are given as follows
\[
(\phi_0,\1), \;(\phi_1,\1),\; (\phi_2,\1),\; (\phi_3,\1),\; (\phi_3,\varrho),\; \mathrm{ and }\;  (\phi_3,\varepsilon),
\]
where in each pair $(\phi_i,\1)$, $\1$ denotes the trivial representation of the corresponding component group, $\varrho$ denotes the unique irreducible two-dimensional representation of $S_3\cong A_{\phi_3}=A_{C_3}$ and $\varepsilon$ denotes the non-trivial one-dimensional representation of $S_3$.

 The Local Langlands correspondence of \cite{Lusztig:Classification1} for irreducible unipotent representations with infinitesimal parameter $\lambda_{\mathrm{sub}}$ is explicitly given in the following theorem.

\begin{theorem}[\cite{Kazhdan-Lusztig, Lusztig:Classification1}]\label{thm:LLC} The local Langlands correspondence for representations with subregular infinitesimal parameter is: 
\[ 
\begin{array}{rcl}
\mathrm{enhanced} && \mathrm{irreducible} \\
\mathrm{L-parameter} && \mathrm{admissible\ rep'n} \\
\hline
(\phi_{0},1) &\mapsto& \pi_{0}, \\
(\phi_{1},1) &\mapsto& \pi_1 , \\
(\phi_{2},1) &\mapsto& \pi_2 , \\
(\phi_3,1) &\mapsto& \pi_3 , \\
(\phi_3,\varrho) &\mapsto& \pi_3^{\varrho} , \\
\hline
(\phi_3,\varepsilon) &\mapsto& \pi_3^{\varepsilon} .
\end{array}
\]
The L-packets containing these irreducible admissible representations are
\[
\begin{array}{r c l }
\Pi_{\phi_0}(G_2(F)) &=& \{ \pi_0\} , \\
\Pi_{\phi_1}(G_2(F)) &=& \{ \pi_1\}, \\
\Pi_{\phi_2}(G_2(F)) &=& \{ \pi_2\},  \\
\Pi_{\phi_3}(G_2(F))  &=& \{ \pi_3, \pi_3^\varrho, \pi_3^\varepsilon \}.  
\end{array}
\]
\end{theorem}
\begin{proof}
Temporarily, denote by $\pi(\phi_i,r)$ the representation corresponding to $(\phi_i,r)$ for a Langlands parameter $\phi_i$ and a representation $r$ of the component group. We first notice that  $(\phi_3,\varepsilon)$ is the unique cuspidal data in the sense of \cite{Lusztig:Intersectioncohomology} with infinitesimal parameter $\lambda_{\mathrm{sub}}$, see \cite[p.270]{Lusztig:Intersectioncohomology} and thus $\pi(\phi_3,\varepsilon)$ is a unipotent supercuspidal representation, which must be of the form
$$\mathrm{cInd}_{G_2(\mathcal{O}_F)}^{G_2(F)}(\sigma) $$
for an irreducible cuspidal unipotent representation $\sigma$ of $G_2(\mathbb{F}_q)$. 
There are exactly four irreducible cuspidal unipotent representations of $G_2(\mathbb{F}_q)$ as given in \cite[p.460]{Carter}. 
By \cite{FOS}, we know that the supercuspidal unipotent representation $\pi(\phi_3,\varepsilon)$ satisfies the formal degree conjecture of \cite{HII}, i.e., 
\begin{equation}{\label{eqn:formaldegree}}d(\pi(\phi_3,\varepsilon))=\frac{\dim \varepsilon}{|S_3|}\cdot |\gamma(0,\phi_3,\Ad,\psi^0)|,\end{equation}
where $d(\pi(\phi_3,\varepsilon))$ is the formal degree of $\pi(\phi_3,\varepsilon)$, normalized as in \cite{HII}, $\psi^0$ is a nontrivial unramified additive character of $F$ and
 $$\gamma(s,\phi_3,\Ad,\psi^0)=\epsilon(s,\phi_3,\Ad,\psi^0)\frac{L(1-s,\phi_3,\Ad)}{L(s,\phi_3,\Ad)}$$
is the adjoint $\gamma$-factor of $\phi_3$. On one hand, we have 
$$d(\pi)=\frac{\dim \sigma}{\mu(G_2(\mathcal{O}_F))},$$
where $\mu(G_2(\mathcal{O}_F))$ is the Haar measure of $G_2(\mathcal{O}_F)$, which is 
$$\mu(G_2(\mathcal{O}_F))=q^{-\dim(G_2)}\cdot |G_2(\mathbb{F}_q )|=q^{-14}\cdot q^6(q^6-1)(q^2-1)=q^{-8}(q^6-1)(q^2-1),$$ 
as normalized in \cite{HII}. On the other hand, we can check that 
$$L(s,\phi_3,\Ad)=\frac{1}{(1-q^{-s-2})(1-q^{-s-1})^3},$$
and $\epsilon(s,\phi_3,\Ad)=q^{10(1/2-s)}.$ Thus we can check that 
$$\gamma(0,\phi_3,\Ad)=\frac{q^9}{(q+1)^2(q^2+q+1)}.$$
Then the formal degree conjecture \eqref{eqn:formaldegree} implies that 
$$\dim(\sigma)=\mu(G_2(\mathcal{O}_F))\cdot \frac{\dim \varepsilon}{|S_{\phi_3}|} |\gamma(0,\phi_3,\Ad)|=\frac{q(q^6-1)(q^2-1)}{6(q+1)^2(q^2+q+1)}.$$
If we compare the degree of $\sigma$ with the table given in \cite[p.460]{Carter}, we get that $\sigma=G_2[1]$ and thus $\pi(\phi_3,\varepsilon)=\pi_3^{\varepsilon}$.

By the unramified local Langlands correspondence, we have $\pi_0= \pi(\phi_0,1)$. 
Since all of the $\pi(\phi_i,r)$ have the same infinitesimal parameter as $\pi_0$, we have
\begin{equation}\label{eqn:equalityoftwosets}\{\pi(\phi_1,1),\pi(\phi_2,1),\pi(\phi_3,1),\pi(\phi_3,\varrho)\}=\{\pi_1,\pi_2,\pi_3,\pi_3^{\varrho}\}.\end{equation}

 Reeder \cite{Reedergeneric} showed that the generic representation $\pi_3=\pi(1)'$ (by Lemma \ref{lem:generic}) must correspond to $(\phi_3,1)$, i.e., $\pi_3=\pi(\phi_3,1)$. Since $\phi_3|_{W_F}$ is bounded, we know that $\pi(\phi_3,\varrho)$ is tempered, which implies that $\pi(\phi_3,\varrho)=\pi_3^{\varrho}$. 

Next, notice that $\phi_1(w,x)=\iota_{\da}(\diag(|w|,|w|^{1/2}))\cdot \iota_{\da}(x)$ has image in $\widehat{\GL_2(\alpha)}=\GL_2(\da)$. If we define the Langlands parameter 
\[ \phi_{1,\alpha}:W_F\times \SL_2(\C)\to \widehat{\GL_2(\alpha)}=\GL_2(\da)\] by 
\[ \phi_{1,\alpha}(w,x)=\iota_{\da}(\diag(|w|^{1/2},|w|^{1/2}))\ \iota_{\da}(x),\] 
then $\phi_1$ can be decomposed as the composition of the embedding  $\widehat{\GL_2(\alpha)}\hookrightarrow \widehat{G_2} $ with the map $\phi_{1,\alpha}$. Thus by \cite[Theorem 6.2]{Kazhdan-Lusztig}, we have
$$M(\phi_1)=\Ind_{P_{\alpha}}^{G_2}(\pi^{\alpha}(\phi_{1,\alpha})),$$
where $M(\phi_1)$ is the standard module of $\pi(\phi_1,1)$ and $\pi^{\alpha}(\phi_{1,\alpha}) $ is the representation of $\GL_2(\alpha)$ corresponding to $\phi_{1,\alpha}$ under the local Langlands correspondence for $\GL_2(F)$. We further decompose $\phi_{1,\alpha}$ as $\hat \chi\otimes \Sym^1$, where $\hat \chi$ is a character on $W_F$ defined by $\hat \chi(w)=\iota_{\da}(\diag(|w|^{1/2},|w|^{1/2}))$ and $\Sym^1(x)=\iota_{\da}(x)$. Note that $\hat \chi$ is the dual of the character $|\det|^{1/2}: \GL_2(\alpha)\to \C^\times$. On the other hand, we have $\pi^{\alpha}(\Sym^1)=\delta(1)$. Thus 
$$\pi^{\alpha}(\phi_{1,\alpha})=\pi^{\alpha}(\hat \chi\otimes \Sym^1)=|\det|^{1/2}\otimes\delta(1). $$
Thus we have 
$$ M(\phi_1)=\Ind_{P_{\alpha}}^{G_2(F)}(|\det|^{1/2}\otimes\delta(1))=I_{\alpha}(1/2,\delta(1)). $$
Thus we get $\pi(\phi_1,1)=J_{\alpha}(1/2,\delta(1))=\pi_1$. 

Finally, from \eqref{eqn:equalityoftwosets}, we must have $\pi(\phi_2,\epsilon)=\pi_2$. The last fact can also be checked directly.
\end{proof}

\begin{remark}
The above theorem is well-known for experts. In fact, the local Langlands correspondence for representations with infinitesimal parameters $\lambda_{\mathrm{sub}}$ was studied extensively in the literature. For example, in \cite{Huang-Magaard-Savin} and \cite{GGJ}, the authors could obtain correspondences for certain representations with infinitesimal parameter $\lambda_{\mathrm{sub}}$ by studying the restrictions of minimal representations of an adjoint group of type $D_4$ to $G_2(F)$. See Table~\ref{The admissible representations that appear in this paper.} for an incomplete list of the irreducible representations with infinitesimal parameter $\lambda_\mathrm{sub}$ which have appeared in the literature.
\begin{table}[htp]
\caption{An incomplete list of the irreducible representations in $\Rep(G_2(F))_\mathrm{sub}$ as they appear in related literature}
\begin{center}
\begin{tabular}{|c|c|c|c|c|c|c|c| }
\hline
   & \cite{Muic} &  \cite{GG} \cite{GGJ} & \cite{Reeder} & \cite{Lusztig:Classification1} & \cite{Huang-Magaard-Savin} \\
 \hline
 $\pi_0$ & $J_\beta(1,\pi(1,1))$ & $\pi_{1_v}$ &  &   & $\pi_{1_v}$ \\
$\pi_1$ & $J_\alpha(1/2,\delta(1))$ & $\pi_{r_v}$ &  & & $\pi_{r_v}$\\ 
$\pi_2$ & $J_{\beta}(1/2,\delta(1))$ &  &  &   & \\
$\pi_3$ & $\pi(1)'$ & $\pi'_{1_v}$ & $\mathcal{V}_1 = \mathcal{U}(\tau_2)$ &  & \\
$\pi_3^\varrho$ & $\pi(1)$ & $\pi'_{r_v}$ & $\mathcal{V}_\varrho = \mathcal{U}(\tau_2)'$ &   & \\ 
$\pi_3^\varepsilon$ &   & $\pi'_{\varepsilon_v} = \pi_{\varepsilon_v}$  & $\mathcal{V}_\varepsilon =\mathcal{V}[1] $  & {7.33}  & $\pi_{\varepsilon_v}$ \\
\hline
\end{tabular}
\end{center}
\label{The admissible representations that appear in this paper.}
\end{table}%
\end{remark}

\subsection{Langlands-Vogan correspondence}\label{ssec:LVC}

Consider the Bernstein decomposition of the category $\Rep(G_2(F))$ of smooth representations of $G_2(F)$. The representations that we study appear in two blocks under this decomposition.

First, consider inertial class $[T(F),1]$, where $T$ is a maximal split torus of $G_2$; then the simple summand category
\[
\Rep(G_2(F))_{[T(F),1]}
\]
is the category is unramified principal series representations of $G_2(F)$. 
The representations $\pi_0$, $\pi_1$, $\pi_2$, $\pi_3$ and $\pi_3^\varrho$ all appear in this category.
Since these representations all share the same cuspidal support, we further restrict to the subcategory 
\[
\Rep(G_2(F))_{(T(F),\nu\otimes 1)}
\] 
of representations with cuspidal support $\nu\otimes 1$; see, for example, \cite{Renard}*{VI.7.2}.
This category has exactly the five simple objects $\pi_0$, $\pi_1$, $\pi_2$, $\pi_3$ and $\pi_3^\varrho$.

Second, consider the inertial class $[G_2(F),\pi_3^\varepsilon]$ and the summand category
\[
\Rep(G_2(F))_{[G_2(F),\pi_3^\varepsilon]},
\]
also appearing in the Bernstein decomposition.
Because $G_2(F)$ admits no unramified characters this coincides with the cuspidal support category
\[
\Rep(G_2(F))_{(G_2(F),\pi_3^\varepsilon)}.
\]
Up to equivalence this category contains exactly one irreducible representation, $\pi_3^\varepsilon$.

The category of representations that this paper treats is
\[
\Rep(G_2(F))_{\mathrm{sub}} 
\ceq 
\Rep(G_2(F))_{(T(F),\nu\otimes 1)}
\oplus
\Rep(G_2(F))_{(G_2(F),\pi_3^\varepsilon)}.
\]

Likewise, as in \cite{CFMMX}*{Section~10.2.4}, the category $\Perv_{\dualgroup{G_2}}(X_\mathrm{sub})$ decomposes into two summand categories
\[
\Perv_{\dualgroup{G_2}}(X_\mathrm{sub})
= 
\Perv_{\dualgroup{G_2}}(X_\mathrm{sub})_{(\dualgroup{T},\1)} \oplus \Perv_{\dualgroup{G_2}}(X_\mathrm{sub})_{(\dualgroup{G_2},\mathcal{E})},
\]
where, using Lemma~\ref{lem:fg}, the equivalence \eqref{eqn:VX} and \cite{BBD}*{Th\'eor\`eme 3.4.1}, the former category has exactly five simple objects, while the later has one, $\IC(\mathcal{E}_{\mathcal{C}_3})$. 
To enumerate these simple objects, we use the notation
\begin{equation}
\IC(\mathcal{L}_C) = j_{!*}\mathcal{L}_C[\dim C]
\end{equation}
where $\mathcal{L}_{C}$ is an equivariant local system on $C$ and $C\subseteq V$ is an $H$-orbit and $j: C\hookrightarrow V$ is the inclusion. 
Then the simple objects in $\Perv_{\dualgroup{G_2}}(X_\mathrm{sub})_{(\dualgroup{T},\1)}$ are $\IC(\1_{\mathfrak{C}_0})$, $\IC(\1_{\mathfrak{C}_1})$, $\IC(\1_{\mathfrak{C}_2})$, $\IC(\1_{\mathfrak{C}_3})$ and $\IC(\mathcal{R}_{\mathfrak{C}_3})$, where $\1_{\mathfrak{C}_i}$ refers to the constant sheaf on $\mathfrak{C}_i$ for $i=0,1, 2, 3$ while $\mathcal{R}_{\mathfrak{C}_3}$ refers to the local system on $\mathfrak{C}_3$ for the irreducible $2$-dimensional representation of $A_{\mathfrak{C}_3} = S_3$ \eqref{eqn:AmfC}.
On the other hand, $\Perv_{\dualgroup{G_2}}(X_\mathrm{sub})_{(\dualgroup{G_2},\mathcal{E})}$ has only one simple object, $\IC(\mathcal{E}_{\mathfrak{C}_3})$, where $\mathcal{E}_{\mathfrak{C}_3}$ refers to the local system on $\mathfrak{C}_3$ for the sign representation of $A_{\mathfrak{C}_3} = S_3$. 

In the proposition below, we use equivalence \eqref{eqn:VX} and write $\1_{C_i}$ for the constant sheaf on $C_i$ for $i=0,1, 2, 3$ while $\mathcal{R}_{C_3}$ refers to the local system on $C_3$ for the irreducible $2$-dimensional representation of $A_{C_3} = S_3$ \eqref{eqn:AmfC} and $\mathcal{E}_{C_3}$ refers to the local system on $C_3$ for the sign representation of $A_{C_3} = S_3$.
See Section~\ref{ssec:cubicorbits}, especially Proposition~\ref{prop:IC} 
for a detailed description of the simple objects in $\Perv_{H_\mathrm{sub}}(V_\mathrm{sub})$.

\begin{proposition}[Langlands-Vogan correspondence for $\Rep(G_2(F))_{\mathrm{sub}}$]\label{prop:LVC}
The local Langlands correspondence, Theorem~\ref{thm:LLC}, together with the equivalence \eqref{eqn:VX}, establishes the following bijection between the simple objects in these categories:
\[
\begin{array}{ccccc}
\Rep(G_2(F))  && \Perv_{\dualgroup{G_2}}(X_\mathrm{sub}) && \Perv_{H_\mathrm{sub}}(V_\mathrm{sub}) \\
\hline
\pi_0 &\leftrightarrow& \IC(\1_{\mathfrak{C}_0}) &\leftrightarrow& \IC(\1_{C_0}) \\
\pi_1 &\leftrightarrow& \IC(\1_{\mathfrak{C}_1}) &\leftrightarrow& \IC(\1_{C_1}) \\
\pi_2 &\leftrightarrow& \IC(\1_{\mathfrak{C}_2}) &\leftrightarrow& \IC(\1_{C_2}) \\
\pi_3 &\leftrightarrow& \IC(\1_{\mathfrak{C}_3}) &\leftrightarrow& \IC(\1_{C_3}) \\
\pi_3^\varrho &\leftrightarrow& \IC(\mathcal{R}_{\mathfrak{C}_3}) &\leftrightarrow& \IC(\mathcal{R}_{C_3}) \\
\hline
\pi_3^\varepsilon &\leftrightarrow& \IC(\mathcal{E}_{\mathfrak{C}_3}) &\leftrightarrow& \IC(\mathcal{E}_{C_3}) 
\end{array}
\]
\end{proposition}

\begin{remark}\label{remark:CLLC}

It is natural to ask if the categories $\Rep(G_2(F))_\mathrm{sub}$ and $\Perv_{\dualgroup{G_2}}(X_\mathrm{sub})$ are equivalent. 
If they are, this would be a categorical form of the local Langlands-Vogan correspondence.
We are actively exploring this question.
Here is one possible strategy to check.
First, observe that $\Rep(\G_2(F))_{[T(F),1]}$ is also the category of representations with Iwahori-fixed vectors.
This category is equivalent to the category of $H$-modules, where $H$ is the Iwahori-Hecke algebra for $G_2(F)$ (with equal parameters).
Under the equivalence
\[
\Rep(G_2(F))_{[T(F),1]}
\iso 
H\mathrm{-mod} 
\]
the subcategory $\Rep(G_2(F))_{(T(F), \nu\otimes 1)}$ is equivalent to the category $H\mathrm{-mod}_I$ of $H$-modules that are annihilated by some power of an ideal $I\triangleleft Z(H)$ determined by $\nu\otimes 1$ according to the recipe of \cite{Lusztig:Cuspidal2}*{}:
\[
\Rep(G_2(F))_{(T(F),\nu\otimes 1)}
\iso 
H\mathrm{-mod}_I.
\]
On the other hand, it is not difficult to find a progenerator 
 for $\Perv_{\dualgroup{G_2}}(X_\mathrm{sub})_{(\dualgroup{T},\1)}$, so this category should also admit a module-category description.
Comparing $\Rep(G_2(F))_{(G_2(F),\pi_3^\varepsilon)}$ and 
$\Perv_{\dualgroup{G_2}}(X_\mathrm{sub})_{(\dualgroup{G_2},\mathcal{E})}$ should be easier, since these categories admit natural progenerators, namely $\pi_3^\varepsilon$ and $\IC(\mathcal{E}_{C_3})$, respectively.
\end{remark}

\subsection{Kazhdan-Lusztig Conjecture}\label{ssec:KL}

The Langlands-Vogan correspondence defines a pairing \cite{Vogan:Langlands}
\[
\langle\ ,\ \rangle : K\Rep(G_2(F))_{\mathrm{sub}}\times K\Perv_{\dualgroup{G_2}}(X_\mathrm{sub}) 
\to 
\ZZ
\]
between the Grothendieck groups of these two categories that matches irreducible representations $\Rep(G_2(F))$ with simple perverse sheaves only when $\pi$ corresponds to $\IC(\mathfrak{C},\mathcal{L})\ceq \IC(\mathcal{L}_{\mathfrak{C}})$ under the Langlands-Vogan correspondence:
\[
\langle \pi(\phi,r),\IC(\mathfrak{C}_\phi,\mathcal{L}_r) \rangle = (-1)^{\dim \mathfrak{C}_\phi}
\]
and $\langle \pi,\IC(\mathfrak{C},\mathcal{L}) \rangle=0$ otherwise.
Note that this may be expressed in the form
\[
\langle \pi,\mathcal{L}^\sharp \rangle = 1
\]
where $\mathcal{L}^\sharp \ceq \IC(\mathfrak{C},\mathcal{L})[-\dim \mathfrak{C}]$.

The Kazhdan-Lusztig conjecture is the claim that standard modules and standard sheaves are also dual under this pairing. A standard sheaf on $X_\mathrm{sub}$ is the extension by zero $\mathcal{L}^! = (j_{\mathfrak{C}})_!\mathcal{L}$ of an equivariant local system $\mathcal{L}$ on an orbit $\mathfrak{C}\subseteq X_\mathrm{sub}$, where $j_{\mathfrak{C}} : \mathfrak{C} \hookrightarrow X_\mathrm{sub}$ is inclusion.

\begin{proposition}[Kazhdan-Lusztig conjecture for $\Rep(G_2(F))_\mathrm{sub}$]\label{prop:KL}
With reference to the pairing above, let $M$ be a standard module in $\Rep(G_2(F))_\mathrm{sub}$ and let $\mathcal{L}^!$ be a standard sheaf on $X_\mathrm{sub}$. Then
\[
\langle M ,\mathcal{L}^! \rangle = 1
\]
if, for some irreducible admissible $\pi$, $M$ is the standard module for $\pi$ and $\pi$ corresponds to $\IC(\mathcal{C},\mathcal{L})$ under the Langlands-Vogan correspondence; otherwise, $\langle M ,\mathcal{L}^! \rangle=0$.
\end{proposition}

\begin{proof}
In Section~\ref{ssec:SM} we found the standard modules for the irreducible representations in $\Rep(G_2(F))_\mathrm{sub}$: the standard module for $\pi_0$ is $I_{\beta}(1,\pi(1,1))$; the standard module for $\pi_1$ is $I_{\alpha}(1/2,\delta(1))$; the standard module for $\pi_2$ is $I_{\beta}(1/2,\delta(1))$; while $\pi_3$, $\pi_3^\varrho$ and $\pi_3^\varepsilon$ are their own standard modules. Using \eqref{eqn:muic1} and \eqref{eqn:muic}, we find the multiplicity matrix for representations with infinitesimal parameter $\lambda_\mathrm{sub}$ in Table~\ref{table:repmult}.

On the other hand, in Theorem~\ref{thm:stalks} we find the geometric multiplicity matrix for the simple perverse sheaves on $V_\mathrm{sub}$; see especially Table~\ref{table:geomult}.
Since Table~\ref{table:geomult} is the transpose of Table~\ref{table:repmult}, this verifies the Kazhdan-Lusztig correspondence for $\Rep(G_2(F))_\mathrm{sub}$; see also \cite{CFMMX}*{}.
\end{proof}

\begin{remark}
We are grateful to George Lusztig for pointing out to us that Proposition~\ref{prop:KL} follows, with some work, from \cite{Lusztig:Cuspidal2}*{Corollary 10.7}.
\end{remark}

\begin{table}[htp]
\caption{Multiplicity matrix}
\begin{center}
$
\begin{array}{|l|ccccc|c|}
\hline
\mathrm{Standard modules} & \pi_0 & \pi_1 & \pi_2 & \pi_3 & \pi_3^\varrho &\pi_3^\varepsilon \\
     \hline
   M_0 = I_{\beta}(1,\pi(1,1))  &1&1&2&1&1&0 \\
    M_1 = I_{\alpha}(1/2,\delta(1))  &0&1&1&1&0&0\\
    M_2= I_{\beta}(1/2,\delta(1)) &0&0&1&1&1&0\\
    M_3= \pi_3&0&0&0&1&0&0\\
    M_3^\varrho = \pi_3^\varrho&0&0&0&0&1&0\\
     \hline
    M_3^\varepsilon = \pi_3^\varepsilon & 0 & 0 & 0 & 0 & 0 &1\\
     \hline
\end{array}
$
\end{center}
\label{table:repmult}
\end{table}%

\subsection{Main result}\label{ssec:main}

In this section we describe our progress toward proving an adaptation to $G_2(F)$ of the main local result \cite{Arthur:book}*{Theorem~1.5.1} of Arthur's endoscopic classification.

\begin{lemma}\label{lem:Atype}
\begin{enumerate}
\item The Langlands parameters appearing in Proposition~\ref{prop:langlandsparameters} are all of Arthur type. In fact, if we define
\begin{align*}
\psi_0(w,x,y)&=\varphi_{\mathrm{sub}}(y),\\
\psi_1(w,x,y)&=\iota_{\da}(x)~\iota_{\da+2\db}(y),\\
\psi_2(w,x,y)&=\iota_{\da}(y)~\iota_{\da+2\db}(x),\\
\psi_3(w,x,y)&=\varphi_{\mathrm{sub}}(x),
\end{align*}
then each $\psi_i$ is an Arthur parameter and
$$\psi_i(w,x,d_w)=\phi_i(w,x), i=0,1,2,3.$$
\item Let $A_{\psi_i}:=Z_{\dualgroup{G}_2}(\psi_i)/Z_{\dualgroup{G}_2}(\psi_i)^0$ be the component group of $\psi_i$, then we have $A_{\psi_0}=A_{\psi_3}=S_3$ and $A_{\psi_1}=A_{\psi_2}\cong S_2$, where $S_2$ is the order 2 multiplicative group.
\item The image of $s_{\psi_i}\ceq\psi_i(1,1,-1)\in Z_{\dualgroup{G}_2}(\psi_i)$ in $A_{\psi_i}$ is given by: $s_{\psi_0} =1\in S_3$, $s_{\psi_1}=-1\in S_2$, $s_{\psi_2}=-1\in S_2$ and $s_{\psi_3}=1\in S_3$. 
\end{enumerate}
\end{lemma}
\begin{proof}
A direct calculation shows that each $\psi_i$ is an Arthur parameter and satisfies $\psi_i(w,x,d_w)=\phi_i(w,x)$; in particular, note that each $\phi_i$ is trivial on $W_F$, hence bounded on $W_F$.
The computations of $A_{\psi_0}$ and $A_{\psi_3}$ are the same as the computation of $A_{\phi_3}$. Using \eqref{eqn:rootactionontorus}, we can check that
$$A_{\psi_1}=A_{\psi_2}=\{\wh m(1,1), \wh m(1,-1)\}.$$
Thus $A_{\psi_1}=A_{\psi_2}\cong S_2$. Finally, notice that $s_{\psi_0}=s_{\psi_3}=\wh m(1,1)$ and $s_{\psi_1}=s_{\psi_2}=\wh m(1,-1)$, we get the third part of the lemma.
\end{proof}
Notice that 
\[
\psi_3(w,x,y)=\psi_0(w,y,x)
\qquad \mathrm{and}\qquad  \psi_1(w,x,y)=\psi_2(w,y,x).\qedhere
\]

Recall that we write ${\wh A_\psi}$ for the set of equivalence classes of irreducible representations of $A_\psi$ and that we identify these representations with their characters.

\begin{theorem}[See also Theorem~\ref{intromaintheorem}]\label{thm:main}
For every Arthur parameter $\psi$ with subregular infinitesimal parameter there is a finite set $\Pi_{\psi}(G_2(F))$ of irreducible admissible representations in $\Rep(G_2(F))_\mathrm{sub}$ and a function
\[
\begin{array}{rcl}
\langle \ , \ \rangle_\psi : \Pi_{\psi}(G_2(F)) &\to& {\widehat A_\psi} \\
\pi &\mapsto& \langle \ , \pi\rangle_\psi,
\end{array}
\]
defined using a microlocal study of the moduli space of Langlands parameters $X_\mathrm{sub}$, such that
\begin{enumerate} 
\labitem{(a)}{main:tempered} if $\psi$ is trivial on $\SL_2(\C)$ then all the representations in $\Pi_{\psi}(G_2(F))$ are tempered and $\langle \ , \ \rangle_\psi$ is bijective;
\labitem{(b)}{main:nontempered} if $\psi$ is not trivial on $\SL_2(\C)$ then $\Pi_\psi(G_2(F))$ contains non-tempered representations and $\langle \ , \ \rangle_\psi$ need not be bijective;
\labitem{(c)}{main:spherical} if $\pi$ is spherical then $\langle \ , \pi \rangle_\psi = \1$, the trivial representation of $A_\psi$.
\end{enumerate}
The finite sets $\Pi_{\psi}(G_2(F))$ are ABV-packets, as defined in \cite{CFMMX}*{}; they are listed explicitly in Table~\ref{table:ALpackets}.
The functions $\langle \ , \ \rangle_\psi : \Pi_{\psi}(G_2(F)) \to {\widehat A_\psi}$ are displayed in Table~\ref{table:nevs}, where $\1$ is the trivial representation, $\varrho$ denotes the reflection representation of $A_{\psi_0}$ and $A_{\psi_3}$, $\varepsilon$ denotes the sign representation of $A_{\psi_0}$ and $A_{\psi_3}$, and $\tau$ denotes the non-trivial character of $A_{\psi_1}$ and $A_{\psi_2}$. 
\end{theorem}

\begin{table}[htp]
\caption{The functions $\langle \ , \ \rangle_\psi : \Pi_{\psi}(G_2(F)) \to \Rep(A_\psi)$ appearing in Theorem~\ref{thm:main}}
\begin{center}
$
\begin{array}{| c || cccc |}
\hline
\Rep(G_2(F))_\mathrm{sub} & \Rep(A_{\psi_0})  &  \Rep(A_{\psi_1}) &  \Rep(A_{\psi_2}) &  \Rep(A_{\psi_3})  \\
\hline\hline
\pi_0 & \1 & 0 & 0 & 0 \\
\pi_1  & \varrho & \1 & 0 & 0 \\
\pi_2 & 0 & \tau & \1 & 0 \\
\pi_3 & 0 & 0 & 0 & \1 \\
\pi_3^\varrho &  0 & 0 & \tau & \varrho \\
\hline
\pi_3^\varepsilon & \varepsilon & \1  & \tau & \varepsilon \\
\hline
\end{array}
$
\end{center}
\label{table:nevs}
\end{table}%

\begin{table}[htp]
\caption{L-packets and A-packets for all admissible representations of $G_2(F)$ with subregular infinitesimal parameter}
\begin{center}
$
\begin{array}{l c l }
\mathrm{A-packet} &\hskip1cm& \mathrm{L-packet} \\ \hline
\Pi_{\psi_0}(G_2(F))  =  \{ \pi_0,  \pi_1, \pi_3^\varepsilon \} &\hskip1cm&
\Pi_{\phi_0}(G_2(F)) = \{ \pi_0\} \\
\Pi_{\psi_1}(G_2(F)) =  \{ \pi_1, \pi_2, \pi_3^\varepsilon \} &\hskip1cm&
\Pi_{\phi_1}(G_2(F)) = \{ \pi_1\} \\
\Pi_{\psi_2}(G_2(F)) = \{ \pi_2, \pi_3^\varrho, \pi_3^\varepsilon \} &\hskip1cm& 
\Pi_{\phi_2}(G_2(F)) = \{ \pi_2\}  \\
\Pi_{\psi_3}(G_2(F)) = \{ \pi_3, \pi_3^\varrho, \pi_3^\varepsilon \} &\hskip1cm& 
\Pi_{\phi_3}(G_2(F))  = \{ \pi_3, \pi_3^\varrho, \pi_3^\varepsilon \} 
\end{array}
$
\end{center}
\label{table:ALpackets}
\end{table}%

\begin{proof}
The function $\langle \ , \ \rangle_\psi : \Pi_{\psi}(G_2(F)) \to \Rep(A_\psi)$ is constructed as follows. 
First, observe that
\[
\Big( \Rep(G_2(F))_\mathrm{sub} \Big)^\mathrm{simple}_{/\mathrm{iso}}
= 
\mathop{\cup}\limits_{\psi} \Pi_{\psi}(G_2(F)),
\]
since all Langlands parameters with subregular infinitesimal parameter are of Arthur type, by Lemma~\ref{lem:Atype}.
The Langlands-Vogan correspondence, Proposition~\ref{prop:LVC}, gives a bijection
\[
 \Big( \Rep(G_2(F))_\mathrm{sub} \Big)^\mathrm{simple}_{/\mathrm{iso}}
\to
\left( \Perv_{\dualgroup{G_2}}(X_\mathrm{sub}) \right)^\mathrm{simple}_{/\mathrm{iso}}.
\]
When combined with the categorical equivalence 
\[
\Perv_{\dualgroup{G_2}}(X_\mathrm{sub}) 
\iso
\Perv_{H_\mathrm{sub}}(V_\mathrm{sub}), 
\] 
this gives a bijection
\[
\mathcal{P} : \left( \Rep(G_2(F))_\mathrm{sub} \right)^\mathrm{simple}_{/\mathrm{iso}}
\to
\left( \Perv_{H_\mathrm{sub}}(V_\mathrm{sub}) \right)^\mathrm{simple}_{/\mathrm{iso}}.
\]

Next, recall the "microlocal vanishing cycles" functor
\[
\NEvs: \Perv_{H_\mathrm{sub}}(V_\mathrm{sub}) \to \Loc_{H_\mathrm{sub}}(T^*_{H_\mathrm{sub}}V_\mathrm{sub}^\mathrm{reg}),
\]
introduced in \cite{CFMMX}*{Section~7.10}.
Since the components of $T^*_{H_\mathrm{sub}}V_\mathrm{sub}^\mathrm{reg}$ are the conormal bundles $T^*_{C}V_\mathrm{sub}^\mathrm{reg}$ as $C$ ranges over $H_\mathrm{sub}$-orbits $C\subset V_\mathrm{sub}$, we have
\[
\Loc_{H_\mathrm{sub}}(T^*_{H_\mathrm{sub}}V_\mathrm{sub}^\mathrm{reg})
\iso
\mathop{\oplus}\limits_{C} \Loc_{H_\mathrm{sub}}(T^*_{C}V_\mathrm{sub}^\mathrm{reg}).
\]
By Lemma~\ref{lem:Atype}, each $H_\mathrm{sub}$-orbit $C\subset V_\mathrm{sub}$ is of Arthur type and the collection of these orbits is indexed by the Arthur parameters $\psi_0, \ldots, \psi_3$ defined above, so we write $C=C_\psi$.
By \cite{CFMMX}*{Proposition~6.7.1} and Lemma~\ref{lem:generic}, the equivariant fundamental group of $T^*_{C_\psi}V_\mathrm{sub}^\mathrm{reg}$ is $A_\psi$.
Thus,
\[
\mathop{\oplus}\limits_{C} \Loc_{H_\mathrm{sub}}(T^*_{C}V_\mathrm{sub}^\mathrm{reg})
\iso
\mathop{\oplus}\limits_{\psi} \Rep(A_\psi).
\]
Thus, $\NEvs$ defines a function
\[
\left( \Perv_{H_\mathrm{sub}}(V_\mathrm{sub})  \right)^\mathrm{simple}_{/\mathrm{iso}} \to \prod_{\psi} \Rep(A_\psi)^\mathrm{simple}_{/\mathrm{iso}}.
\]
When combined with the first paragraph in this proof, this defines
\[
\operatorname{nevs} : \Big( \Rep(G_2(F))_\mathrm{sub} \Big)^\mathrm{simple}_{/\mathrm{iso}} \to \prod_{\psi} \Rep(A_\psi)^\mathrm{simple}_{/\mathrm{iso}}.
\]
Now, for each $\psi$, compose this with the projection to $\Rep(A_\psi)^\mathrm{simple}_{/\mathrm{iso}}$ to define
\[
\operatorname{nevs}_\psi  : 
\Big( \Rep(G_2(F))_\mathrm{sub} \Big)^\mathrm{simple}_{/\mathrm{iso}} \to \Rep(A_\psi)^\mathrm{simple}_{/\mathrm{iso}}.
\]
Define 
\begin{equation}\label{eqn:Pipsi}
\Pi_\psi(G_2(F)) \ceq \operatorname{supp} \operatorname{nevs}_\psi
\end{equation}
and define $\langle \ , \ \rangle_\psi : \Pi_\psi(G_2(F)) \to {\widehat A_\psi}$ by
\begin{equation}\label{eqn:tracenevs}
\langle a , \pi \rangle_\psi \ceq  \trace_{a } \NEvs_\psi \mathcal{P}(\pi).
\end{equation}

We calculate the functor $\NEvs$ on simple objects in Theorem~\ref{thm:NEvs}; when combined with Proposition~\ref{prop:cubics}, the results are presented in Table~\ref{table:NEvs}.
Now Items~\ref{main:tempered}, \ref{main:nontempered} and \ref{main:spherical} follow by inspection of the table, in the last case using Lemma~\ref{lem:generic} to see that $\pi_0$ is spherical.
\end{proof}

\begin{remark}
Using the definition of $\langle \ , \ \rangle_\psi$ in \eqref{eqn:Pipsi} above, the following statement is readily verified, case-by-case:
\begin{equation}
\langle a , \pi \rangle_\psi = (-1)^{\dim C_\psi} (-1)^{\dim \operatorname{supp} \mathcal{P}(\pi)} \trace_{a a_\psi} \NEvs_\psi \mathcal{P}(\pi).
\end{equation}
This establishes \cite{CFMMX}*{Conjecture 1, (b) and (c)}; see especially \cite{CFMMX}*{(110)}.
\end{remark}

\begin{remark}\label{remark:forthcoming}
In the forthcoming papers \cite{CFZ:unipotent} and \cite{CFZ:endoscopic}, we extend and strengthen Theorem~\ref{thm:main} by defining $\langle \ , \ \rangle_\phi : \Pi_{\psi}(G_2(F)) \to {\widehat A_\phi}$ for every unramified Langlands parameter $\phi$ and by showing that it is compatible with the theory of endoscopy for $A_2 = \PGL_3$ and $D_2 = \SO_4$ and twisted endoscopy for $\PGSpin_8$.
This gives a construction of Arthur packets of all unipotent representations of $G_2(F)$.
\end{remark}

\begin{remark}
Observe that, by construction, the function $\langle\ ,\ \rangle_\psi$ is the shadow of a functor $\Rep(G_2(G))_\mathrm{sub} \to \Loc_{\dualgroup{G_2}}(T^*_{\dualgroup{G_2}}X_\mathrm{sub}^\mathrm{reg})$ if the suggestion of Remark~\ref{remark:CLLC} is correct.
\end{remark}

\begin{remark}
Nothing more general can be said about the case treated in Theorem~\ref{thm:main}, Item~\ref{main:nontempered}.
In our case each $\Pi_\psi(G_2(F))$ does contain a tempered representation, this is not to be expected more generally. 
Also, in our case, $\langle \ ,\ \rangle_\psi : \Pi_\psi(G_2(F)) \to {\widehat A_\psi}$ is always surjective but fails to be a bijection for $\psi_1$ and $\psi_2$.
It is not difficult to find examples that show that $\langle \ ,\ \rangle_\psi : \Pi_\psi(G_2(F)) \to {\widehat A_\psi}$ need not be surjective, too.
For instance, for the split connected reductive group $\SO_5$ there is an Arthur parameter, described in \cite{CFMMX}*{Section~15.1.1} as $\psi_2$, for which $\langle \ ,\ \rangle_{\psi_2} : \Pi_{\psi_2}(\SO_5(F)) \to {\widehat A_{\psi_2}}$ is not surjective, and $\Pi_{\psi_2}(\SO_5(F))$ contains no tempered representation. 
\end{remark}

\subsection{Stable distributions}\label{ssec:stable}

We now repeat Theorem~\ref{introstabletheorem}.
Recall that we write $\Theta_{\pi}$ for the Harish-Chandra distribution character determined by an admissible representation $\pi$.

\begin{theorem}[See also Theorem~\ref{introstabletheorem}]\label{thm:stable}
For every Arthur parameter $\psi$ with subregular infinitesimal parameter, consider the invariant distribution 
\begin{equation}\label{eqn:Thetapsi}
\Theta_{\psi} \ceq \sum_{\pi \in\Pi_{\psi}(G_2(F))} \langle a_\psi , \pi\rangle_\psi \ \Theta_{\pi},
\end{equation}
where $a_\psi$ is the image of $s_\psi \ceq \psi(1,-1)$ in $A_\psi$.
Suppose $\Theta_{\psi}$ is stable when $\psi$ is tempered.
Then the distributions $\Theta_\psi$ are stable for all Arthur parameters $\psi$.
The stable distributions $\Theta_\psi$ are displayed in Table~\ref{table:stable}.
\end{theorem}

\begin{table}[htp]
\caption{Distributions attached to Arthur parameters with subregular infinitesimal parameter}
\begin{center}
$
\begin{array}{l}
\Theta_{\psi_0} = \Theta_{\pi_0} + 2\Theta_{\pi_1} + \Theta_{\pi_3^\varepsilon} , \\
\Theta_{\psi_1} = \Theta_{\pi_1} - \Theta_{\pi_2} + \Theta_{\pi_3^\varepsilon} , \\
\Theta_{\psi_2} = \Theta_{\pi_2} - \Theta_{\pi_3^\varrho} - \Theta_{\pi_3^\varepsilon} , \\
\Theta_{\psi_3} = \Theta_{\pi_3} + 2\Theta_{\pi_3^\varrho} + \Theta_{\pi_3^\varepsilon} .
\end{array}
$
\end{center}
\label{table:stable}
\end{table}%

\begin{proof}
From \eqref{eqn:tracenevs} we see that
\[
\langle s_\psi , \pi \rangle_\psi = \trace_{s_\psi} \NEvs_{C_\psi} \mathcal{P}(\pi),
\]
where $\mathcal{P}: \pi\mapsto \mathcal{P}(\pi)$ is given in Proposition~\ref{prop:LVC}.
In Lemma~\ref{lem:Atype} we found $s_{\psi_0} =1 = s_{\psi_3}$ and  $s_{\psi_1} =-1 = s_{\psi_2}$. 
Now Table~\ref{table:nevs} gives Table~\ref{table:stable}.

From Section~\ref{ssec:KL}, recall that $M_i$ denotes the standard module for $\pi_i$ for $i=0,1,2$; the remaining three representations, $\pi_3$, $\pi_3^\varrho$ and $\pi_3^\varepsilon$, are their own standard modules. 
 Since $M_0$, $M_1$ and $M_2$ are obtained by parabolic induction from representations of $\GL_2(F)$, they are stable standard modules. 
Write $\Theta_{M_0}$, $\Theta_{M_1}$ and $\Theta_{M_2}$ for the Harish-Chandra distribution characters attached to these representations. 
Then these distributions are stable, by \cite{Arthur:Character}. 
The distributions $\Theta_{\psi_i}$, for $i=0, 1, 2$ are expressed in terms of these four stable distributions, $\Theta_{M_0}$, $\Theta_{M_1}$, $\Theta_{M_2}$ and $\Theta_{\psi_3}$, as follows:
\[
\begin{pmatrix}
\Theta_{\psi_0} \\
\Theta_{\psi_1} \\
\Theta_{\psi_2} \\
\Theta_{\psi_3} \\
\end{pmatrix}
=
\begin{pmatrix}
1 & 1 & -3 & 1 \\
0 & 1 & -2 & 1 \\
0 & 0 & 1 & -1 \\
0 & 0 & 0 & 1
\end{pmatrix}
\begin{pmatrix}
\Theta_{M_0} \\
\Theta_{M_1} \\
\Theta_{M_2} \\
\Theta_{\psi_3} \\
\end{pmatrix}
\]
These identities follow from the multiplicity matrix of Section~\ref{ssec:KL} and the explicit form for $\Theta_{\psi_i}$ listed in Table~\ref{table:stable}, derived from the definition stated above.
Note that $\psi_3$ is the unique tempered Arthur parameter; therefore, the statement of Theorem~\ref{thm:stable} assumes $\Theta_{\psi_3}$ is stable.
This concludes the proof that $\Theta_{\psi_0}$, $\Theta_{\psi_1}$, $\Theta_{\psi_2}$ are stable.\end{proof}

\begin{remark}\label{remark:basis}
As mentioned in the Introduction, we expect to prove that $\Theta_{\psi_3}$ is stable.
Moreover, we expect that $\Theta_{\psi_0}$, $\Theta_{\psi_1}$, $\Theta_{\psi_2}$ and $\Theta_{\psi_3}$ is a basis for the intersection of the space of stable distributions on $G_2(F)$ and the space of invariant distributions spanned by the $\Theta_\pi$ as $\pi$ ranges over irreducible representations in $\Rep(G_2(F))_\mathrm{sub}$ up to equivalence.
\end{remark}

\subsection{Aubert involution and the Fourier transform of perverse sheaves}\label{ssec:Aubert}

\begin{proposition}[\cite{Muic}]\label{prop:Aubert}
The Aubert involution on the Grothendieck group $K\Rep(G_2(F))_\mathrm{sub}$ is given by
\[
\begin{array}{rcl}
\pi_0 	&\mapsto& +\pi_3, \\
\pi_1		&\mapsto& +\pi_3^\varrho, \\
\pi_2		&\mapsto& +\pi_2, \\
\pi_3	&\mapsto& +\pi_0, \\
\pi_3^\varrho 	&\mapsto& +\pi_1, \\
\pi_3^\varepsilon	&\mapsto& +\pi_3^\varepsilon .
\end{array}
\]
\end{proposition}
Consequently the Aubert involution acts on Arthur packets according to
\[
\begin{array}{rcl}
\Pi_{\psi_0}(G_2(F))  =  \{ \pi_0,  \pi_1, \pi_3^\varepsilon \}   
	&\mapsto& \Pi_{\psi_3}(G_2(F)) = \{ \pi_3, \pi_3^\varrho, \pi_3^\varepsilon \} \\
\Pi_{\psi_1}(G_2(F)) =  \{ \pi_1, \pi_2, \pi_3^\varepsilon \} 
	&\mapsto& \Pi_{\psi_2}(G_2(F)) = \{ \pi_2, \pi_3^\varrho, \pi_3^\varepsilon \} . 
\end{array}
\]
This is compatible with the observations 
\[ \psi_2(w,x,y) = \psi_1(w,y,x) \quad\mathrm{ and }\quad \psi_3(w,x,y) = \psi_0(w,y,x). \]

\begin{theorem}[See also Theorem~\ref{introstabletheorem}]\label{thm:Ft}
Under the Langlands-Vogan correspondence, Proposition~\ref{prop:LVC}, the Aubert involution on $\Rep(G_2(F))_\mathrm{sub}$ is given by the Fourier transform on $\Perv_{\dualgroup{G_2}}(X_\mathrm{sub})$.
\end{theorem}

\begin{proof}
This is a direct consequence of 
 Theorem~\ref{thm:geoFt}.
\end{proof}

\section{Equivariant perverse sheaves on $\det^{-1} \otimes \Sym^3$}\label{sec:cubics}

In this Section we study the simple objects in the category 
\[
\Perv_{\GL_2}(\det\!\,^{-1} \otimes \Sym^3)
\] 
of $\GL_2$-equivariant perverse sheaves on the affine variety space $P_3[x,y]$ of homogeneous cubic polynomials in two variables $x,y$ for the action 
\[
\det\!\,^{-1} \otimes \Sym^3 : \GL_2 \to \Aut(P_3[x,y])
\]
defined as follows: 
 for $h\in \GL_2$ and $r\in P_3[x,y]$, 
\begin{equation}\label{eqn:action}
\begin{array}{rcl}
&& \hskip-8mm ((\det\!\,^{-1} \otimes \Sym^3)(h).r)(x,y)\\
&\ceq& \det(h)^{-1}r((x,y)h)\\
&=&\det(h)^{-1}r(ax+cy,bx+dy), \qquad h=\left(\begin{smallmatrix} a&b\\ c&d \end{smallmatrix}\right). 
\end{array}
\end{equation}
Specifically, we find the geometric multiplicity matrix for this category, calculate the Fourier transform and the microlocal vanishing cycles of its simple objects.

Although this may be of independent interest, the point of departure for our application of this study of $\Perv_{\GL_2}(\det\!\,^{-1} \otimes \Sym^3)$ is given by Proposition~\ref{prop:cubics}.

\begin{proposition}\label{prop:cubics}
To $r=(r_0,r_1,r_2,r_3)^t\in V_{\mathrm{sub}}$, we associate the polynomial $r(x,y)\in P_3[x,y]$ given by
 \[ r(x,y) = r_0y^3-3r_1y^2x-3r_2yx^2-r_3x^3. \]
The map $V_\mathrm{sub} \to P_3[x,y]$ given by $r\mapsto r(x,y)$ defines an equivalence of representations of $H_\mathrm{sub}$ on $V_\mathrm{sub}$ as given in Section~\ref{sec:Vsub} with the representation $\det^{-1}\otimes \Sym^3$ of $\GL_2$ on $P_3[x,y]$ given here.
\end{proposition}
The proof of Proposition~\ref{prop:cubics} is elementary and will be omitted. 

Note that Proposition~\ref{prop:cubics} says that Langlands parameters for $G_2$ with infinitesimal parameter $\lambda_\mathrm{sub}$ may be interpreted as homogeneous cubics in two variables.  

The main geometric results of this article are Theorems~\ref{thm:stalks}, \ref{thm:Evs}, \ref{thm:geoFt} and \ref{thm:NEvs}.
These results are used to prove the main arithmetic results of this article in Section~\ref{sec:main}.

\subsection{Simple equivariant perverse sheaves}\label{ssec:cubicorbits}

In this section we enumerate all six simple $\GL_2$-equivariant perverse sheaves on $\det^{-1} \otimes \Sym^3$ 

Recall from \cite{BBD}*{Th\'eor\`eme 3.4.1} that, for a connected group $H$ acting on a connected variety $V$, simple equivariant perverse sheaves all take the form of the perverse extension
\begin{equation}
\IC(\mathcal{L}_C) = j_{!*}\mathcal{L}_C[\dim C]
\end{equation}
of a simple equivariant local system  $\mathcal{L}_{C}$ on $C$, where $C\subseteq V$ is an $H$-orbit and where $j: C\hookrightarrow V$ is the inclusion. 
Thus, to find all simple $\GL_2$-equivariant perverse sheaves on $P_3[x,y]$ we must find all $\GL_2$-orbits $C\subset P_3[x,y]$ and then find all simple $\GL_2$-equivariant local systems in each such $C$. Recall that such $\mathcal{T}_C$ may be viewed as an irreducible representation of the equivariant fundamental group 
\[
A_C \ceq \pi_0(Z_{\GL_2}(c)) = \pi_1(C,c)_{Z_{\GL_2}(c)},
\]
for a choice of $c\in C$.

Let $P_1[x,y]$ be the vector space of linear polynomials in $x$ and $y$.
Let $P_1[x,y]_\times$ be the monoid of non-zero linear polynomials in $x$ and $y$ and define $P_1[x,y]_\times \to \mathbb{P}^1$ by $u_1y-u_2x \mapsto [u_1:u_2]$; this gives $\mathbb{P}(P_1[x,y]) \ceq P_1[x,y]_\times\!/\mathbb{G}_m \iso \mathbb{P}^1$.
This isomorphism is equivariant for the action $h: [u_1:u_2] \mapsto [u_1:u_2] h^{-1}$ of $\GL_2$ on $\mathbb{P}^1$ and the action $h: r(x,y) \mapsto \det(h)^{-1} r((x,y)h)$ of $\GL_2$ on $P_1[x,y]$, which is the restriction of the action of $\GL_2$ on $P_3[x,y]$ appearing in \eqref{eqn:action}.
Henceforth, we identify $\mathbb{P}^1$ with $\mathbb{P}(P_1[x,y])$ using this isomorphism and these actions.
In particular, we will write $[u] \in \mathbb{P}^1$ for $u(x,y)= u_1y-u_2x\in P_1[x,y]_\times$ or $[u]=[u_1:u_2]\in \BP^1$ for $u=(u_1,u_2)\in \mathbb{A}^2_\times\ceq \mathbb{A}^2\setminus \{(0,0)\}$, depending on the context, and fervently hope this will not cause confusion.

The $\GL_2$ action $\det^{-1}\otimes\Sym^3$ on $P_3[x,y]$ has the following 4 orbits:
\begin{enumerate}
    \item $C_0=\wpair{0};$
    \item $C_1=\wpair{u^3 \tq u \in P_1[x,y], p\ne 0};$
    \item $C_2=\wpair{u^2u'  \tq u,u'\in P_1[x,y],\mathrm{\ and\ } u,u' \mathrm{\ are\ linearly\ independent} };$
    \item $C_3=\wpair{uu'u''  \tq u,u',u''\in P_1[x,y], \mathrm{\ and\ } u,u',u'' \mathrm{\ are\ linearly\ independent}}.$
\end{enumerate}
Under the identification of the action of $H_{\mathrm{sub}}$ on $V_{\mathrm{sub}}$ with $\det^{-1}\otimes \Sym^3$ given by Proposition \ref{prop:cubics}, these 4 orbits match those introduced in Lemma \ref{lem:fg}, so we use the same notation here.

We note that 
\begin{align*}
 \ov C_1 &= C_1 \cup C_0, \\
 \ov C_2 &= C_2 \cup C_1 \cup C_0 , \\
  \ov C_3 &= C_3 \cup C_2 \cup C_1 \cup C_0,
 \end{align*}

We now introduce algebraic descriptions of these orbits by introducing auxilliary quantities which distinguish these orbits.
Given
\[ r(x,y) = r_0y^3-3r_1y^2x-3r_2yx^2-r_3x^3\in P_3[x,y] \]
 we may introduce the auxiliary element of  $P_2[x,y]$
 \begin{align*}
  \Delta_{r}(x,y) &\ceq  \frac{1}{4}{\rm det}({\rm Hess}(r) )\\
  &=  -9(r_2r_0+r_1^2)y^2   -9(r_0r_3+r_1r_2) xy  + 9(r_1r_3-r_2^2)x^2.
  \end{align*}
We set 
\[ d_0(r) \ceq -9(r_2r_0+r_1^2),\; d_1(r) \ceq -9(r_0r_3+r_1r_2),\; d_2(r) \ceq 9(r_1r_3-r_2^2) \]
so that
\[
 \Delta_{r}(x,y)   =  d_0 y^2 + d_1xy + d_2 x^2.
\]
The discriminant of the polynomial $r(x,y)$ is then given by
 \[ 
 D_{r} =  d_1^2 - 4 d_0d_2 = 81(r_0r_3+r_1r_2)^2 + 4*81(r_2r_0+r_1^2)(r_1r_3-r_2^2). \]
With these equations we now have
\begin{lemma}
\begin{equation}
\ov C_2 =  \{ r(x,y) \in P_3[x,y] \tq D_{r} = 0 \}  
\end{equation}
and
\begin{equation}
\ov C_1 = \{ r(x,y)\in P_3[x,y] \tq \Delta_{r}(x,y) = 0 \}.
\end{equation}
\end{lemma}
\begin{proof}
That the discriminant, $D_r$, vanishes precisely when a polynomial has a repeated root is well know for degree three polynomials in one variable
as is the fact that the second derivative vanishes when a polynomial has a triple root.

We may understand this lemma as a generalization of these claims. The proof is as follows.
We first note that
\[ {\rm Grad}(r)_{(x,y)}=  \frac{1}{2}\begin{pmatrix} x & y \end{pmatrix} {\rm Hess}(r)_{(x,y)} \qquad r(x,y) = \frac{1}{3}{\rm Grad}(r)_{(x,y)}\begin{pmatrix} x & y \end{pmatrix}^t. \]
From which we may conclude that $r$ has a double point, at say $[a:b]$, if and only if ${\rm Grad}(r)_{(a,b)} = 0$ and a triple point at $[a:b]$ if and only if  ${\rm Hess}(r)_{(a,b)} = 0$.

It follows from this that $r$ has a double at $[a:b]$ if and only if ${\rm Hess}(r)_{(a,b)}$ has kernel precisely $[a:b]$ which occurs if and only if $\Delta_r$ has a double root at $[a:b]$.
Consequently $r$ has a double point somewhere if and only if $\Delta_r$ has a double root which occurs if and only if the discriminant, $D_r$, of $\Delta_r$ is zero.

Now, exploiting the symmetry
\[ \begin{pmatrix} x & y\end{pmatrix} {\rm Hess}(r)_{(a,b)} = \begin{pmatrix} a & b \end{pmatrix} {\rm Hess}(r)_{(x,y)} \]
we can conclude that  $ {\rm Hess}(r)_{(a,b)} $ is identically zero for any $[a:b]$ precisely when the polynomial ${\rm det}({\rm Hess}(r) )$ is identically zero.
\end{proof}
The maps $P_3[x,y] \rightarrow P_2[x,y]$ and $P_3[x,y] \rightarrow \mathbb{C}$ given by $r \mapsto \Delta_r$ and $r \mapsto D_{r}$ are equivariant for the action $\Sym^2$ on $P_2[x,y]$ and $\det^2$ on $\mathbb{A}^1$.

From the above descriptions we obtain the following
\begin{lemma}\label{lem:stabilizers}
The stabilizers for elements in the various orbits are the following:
\begin{enumerate}
\item The stabilizer of $p^3\in C_1$ is isomorphic to the group 
\[ 
\left\{ \left( \begin{smallmatrix} a & b  \\ 0 & a^2 \end{smallmatrix}\right) | b\in \mathbb{A}^1,d\in\mathbb{A}^1_\times \right\} 
\]
by selecting any basis for $P_1[x,y]$ in which $p$ is the first basis vector.
\item The stabilizer of $p^2q\in C_2$ is isomorphic to the group 
\[ \left\{ \left(\begin{smallmatrix} 1 & 0 \\ 0 & d \end{smallmatrix} \right) | d\in\mathbb{A}^1_\times \right\} 
\]
by selecting the basis $p,q$ for $P_1[x,y]$.
\item The stabilizer of $pqr \in C_3$ is isomorphic to the group $S_3$ as the permutation group on the lines $[p],[q],[r] \in \mathbb{P}(P_1[x,y])$.
\end{enumerate}
From which it follows
\[
\begin{array}{ll} 
A_{C_0} =1 &\qquad A_{C_1} = 1\\
A_{C_2} = 1 &\qquad A_{C_3} = S_3.
\end{array}
\]
\end{lemma}
\begin{proof}
We first treat the case $p^3\in C_1$. If $h\in \GL_2$ stabilizes $p^3$ it stabilizes the line $[p]$ and hence the stabilizer is contained in the Borel subgroup, $B$, for this line.
If the eigenvalue for $h \in B$ acting on the line $[p]$ is $a$, then $h \cdot p^3 = \frac{a^3}{\det h} p^3$ from which the result follows.

For the case $p^2q\in C_2$.  If $h\in \GL_2$ stabilizes $p^2q$ it stabilizes both the line $[p]$ and $[q]$ and hence the stabilizer is contained in the torus for these two lines.
If the eigenvalue for $h \in B$ on these lines are respectively $a$ and $d$ then $h \cdot p^2q =a p^2q $ from which the result follows.

for the case  of $pqr \in C_3$.  If $h\in \GL_2$ stabilizes $pqr$ it stabilizes the collection of lines $\{ [p], [q], [r] \}$. The stabilizer of this collection is the permutation group crossed with the center of $\GL_2$. By noting that for $h = a I$ we have $h\cdot pqr = a pqr$ we conclude the result.

The claim about equivariant fundamental groups now follows by considering component groups of the stabilizers.
\end{proof}

Now we pick the following orbit representatives which correspond to $c_i$ in Lemma \ref{lem:fg} to denote these
\[ \begin{array}{ll}
c_0(x,y)=0\in C_0, & c_1(x,y) = y^3 \in C_1, \\
c_2(x,y) =-3 x y^2\in C_2, & c_3(x,y) = y(y^2-3x^2)\in C_3.
\end{array}
\]

In summary, we have now proved
\begin{proposition}\label{prop:IC}
The simple objects in $\Perv_{\GL_2}(P_3[x,y])$ are
\[
\IC(\1_{C_0}),\ \IC(\1_{C_1}),\ \IC(\1_{C_2}),\ \IC(\1_{C_3}),\ \IC(\mathcal{R}_{C_3}),\ \IC(\mathcal{E}_{C_3}).
\] 
where:
\begin{enumerate}
\item $\1_{C_i}$ is the local system on $C_0$ corresponding to the trivial representation of $A_{C_i}$ for $i=0,1,2,3$ and where 
\item $\mathcal{R}_{C_3}$ is the simple $\GL_2$-equivariant local system on $C_3$ corresponding to the irreducible 2-dimensional (reflection) representation $\varrho$ of $S_3$  and where 
\item $\mathcal{E}_{C_3}$ is the simple $\GL_2$-equivariant local system on $C_3$ corresponding to the non-trivial irreducible 1-dimensional (sign) representation $\varepsilon$ of $S_3$.
\end{enumerate}
\end{proposition}

\subsection{Conormal bundle}\label{ssec:conormal}

We fix an isomorphism $P_3[x,y]^* \iso P_3[x,y]$ according to the pairing $\KPair{}{} : P_3[x,y] \times P_3[x,y] \to \BA^1$ defined by
\begin{equation}\label{eqn:}
\begin{array}{rcl}
&& \hskip-8mm \KPair{r_0y^3 -3 r_1y^2x - 3r_2yx^2 - r_3x^3}{s_0y^3 -3 s_1y^2x - 3s_2yx^2 - r_3x^3} \\
&=& r_0 s_0 +3 r_1 s_1 + 3r_2 s_2 + r_3 s_3 .
\end{array}
\end{equation}
We will often write $s(x,y)\in P_3[x,y]^*$ for
\[
s(x,y) = s_0y^3 -3 s_1y^2x - 3s_2yx^2 - r_3x^3,
\]
making use of this isomorphism.

Consider the action of $\GL_2$ on $P_3[x,y]^*$ as follows. For $h\in \GL_2$ and $s\in P_3[x,y]^*$, define $h.s\in P_3[x,y]$ by
\begin{equation}\label{eqn:HactiononV*}
 (h.s)(x,y)=\det(h)f((x,y)(\,^th^{-1}).
\end{equation}
With this action, we can check that the pair $\KPair{}{}$ on $P_3[x,y]\times P_3[x,y]^*$ is $\GL_2$-invariant.

%
The pairing $\KPair{}{}$ admits the following interpretation which is helpful calculations.

\begin{lemma}\label{lem:KPair}
For
\begin{align*}
   r= r(x,y) &= r_0y^3 -3 r_1y^2x - 3r_2yx^2 - r_3x^3 \in P_3[x,y],\\
   s= s(x,y) &= (v_1y+v_2x)(v_3y+v_4x)(v_5y+v_6x)\in P_3[x,y]^*,
\end{align*}
the pairing $\KPair{r}{s}$ is given by the formula
\[
\KPair{r}{s} = \frac{1}{6}
\begin{pmatrix} v_1 & v_2 \end{pmatrix} 
\bpm \frac{\partial^2r}{\partial y^2} & \frac{\partial ^2 r}{\partial x \partial y}\\ \frac{\partial^2 r}{\partial y\partial x}& \frac{\partial^2 r}{\partial x^2} \epm_{y=v_3,x=v_4} 
\begin{pmatrix} v_5 \\ v_6 \end{pmatrix} 
\]
\end{lemma}
This follows from a direct computation and the details are omitted.\\
%
The conormal variety $\Lambda = T^*_{\GL_2}(P_3[x,y])$ is defined by 
\[
\Lambda \ceq \{ (r,s)\in T^*(P_3[x,y]) = P_3[x,y]\times P_3[x,y]^* \tq [r,s]=0\},
\]
where $[\ ,\ ] : T^*(P_3[x,y])  \to \mathfrak{gl}_2$ is the moment map which, in this case, is given by
\begin{equation}\label{eqn:Lie}
[r,s] = 
\begin{pmatrix} 
r_0s_0+2r_1s_1+r_2s_2 & -r_1s_0+2r_2s_1+r_3s_2 \\
-r_0s_1+2r_1s_2+r_2s_3  & r_1s_1+2r_2s_2+r_3s_3
\end{pmatrix},
\end{equation}
 see Lemma \ref{lemma:momentmap}. We note that the trace of $[r,s]$ is $\KPair{r}{s}$.
To simplify notation slightly, we set
\[
\Lambda_{i} \ceq \{ (r,s)\in  T^*(P_3[x,y]) \tq r\in C_i, \, [r,s]=0 \}
\]

\begin{lemma}
For
\begin{align*}
   r= r(x,y) &= r_0y^3 -3 r_1y^2x - 3r_2yx^2 - r_3x^3 \in P_3[x,y]\\
   s= s(x,y) &= s_0y^3-3 s_1y^2x - 3s_2yx^2 - s_3x^3 \in P_3[x,y]^*,
\end{align*}
we have 
\[
\begin{array}{rcl}
r_0s_0+2r_1s_1+r_2s_2&=& \frac{1}{3} \KPair{ \left(y  \frac{\partial r}{\partial y} \right)}{s}, \\
r_1s_1+2r_2s_2+r_3s_3&=& \frac{1}{3} \KPair{ \left(x  \frac{\partial r}{\partial x} \right)}{s},\\
-r_1s_0+2r_2s_1+r_3s_2&=&\frac{1}{3} \KPair{ \left(y  \frac{\partial r}{\partial x} \right)}{s}, \\
-r_0s_1+2r_1s_2+r_2s_3&=&\frac{1}{3} \KPair{ \left(x  \frac{\partial r}{\partial y} \right)}{s}.
\end{array}
\]
and consequently the ideal generated by these equations is $\GL_2$ stable.
\end{lemma}
This follows a simple calculation and we omit the proof. An analogous result holds if we interchange the role of $r$ and $s$.

The next result, Lemma~\ref{lem:conormal}, will be used repeatedly in calculations appearing in Sections~\ref{ssec:covC0} through \ref{ssec:covC0}, which gives the proof of the main result of Section~\ref{sec:cubics}, Theorem~\ref{thm:Evs}.
In order to state it, review how we pass between non-zero linear polynomials and elements of $\mathbb{P}^1$ in Section~\ref{ssec:cubicorbits} where, to $u(x,y) = u_1y-u_2x\in P_1[x,y]_\times$ we attached $[u]=[u_1:u_2]\in \mathbb{P}^1$ and write $u\vert r$ to mean $u(x,y)$ divides $r(x,y)$ in $P_3[x,y]$. 
Now we consider the dual construction in $P_3[x,y]^*$ and for $v(x,y) = v_1y-v_2x\in P_1[x,y]_\times^*$ we attached $[v]=[v_1:v_2]\in \mathbb{P}^1$ and write $v\vert s$ to mean $v(x,y)$ divides $s(x,y)$ in $P_3[x,y]^*$. 
In Lemma~\ref{lem:conormal} we see the condition $u\perp v$, which means $u_1v_1+u_2v_2=0$.

\begin{lemma}\label{lem:conormal}
\begin{enumerate}
\item If $r=uu'u'' \in C_3$ then we have $[r,s] = 0$ if and only if $s=0$.
\item If $r=u^2u' \in C_2$ then we have $[r,s]=0$ if and only if $s=v^3$ where $u\perp v$.
\item If $r=u^3 \in C_1$ then we have $[r,s] = 0$ if and only if $s=v^2v'$ where $u \perp v$ and $v'\in P_1[x,y]^*$.
\end{enumerate}
It follows that
\[
\begin{array}{rcl}
\Lambda_1 &\iso& \{ ([u], r, s, [v]) \in \BP^1\times C_1\times P_3[x,y]^*\times \BP^1\tq u^3 \vert r,\ u\perp v,\ v^2\vert s \} \\
\Lambda_2 &\iso& \{ ([u], r, s, [v]) \in \BP^1\times C_2\times P_3[x,y]^*\times \BP^1\tq u^2 \vert r,\ u\perp v,\ v^3\vert s \} \\
\Lambda_3 &\iso& \{ ( r, s ) \in  C_3\times P_3[x,y]^*\tq  s = 0 \}.
\end{array}
\]
\end{lemma}

\begin{proof}
Because the ideal generated by $[r,s]=0$ is $\GL_2$ stable it suffices to verify the claim on any fixed base points.
For $C_3$ we may use
\[ r(x,y) = y^3-3x^2y. \]
Then from the second and third equations we obtain $s_2=0$ and $s_1=0$. Using this with the first and fourth equation we then obtain $s_0=0$ and $s_4=0$.

For $C_2$ we may use
\[ r(x,y) =-3 xy^2. \]
Then we quickly obtain that $s_1=s_2=s_3=0$ so that indeed $s(x,y)$ is a scalar multiple of $x^3$, noting that $x\perp y$.

For $C_1$ we may use
\[ r(x,y) = y^3. \]
We then quickly obtain that $r_2=r_3=0$ so that indeed $s(x,y) = x^2(ax+by)$, noting that $x\perp y$ and $(ax+by)\in P_1[x,y]^\ast$ is arbitrary.

Because the desired relations hold for our chosen base points they also hold any point in the orbit.
\end{proof}

Now we find the equivariant microlocal fundamental groups 
\[
A_{C_i}^\mathrm{mic} \ceq \pi_0(Z_{\GL_2}(c_i,d_i)) = \pi_1(\Lambda_{i}^\mathrm{reg},(r,s))_{Z_{\GL_2}(c_i,d_i)},
\]
for fixed base points $(c_i,d_i)\in \Lambda_{i}^\mathrm{reg}$; as the notation suggests.
One might use, for example: 
\[ \begin{array}{ll} (c_0,d_0) = (0,-x(x^2+3y^2)),& (c_1,d_1) = (-3y^2x,x^3),\\ (c_2,d_2) = (y^3,-3x^2),& (c_3,d_3) = (y(y^2-3x^2),0). \end{array} \]

\begin{lemma}
For each $i=0,1,2,3$, the subvariety
\[
\Lambda_{i}^\mathrm{reg} \ceq \{ (r,s)\in  C_i\times C_i^* \tq \, [r,s]=0 \}
\]
is a single $\GL_2$-orbit. Moreover, we have that
the equivariant microlocal fundamental groups are:
\[
\begin{array}{lr}
A_{C_0}^{\mathrm{mic}} = S_3 &  A_{C_0}^{\mathrm{mic}} = S_3 \\
A_{C_1}^{\mathrm{mic}} = S_2 &  A_{C_2}^{\mathrm{mic}} = S_2 .
\end{array}
\]
\end{lemma}
\begin{proof}
We remark that $\GL_2$ obviously acts transitively on $C_0\times C_0^\ast$ and $C_3\times C_3^\ast$ and the claim about equivariant microlocal fundamental groups is an immediate consequence of Lemma  \ref{lem:stabilizers}.

Now, we claim $\GL_2$ acts transitively on
\[ \Lambda_{1}^\mathrm{reg} = (C_1 \times C_1^\ast) \cap \Lambda_1 \]
and
\[ \Lambda_{2}^\mathrm{reg} = (C_2 \times C_2^\ast) \cap \Lambda_2. \]

We consider first $\Lambda_2$. We note that $\Lambda_2$ is a rank $1$ vector bundle over $C_2$, and $\Lambda_2^\mathrm{reg}$ is the complement of the zero-section in this bundle.
To see that $\Lambda_2^\mathrm{reg}$ has a single orbit it suffices to show that for any given $r\in C_2$ the stabilizer of $r$ acts transitively on the complement of
the zero section in the fiber over $\Lambda_{2r}$ over $r$.

Fix $r = u^2v \in C_2$ then by Lemma \ref{lem:stabilizers} we have that the stabilizer of $r$ is isomorphic to
$\left\{ \left(\begin{smallmatrix} 1 & 0 \\ 0 & d \end{smallmatrix} \right) | d\in\mathbb{C}^\times \right\}$ by
in the basis the basis $u,v$ for $P_1[x,y]$.
It follows that the stabilizer then acts on the line $[v]$ by $d^{-2}$, which we note acts transitively on complement of the zero section.
It also follows from this that the stabilizer of a point $(u^2v,v^3) \in \Lambda_1^\mathrm{reg}$ are those elements with $d^2=1$, from which the claim about the equivariant fundamental group follows.

We note that by the symmetry between $\Lambda_1$ and $\Lambda_2$ we do not need to treat explicitly the former case.
However, it is worth noting that in this case the stabilizer of a fixed $r=u^3$ was $\left\{ \left( \begin{smallmatrix} a & b\\ 0 & a^2 \end{smallmatrix}\right) | b\in \mathbb{C},d\in\mathbb{C}^\times \right\}$
with respect to a basis $u$ and $v$ where we may assume $u\perp v$. This stabilizer will act transitively on the complement of the $[v]$.
The line $[v]$ cuts out $(C_1\times \ov C_2^\ast) \cap \Lambda_2$.
\end{proof}

\begin{remark}
The complete list of $\GL_2$-orbits inside $\Lambda$ is given by the following decompositions of the conormal bundles:
\begin{align*} 
\Lambda_0 &= \Lambda_0^\mathrm{reg} \sqcup (C_0 \times C_1^*)  \sqcup (C_0 \times C_2^*) \sqcup (C_0 \times C_3^*) \\
\Lambda_1 &= \Lambda_1^\mathrm{reg}  \sqcup \{ ([u], r, s, [v]) \in \BP^1\times C_1\times C_2^* \times \BP^1 \tq u^3 \vert r,\ u\perp v,\ v^3\vert s \} \sqcup (C_1 \times C_3^*) \\
\Lambda_2 &= \Lambda_2^\mathrm{reg}  \sqcup (C_2 \times C_3^*) \\
\Lambda_3 &= \Lambda_3^\mathrm{reg} , 
\end{align*}
where we note $\Lambda =  \Lambda_0 \sqcup  \Lambda_1 \sqcup  \Lambda_2 \sqcup  \Lambda_3$.
Since each $\Lambda_i^\mathrm{reg}$ is a single $H_\mathrm{sub}$-orbit, the notions of ``regular" and ``strongly regular" and ``generic", as they appear in \cite{CFMMX}, all coincide.
Aside from $\Lambda_i^\mathrm{reg}$ all of the other $\GL_2$-orbits in $\Lambda$ have trivial equivariant fundamental groups, except
\[  \{ ([u], r, s, [v]) \in \BP^1\times C_1\times C_2^* \tq u^3 \vert r,\ u\perp v,\ v^3\vert s \} . \]
We compute the equivariant fundamental group for an element $([u], r, s, [v])$ by considering the stabilizer of $u^3$ in the basis $u$ and $v$. It has the form
\[ h =  \left( \begin{smallmatrix} a & b \\ 0 & a^2 \end{smallmatrix} \right) \]
 Then $h. v^3 = a^{-3} v^3 $ from which we can conclude the equivariant fundamental group is $\mu_3$.
 \end{remark}

The lemma above now give the following

\begin{proposition}\label{prop:conormal}
Consider the cover ${\wt \Lambda} \ceq \{ (g,(r,s))\in \GL_2\times \Lambda \tq g.(r,s) = (r,s)\}$ and the projection $\rho_{\Lambda} : {\wt \Lambda} \to \Lambda$ defined by $\rho(g,(r,s))= (r,s)$. While this is not proper, the pull-back $\rho_{\Lambda^\mathrm{reg}} : \GL_2\times \Lambda^\mathrm{reg} \to \Lambda^\mathrm{reg}$ is finite and hence proper, and
\[
\begin{array}{rcl}
&&\hskip-8mm
(\rho_{\Lambda^\mathrm{reg}})_*\1_{\wt \Lambda}  \\
&=&
\1_{\Lambda_0^\mathrm{reg}} \oplus 2\mathcal{R}_{\Lambda_0^\mathrm{reg}}
\oplus \mathcal{E}_{\Lambda_0^\mathrm{reg}}  \oplus
\1_{\Lambda_1^\mathrm{reg}} \oplus \mathcal{T}_{\Lambda_1^\mathrm{reg}} \\
&& \oplus\ 
\1_{\Lambda_3^\mathrm{reg}} \oplus 2\mathcal{R}_{\Lambda_3^\mathrm{reg}}
\oplus \mathcal{E}_{\Lambda_3^\mathrm{reg}} 
\oplus 
\1_{\Lambda_2^\mathrm{reg}} \oplus \mathcal{T}_{\Lambda_2^\mathrm{reg}} ,
\end{array}
\]
where:
\begin{enumerate}
\item $\1_{\Lambda_{i}^\mathrm{reg}}$ is the constant local system on $\Lambda_{i}^\mathrm{reg}$ for $i=0,1,2,3$;
\item $\mathcal{T}_{\Lambda_{1}^\mathrm{reg}}$ (resp. $\mathcal{T}_{\Lambda_{2}^\mathrm{reg}}$) is the rank $1$ local system on $\Lambda_{1}^\mathrm{reg}$ (resp. $\Lambda_{2}^\mathrm{reg}$) corresponding to the non-trivial character $\tau$ of $A^\mathrm{mic}_{C_1} = S_2$ (resp. $A^\mathrm{mic}_{C_2} = S_2$);
\item $\mathcal{R}_{\Lambda_3^\mathrm{reg}}$ is the rank $2$ local system on $\Lambda_{3}^\mathrm{reg}$ corresponding to the reflection ($2$-dimensional) representation $\varrho$ of $A^\mathrm{mic}_{C_3} = S_3$;
\item $\mathcal{E}_{\Lambda_3^\mathrm{reg}}$ is the rank $1$ local system on $\Lambda_{3}^\mathrm{reg}$ corresponding to the sign representation $\varepsilon$ of $A^\mathrm{mic}_{C_3} = S_3$.
\end{enumerate}
\end{proposition}

\subsection{Microlocal vanishing cycles and the Fourier transform}\label{sssec:Evs}

The microlocal vanishing cycles functor
\[
\Evs: \Perv_{H}(V) \to \Loc_{H}(\Lambda^\mathrm{reg})
\]
is defined in \cite{CFMMX}*{}, for any reductive (not necessarily connected) group $H$ acting on $V\iso \BA^{n}$ with finitely many orbits, where $\Lambda = T^*_{H}(V)$ is the conormal variety to $V$. 
The main result of this section, Theorem~\ref{thm:Evs}, calculates this functor on simple objects in $ \Perv_{\GL_2}(P_3[x,y])$.
From \cite{CFMMX}*{} we observe that, for any $\mathcal{F}\in \Perv_{\GL_2}(P_3[x,y])$, we have
\begin{equation}\label{eqn:EVC}
\Evs_{C_i} \mathcal{F}
= \left(\RPhi_{\KPair{}{} }[-1] \mathcal{F} \boxtimes \1^!_{C_i^*}[\dim C_i^*]\right)\vert_{\Lambda_{i}^\mathrm{reg}}[-4].
\end{equation}
We will calculate \eqref{eqn:EVC} case-by-case in Sections~\ref{ssec:covC0} through \ref{sssec:covC3''}. 

In these sections we also calculate, on simple objects, the Fourier transform 
\[
\Ft : \Perv_{\GL_2}(P_3[x,y]) \to \Perv_{\GL_2}(P_3[x,y]^*),
\]
defined by
\[
\Ft = q_* \RPhi_{\KPair{}{}}[-1] p^*,
\]
where $p : T^*(P_3[x,y])\to P_3[x,y]$ and $q: T^*(P_3[x,y])\to P_3[x,y]^*$ are the obvious projections.
It follows that
\begin{equation}\label{eqn:Ft}
\RPhi_{\KPair{}{}}[-1] \left( \mathcal{F} \boxtimes \1_{P_3[x,y]^*}[4] \right)
= \1^!_{C_0} \boxtimes \left( \Ft  \mathcal{F} \right)
\end{equation} 
and this is what we will use to determine $\Ft  \mathcal{F}$.

\begin{remark}
The usual definition of the Fourier transform, adapted to our case, is the following. Recall the trait $Z\ceq \Spec{\CC[[t]]}$.
Consider the morphism 
\[ \gamma = \KPair{}{}\times\id : T^*(V) = V \times V^* \to Z\times V^* \]
 defined by $\gamma(r,s) = (\KPair{r}{s},s)$.
Let $t : Z\times V^* \to Z$ be the projection, so that $t(z,s) = z$. 
Then the usual definition of the Fourier-Sato transform is $\RPhi_{t}[-1]\gamma_! p^*$ which is supported on $\{ ( 0,s) \tq s\in V^*\} \iso V^*$, where $0\in Z$ is the special fibre of the trait $Z$ \cite{Schurmann}*{Example 6.0.16}. Our definition \eqref{eqn:Ft} is in fact modelled after the definition of ${\wt \Psi}_{P}$ appearing in \cite{KS}*{(3.7.7)}, which, by \cite{KS}*{Theorem 3.7.7}, is equivalent to ${\wt \Phi}_{P'}$, which is the complex analytic version of the Fourier-Sato transform above, see also \cite{Schurmann}*{Section 5.4}.
\end{remark}

\subsection{Overview of the calculations}\label{sssec:Evsoverview}

We now give an overview of the calculations to be performed in Sections~\ref{ssec:covC0} through \ref{sssec:covC3''}.

For each orbit $C_j\subset P_3[x,y]$ we find a proper cover $\rho_j : {\widetilde C_j}\to {\overline C_j}$ where ${\widetilde C_j}$ is smooth. Using the decomposition theorem, we find all the simple equivariant perverse sheaves $\mathcal{F}$ that appear in
\begin{equation}\label{eqn:rhoCj}
(\rho_j)_* \1_{\widetilde C_j}[\dim {\widetilde C_j}].
\end{equation}
Using this, for each such $\mathcal{F}$ we calculate the left hand side of \eqref{eqn:Ft}.
The way we do this is to calculate
\begin{equation}\label{eqn:VCFt}
\begin{array}{rcl}
&&\hskip-8mm \RPhi_{\KPair{}{}}[-1] \left( (\rho_j)_* \1_{\widetilde C_j}[\dim {\widetilde C_j}] \boxtimes \1_{P_3[x,y]^*}[4]\right)\\
&=&
(\rho_j\times\id_{P_3[x,y]^*})_* \RPhi_{{\tilde f}_{j,0}}[-1] \left( \1_{\widetilde C_j}[\dim {\widetilde C_j}] \boxtimes \1_{P_3[x,y]^*}[4]\right),
\end{array}
\end{equation}
where 
${\tilde f}_{j,i} : {\widetilde C_j}\times {\overline C_i^*} \to S$
\[
\begin{tikzcd}
{\widetilde C_j}\times {\overline C_i^*} \arrow{rd}[swap]{(\rho_j \times \id_{\ov C_i^*})} \arrow{rr}{{\tilde f}_{j,i}} &&  S\\
& {\overline C_j}\times {\overline C_i^*} \arrow{ru}[swap]{\KPair{}{}\vert_{{\overline C_j}\times {\overline C_i^*}}} &
\end{tikzcd}
\]
 is the composition of $(\rho_j \times \id) : {\widetilde C_j}\times {\overline C_i^*} \to {\overline C_j}\times {\overline C_i^*}$ with the restriction of the pairing $\KPair{}{} : P_3[x,y] \times P_3[x,y]^*\to S$ to ${\overline C_j}\times {\overline C_i^*}$ 
and isolate \eqref{eqn:Ft}, inductively.

Then, for each $C_i \leq C_j$, we calculate the vanishing cycles 
\begin{equation}\label{eqn:VtC}
\RPhi_{{\tilde f}_{j,i}} \left( \1_{\widetilde C_j}[\dim {\widetilde C_j}]\boxtimes \1_{\overline C_i^*}\right)\vert_{(\rho_j\times \id_{\ov C_i^*})^{-1}\Lambda_i^\mathrm{reg}}
\end{equation}
and the push-forward
\begin{equation}\label{eqn:pVtC}
\begin{array}{rcl}
&&\hskip-8mm (\rho_j \times \id)_* \RPhi_{{\tilde f}_{j,i}} \left( \1_{\widetilde C_j}[\dim {\widetilde C_j}]\boxtimes \1_{\overline C_i^*}\right) \vert_{\Lambda_i^\mathrm{reg}} \\
&=&
\RPhi_{\KPair{}{}} \left(  (\rho_j)_*\1_{\widetilde C_j}[\dim {\widetilde C_j}] \boxtimes \1_{\overline C_i^*}\right)\vert_{\Lambda_i^\mathrm{reg}}.
\end{array}
\end{equation}
Again working inductively, we isolate
\begin{equation}\label{eqn:Evs}
\Evs_{C_i} \mathcal{F} = \left(\RPhi_{\KPair{}{}}[-1] \left( \mathcal{F} \boxtimes \1_{\overline C_i^*}[\dim C_i^*]\right)\right)\vert_{\Lambda_i^\mathrm{reg}}[-4],
\end{equation}
for each simple $\mathcal{F}$ appearing in $(\rho_j)_* \1_{\widetilde C_j}[\dim {\widetilde C_j}]$.

\subsection{$\IC(\1_{C_0})$}\label{ssec:covC0}

This case is degenerate: $\Lambda_{C_0} = C_0 \times C_0^*$ and ${\widetilde C_0} = {\overline C_0} = C_0$ so $\Lambda_{C_0}^{\rho_0} = \Lambda_{C_0}  = C_0 \times C_0^*$ and ${\tilde f}_{0,0}(r,s) = 0$ for all $(r,s)\in \Lambda_{C_0}$.
Therefore, 
\[
\RPhi_{0}[-1] \1_{\Lambda_{C_0}} = \1_{\Lambda_{C_0}}
=
\1_{{C_0}\times {\overline C_0^*}}.
\]
In other words,
\begin{equation}\label{eqn:VtC00}
\RPhi_{f_{0,0}}[-1] \left(\1_{C_0}\boxtimes \1_{\overline C_0^*}[4]\right)
=
\1_{C_0}\boxtimes \1_{\overline C_0^*}[4].
\end{equation}
This computes \eqref{eqn:VtC} in this case.
Note that \eqref{eqn:VtC00} may also be written in the form
\begin{equation}\label{eqn:Ft0}
\RPhi_{\KPair{}{}}[-1] \left( \IC(\1_{C_0})\boxtimes \1_{\overline C_0^*}[4]\right)
=
\1_{C_0}\boxtimes \IC(\1_{C_0^*}).
\end{equation}
It follows that
\[
\Ft \IC(\1_{C_0}) = \IC(\1_{C_0^*}).
\]
and also that
\[
\begin{array}{rcl}
&& \hskip-8mm \Evs_{C_0} \IC(\1_{C_0}) \\
&\ceq& \left(\RPhi_{\KPair{}{}}[-1]\left( \IC(\1_{C_0})\boxtimes\1_{C_0^*}[4]\right)\right)\vert_{\Lambda_0}[-4] \\
&=& \1_{C_0}\boxtimes \IC(\1_{C_0^*})[-4] \\
&=& \1_{\Lambda_{C_0}}.
\end{array}
\]

\subsection{$\IC(\1_{C_1})$}\label{sssec:covC1}

Define
\[
\widetilde C_1 \iso \{ ([u],r) \in  \BP^1\times \ov C_1  \tq u\vert r \} \\
\]
Recall that the condition $u\vert r$ means $u(x,y)=u_1y-u_2x$ divides $r(x,y)$.
This condition is algebraic in the variables that define $r$ and $u$; it corresponds to three polynomial equations, as we see in the alternate description of $\wt C_1$, below.
As a variety over $\CC$,  $\widetilde C_1$ is smooth as the space is a rank $1$ vector bundle over $\mathbb{P}^1$.
We equip ${\wt C_1}$ with the $\GL_2$-action $h([u],r) = (h.[u],h.r)$. 
 
This cover may also be written in the following form:
\begin{align}
\widetilde C_1 
&\iso \{ ([u], r) \in  \BP^1\times P_3[x,y] \tq {\rm Hess}(r)_{u} = 0 \} 
\end{align}
The equivalence of the two descriptions of ${\widetilde C_1}$ follows from the identities
\begin{align*}
{\rm Grad}(r)_{(u_1,u_2)}
&=  
\frac{1}{2}
\begin{pmatrix} u_1 & u_2 \end{pmatrix} {\rm Hess}(r)_{(u_1,u_2)} 
\\
r(u_1,u_2) 
&= \frac{1}{3}{\rm Grad}(r)_{(u_1,u_2)}\,^t\begin{pmatrix} u_1 & u_2 \end{pmatrix}. 
\end{align*}
The above equations imply both the gradient of $r(x,y)$ and $r(x,y)$ itself vanish at $(u_1,u_2)$, but also that $(u_1,u_2)$ is a triple point of $r(x,y)$ and therefore that $u = u_1y-u_2x$ divides $r(x,y)$.
 
Consider the map $\rho_1: \widetilde C_1 \rightarrow \ov C_1$ defined by the projection $\rho_1([u],r) = r$.
 Then $\rho_1$ is proper, as it is a closed immersion composed with a base change of a proper map, and realizes the blowup of $\ov C_1$ at $C_0$.
The map $\rho_1$ is not smooth, but it is an isomorphism over $C_1$.
The fibers of $\rho_1$ are isomorphic to $\mathbb{P}^1$ over $C_0$, hence $\pi_1$ is a semi-small map.
%
%
It follows from the Decomposition Theorem that
 \begin{equation}\label{eqn:pi1}
 (\rho_{1})_*(\1_{\widetilde C_1}[2]) = \IC(\1_{C_0})\oplus \IC(\1_{C_1}) 
 \end{equation}
 
For later use in this paper, we give yet another description of $\rho_1 : {\wt C_1}\to {\ov C_1}$.
Set $\BA^2_\times = \BA^2-\wpair{0}$ and consider the map 
\[
\begin{array}{rcl}
\BA^1\times \BA^2_\times &\rightarrow& \ov C_1 \\
(\lambda, (u_1,u_2)) &\mapsto& \lambda (u_1y-u_2x)^3
\end{array}
\]
Let $\mathbb{G}_m$ act on $\BA^1\times \BA^2_\times$ by 
$$a.(\lambda,(u_1,u_2))=(a^{-3}\lambda,(au_1,au_2)).$$
Then $\BA^1\times \BA^2_\times\rightarrow \ov C_1$ is $\mathbb{G}_m$ invariant. This implies that the map
\begin{equation}\label{eqn:r1}
\begin{array}{rcl}
(\BA^1\times \BA^2_\times)/\mathbb{G}_m &\to& \widetilde C_1\\
\left[\lambda,u_1,u_2\right] &\mapsto& ( \lambda(u_1y-u_2x)^3, [u_1:u_2])
\end{array}
\end{equation}
is an isomorphism and induces the map $\rho_1:\widetilde C_1\rightarrow \ov C_1$ above.


\subsubsection{$\Ft \IC(\1_{C_1})$ and $\Evs_{C_0}\IC(\1_{C_1})$}\label{sssec:Evs01}

We begin by studying the function 
${\tilde f}_{1,0} : {\widetilde C_1}\times {\overline C_0^*} \to S$ locally. 
Using Section~\ref{sssec:covC1}, consider the affine part $\mathbb{A}^2 \hookrightarrow {\widetilde C_1}$ given by $(\lambda,u_2) \mapsto (  [1:u_2], \lambda(y-u_2x)^3)$.
Then, by Lemma~\ref{lem:KPair},
\[
{\tilde f}_{1,0}\vert_{\mathbb{A}^2\times {\overline C_0^*}} (\lambda,u_2,s) 
=  \KPair{\lambda(y-u_2x)^3}{s} 
=  \lambda ( s_0 + 3 s_1 u_2 - 3 s_2 u_2^2 +s_3 u_2^3).
\]
Set $z(\lambda,u_2,s) \ceq \lambda$ and $g(\lambda,u_2,s) \ceq s_0 +3 s_1 u_2 - 3 s_2 u_2^2 + s_3 u_2^3$ as functions on this part of ${\widetilde C_1}\times {\overline C_0^*}$.
Observe that the locus of $z$ is $C_0\times {\overline C_0^*}$ in the part of ${\widetilde C_1}\times {\overline C_0^*}$ and that the locus of $g$ is the affine part of
\[
\{ ([1:u_2],r,s) \in \BP^1\times {\overline C_1}\times  {\overline C_0^*} \tq s(u_2,-1) =0 \}.
\]
If we then turn our attention to the affine part of ${\widetilde C_1}$ parametrized by $(\lambda,u_1) \mapsto [\lambda,(u_1,1)] \mapsto ( [u_1:1], \lambda(u_1y-x)^3)$ then we reach the same conclusions and we can extend the definitions of $z$ and $g$ to all of ${\widetilde C_1}\times C_0^*$ such that 
\[
{\tilde f}_{1,0} = z g.
\]
Using \cite[Corollary 7.2.6]{CFMMX}, it follows that
\[
\RPhi_{f_{1,0}} \1_{{\widetilde C_1}\times {\overline C_0^*}}
=
\RPhi_{zg} \1_{{\widetilde C_1}\times {\overline C_0^*}} 
\]
is the skyscraper sheaf supported by $g=0$ in $C_0\times {\overline C_0^*}$. To describe this support, consider
\begin{equation}
{\widetilde C_0^*} \ceq \{ (s, [v])\in  {\overline C_0^*} \times \BP^1\tq s(v)=0\}
\end{equation}
and let $\rho_0^* : {\widetilde C_0^*} \to {\overline C_0^*}$ be the cover $\rho_0^*(s, [v]) = s$. 
Here it is important to remark that the conditions $[v]\in (\mathbb{P}^1)^*$ and $s(v)=0$ mean, in coordinates, $[v]= [v_1:v_2]$ and $s(v) = s(v_2,-v_1)=0$.
We remark that ${\widetilde C_0^*}$ is precisely the dual of the cover $\rho_3 : {\widetilde C_3} \to {\overline C_3}$ that will appear in Section~\ref{sssec:covC3}.
Then
\[
\RPhi_{zg} \1_{{\widetilde C_1}\times {\overline C_0^*}} 
=
\1_{{\overline C_0}\times {\widetilde C_0^*}}[-1].
\]
In summary, we have found
\begin{equation}\label{eqn:VtC10}
\RPhi_{f_{1,0}}[-1] \left(\1_{\widetilde C_1}[2]\boxtimes \1_{\overline C_0^*}\right)
=
\1_{C_0}\boxtimes \1_{\widetilde C_0^*}.
\end{equation}
This computes \eqref{eqn:VtC} in this case.

We now see that
\[
\begin{array}{rcl}
&& \hskip-8mm (\rho_1\times \id_{\overline C_0*})_* \RPhi_{f_{1,0}}[-1] \left(\1_{\widetilde C_1}[2]\boxtimes \1_{\overline C_0^*}[4] \right) \\
&=&
(\id_{\overline C_1}\times \rho_0^*)_* \left(\1_{C_0}\boxtimes \1_{\widetilde C_0^*}[4]\right) \\
&=&
\1_{C_0} \boxtimes (\rho_0^*)_* \1_{\widetilde C_0^*}[4] \\
&=&
\1_{\overline C_0} \boxtimes \left(\IC(\1_{C_0^*})\oplus \IC(\mathcal{R}_{C_0^*})\right) \\
&=&
\left( \1_{C_0} \boxtimes \IC(\1_{C_0^*}) \right) 
\oplus 
\left( \1_{C_0} \boxtimes \IC(\mathcal{R}_{C_0^*})\right) \\
\end{array}
\]
On the other hand,
\[
\begin{array}{rcl}
&& \hskip-8mm (\rho_1\times \id_{\overline C_0*})_* \RPhi_{f_{1,0}}[-1] \left(\1_{\widetilde C_1}[2]\boxtimes \1_{\overline C_0^*}[4]\right) \\
&=& \RPhi_{\KPair{}{}}[-1] \left( \left( \IC(\1_{C_0})\oplus \IC(\1_{C_1}) \right) \boxtimes \1_{\overline C_0^*}[4]\right) \\
&=& \RPhi_{\KPair{}{}}[-1] \left(  \IC(\1_{C_0}) \boxtimes \1_{\overline C_0^*}[4] \right) \oplus \RPhi_{\KPair{}{}}[-1] \left( \IC(\1_{C_1}) \boxtimes \1_{\overline C_0^*}[4]\right) \\
&=& \left( \1_{C_0}\boxtimes \IC(\1_{C_0^*}) \right) \oplus \RPhi_{\KPair{}{}} \left( \IC(\1_{C_1}) \boxtimes \1_{\overline C_0^*} \right).
\end{array}
\]
For the last equality, we used \eqref{eqn:Ft0}.
Therefore,
\begin{equation}\label{eqn:Ft1}
\RPhi_{\KPair{}{}} [-1]\left( \IC(\1_{C_1})\boxtimes \1_{\ov C_0^*}[4]\right)
= \1_{C_0}\boxtimes \IC(\mathcal{R}_{C_0^*})
\end{equation}
If follows that
\[
\Ft \IC(\1_{C_1}) = \IC(\mathcal{R}_{C_0^*}).
\]

We now restrict  \eqref{eqn:Ft1} to $\Lambda_0 = \Lambda_0^\mathrm{reg}$ to find
\begin{equation}
\begin{array}{rcl}
&& \hskip-8mm \Evs_{C_0}  \IC(\1_{C_1})  \\
&=&
\left( \RPhi_{\KPair{}{}}[-1] \left( \IC(\1_{C_1})\boxtimes \1_{\ov C_0^*}[4]\right)\right)\vert_{\Lambda_0^\mathrm{reg}} [-4] \\
&=&
\left( \1_{C_0}\boxtimes \IC(\mathcal{R}_{C_0^*}) \right)\vert_{\Lambda_0^\mathrm{reg}} [-4] \\
&=&
\1_{C_0}\boxtimes \mathcal{R}_{C_0^*}[-4] \\
&=&
\mathcal{R}_{\Lambda_0^\mathrm{reg}}.
\end{array}
\end{equation}

\subsubsection{$\Evs_{C_1}\IC(\1_{C_1})$}\label{sssec:Evs11}

From Section~\ref{sssec:covC1} recall  isomorphism \eqref{eqn:r1}
\[
\begin{array}{rcl}
(\BA^1\times \BA^2_\times)/\mathbb{G}_m &\to& \widetilde C_1\\
\left[\lambda,u_1,u_2\right] &\mapsto& ( [u_1:u_2], \lambda(u_1y-u_2x)^3)
\end{array}
\]
Consider the affine part $\BA^2 \hookrightarrow {\widetilde C_1}$  defined by $(\lambda,u_2) \mapsto (  [1:u_2], \lambda(y-u_2x)^3)$. 
Set 
\[
r =  \lambda(y-u_2x)^3  
\]
Now give $s\in {\overline C_1^*}$ local coordinates by writing  
\[
s = (y-v_2x)^2(v_3y-v_4x) 
\]
Then, by Lemma~\ref{lem:KPair},
\[
\begin{array}{rcl}
\KPair{r}{s} &=& \lambda \KPair{(y-u_2x)^3}{(y-v_2x)^2(v_3y-v_4x)}\\
&=&  \lambda (1+u_2v_2)^2 (v_3+u_2 v_4) 
\end{array}
\]
We set $z(\lambda,u_2,v_2,v_3,v_4) \ceq \lambda(1+u_2v_2)$ and $g(\lambda,u_2,v_2,v_3,v_4)\ceq v_3+u_2 v_4$ . 
Then
\[
\RPhi_{{\tilde f}_{1,1}} \1_{{\widetilde C_1}\times {\overline C_1^*}}
=
\RPhi_{z^2g} \1_{{\widetilde C_1}\times {\overline C_1^*}}.
\]
Introduce
\[
{\widetilde \Lambda_1} 
\ceq 
\left\{ ([u],r,s,[v])\in \mathbb{P}^1\times {\overline C_1}\times {\overline C_1^*} \times \BP^1 \tq 
\begin{array}{c} 
u\vert r,\  v\vert s \\
v\perp u  
\end{array} \right\}.
\]
Now we restrict these vanishing cycles to
\[
(\rho_1\times \id_{\ov C_1^*})^{-1}\Lambda_1^\mathrm{reg}
\ceq
\{ 
([u],r,s)\in {\wt C_1}\times {C_1^*} \tq r\in C_1, [r,s] =0
\}.
\]
Here we have just used the fact that $\Lambda_1^\mathrm{reg}\subset C_1\times C_1^*$.
We may identify this with 
\[
\left\{ 
([u],r,s,[v])\in  \BP^1\times {C_1}\times{C_1^*}\times \BP^1  \tq \ [r,s] =0,\  
\begin{array}{c}
u\vert r,\ v\vert s \\
u\perp v
\end{array}
\right\}
\]
Observe that $z=0$ on this variety while $g$ is never zero on this variety since if $g(u,r,s,v) =0$ then $s=v^3$ or $s=0$ in which case $s\in {\ov C_2^*}$ which does not contain $C_1^*$. 
Therefore, by \cite{CFMMX}*{}, 
\begin{equation}\label{eqn:VtC11}
\left( \RPhi_{z^2g} \1_{{\widetilde C_1}\times {\overline C_1^*}}\right)\vert_{(\rho_1\times \id_{\ov C_1^*})^{-1}\Lambda_1^\mathrm{reg}}
=
\mathcal{T}_{(\rho_1\times \id_{\ov C_1^*})^{-1}\Lambda_1^\mathrm{reg}},
\end{equation}
where $\mathcal{T}$ is the local system defined by the double cover $\sqrt{g}$.
This computes \eqref{eqn:VtC} in this case.

Now we turn to \eqref{eqn:pVtC}. 
Since $\rho_1$ is an isomorphism over $C_1$, we now have
\[
\begin{array}{rcl}
&&\hskip-8mm \left( (\rho_1\times \id_{\ov C_1^*})_* 
\RPhi_{{\tilde f}_{1,1}}[-1] \1_{\widetilde C_1}[2]\boxtimes \1_{\overline C_1^*}[3] \right)\vert_{\Lambda_1^\mathrm{reg}} [-4]\\
&=&
 (\rho_1\times \id_{\ov C_1^*})_* \mathcal{T}_{(\rho_1\times \id_{\ov C_1^*})^{-1}\Lambda_1^\mathrm{reg}} \\
&=&
\mathcal{T}_{\Lambda_1^\mathrm{reg}}.
\end{array}
\]
On the other hand,
\[
\begin{array}{rcl}
&&\hskip-8mm \left( (\rho_1\times \id_{\ov C_1^*})_* 
\RPhi_{{\tilde f}_{1,1}}[-1] \1_{\widetilde C_1}[2]\boxtimes \1_{\overline C_1^*}[3] \right)\vert_{\Lambda_1^\mathrm{reg}}[-4] \\
&=&
\RPhi_{\KPair{}{}}[-1] \left( (\rho_1)_* \1_{\wt C_1}[2]\boxtimes \1_{\ov C_1^*}[3] \right)\vert_{\Lambda_1^\mathrm{reg}}[-4] \\
&=&
\RPhi_{\KPair{}{}}[-1] \left( \left( \IC(\1_{C_0}) \oplus \IC(\1_{C_1}) \right) \boxtimes \1_{\ov C_1^*}[3] \right)\vert_{\Lambda_1^\mathrm{reg}}[-4] \\
&=&
\RPhi_{\KPair{}{}}[-1] \left( \IC(\1_{C_0})  \boxtimes \1_{\ov C_1^*}[3] \right)\vert_{\Lambda_1^\mathrm{reg}}[-4]  \\
&& \oplus \ 
\RPhi_{\KPair{}{}}[-1] \left(\IC(\1_{C_1}) \boxtimes \1_{\ov C_1^*} [3]\right)\vert_{\Lambda_1^\mathrm{reg}}[-4] \\
&=& \Evs_{C_1}\IC(\1_{C_0}) \oplus \Evs_{C_1}\IC(\1_{C_1}) \\
&=&  \Evs_{C_1}\IC(\1_{C_1}),
\end{array}
\]
since $\Evs_{C_1}\IC(\1_{C_0}) =0$.
Therefore,
\[
\Evs_{C_1}\IC(\1_{C_1})
=
\mathcal{T}_{\Lambda_1^\mathrm{reg}}.
\]

\subsection{$\IC(\1_{C_2})$}\label{sssec:covC2}
 
 We define 
\begin{align}\label{eqn:covC2}
 \widetilde C_2 
&= \{ ([u],r) \in  \BP^1 \times P_3[x,y] \tq {\rm Grad}(f)_{(u)} = 0 \} 
 \end{align}
and equip it with an action of $\GL_2$ coming from existing action on $P_3[x,y]$ and the action $[u_1:u_2] \mapsto [u_1:u_2] h^{-1}$ on $\mathbb{P}^1$.
As a complex variety, $\widetilde C_2$ is smooth and $\widetilde C_2$ is a rank $2$ vector bundle over $\mathbb{P}^1$.
The condition ${\rm Grad}(f)_{(u)} = 0$ means that that both $r$ and $\Delta_r$ vanish at $u$ and thus that $r$ vanishes to order $2$ at $u$ and thus that $u(x,y)^2 = (u_1y-u_2x)^2$ divides $r(x,y)$.
Using this, together with the isomorphism $P_1[x,y]^\times/\mathbb{G}_m \iso \mathbb{P}^1$ from Section~\ref{sssec:covC1}, we can also write
\[
 \widetilde C_2 = \{ ([u],r) \in \BP^1 \times {\ov C_2} \tq u^2 \vert r \}.
\] 

The projection map $\rho_2: \widetilde C_2 \rightarrow \ov C_2$ is proper, as it is a closed immersion composed with a base change of a proper map, and it realizes the blowup of $\ov C_2$ at $C_1$.
The fiber product of $\widetilde C_2$ over $\ov C_1$ is isomorphic to $\widetilde C_1$.
The map $\rho_2$ is an isomorphism over $C_2 \cup C_1$ while the fibres of $\rho_2$ are $\mathbb{P}^1$ over $C_0$.
It follows that $\rho_2$ is a small map.
By the Decomposition Theorem,
\begin{equation}\label{eqn:pi2}
(\rho_{2})_*(\1_{\widetilde C_2}[3]) = \IC(\1_{C_2}) .
\end{equation}

 An alternate description of this cover is as follows:
Define $\BA^2_\times \times \BA^2\rightarrow \ov C_2$ by 
\[
(u_1,u_2,u_3,u_4 )\mapsto (u_1y-u_2x )^2(u_3y-u_4x).
\]
Let $\mathbb{G}_m$ act on $\BA^2_\times \times \BA^2$ by 
$$a.((u_1,u_2),(u_3,u_4))=((au_1,au_2),(a^{-2}u_3,a^{-2}u_4)).$$
Then 
\[
(\BA^2_\times \times \BA^2)/\mathbb{G}_m \iso \widetilde C_2.
\]
The map $\BA^2_\times \times \BA^2\rightarrow \ov C_2$ is $\mathbb{G}_m$-invariant and induces $\rho_2: \widetilde C_2\rightarrow \ov C_2$.

\subsubsection{$\Ft \IC(\1_{C_2})$ and $\Evs_{C_0}\IC(\1_{C_2})$}\label{sssec:Evs02}

Consider the affine part of $\widetilde C_2$ given by
\[
(u_2,u_3,u_4 )\mapsto ( [1:u_2],(y-u_2x )^2(u_3y-u_4x) ).
\]
For $r=(y-u_2x )^2(u_3y-u_4x) \in \ov C_2$ and $s=s_0y^3-3s_1y^2x-3s_2yx^2-s_3x^3\in C_0^*$ we have 
\[
\begin{array}{rcl}
&& \hskip-8mm \widetilde f_{2,0}(u_2,u_3,u_4,s) \\
&=&\KPair{(y-u_2x )^2(u_3y-u_4x) }{s}\\
&=& u_3s_0+s_1(2u_2u_3+u_4)-s_2(2u_2u_4+u_2^2u_3)+u_2^2u_4s_3\\
&=& u_3(s_0+2s_1u_2-s_2u_2^2) + u_4(s_1-2s_2u_2+s_3u_2^2).
\end{array}
\]
Set $z_1(u_2,u_3,u_4)=u_3$, $z_2(u_2,u_3,u_4)=u_4$, 
$g_1(u_2,u_3,u_4,s)=s_0+2s_1u_2-s_2u_2^2$ and 
$g_2(u_2,u_3,u_4,s)=s_1-2s_2u_2+s_3u_2^2$. 
Then we can write 
\[
\widetilde f_{2,0}(u_2,u_3,u_4,s)=z_1g_1+z_2g_2.
\]
Thus, $\RPhi_{\widetilde f_{2,0}} \1_{{\widetilde C_2}\times {\overline C_0^*}}$ is the skyscraper sheaf supported on the locus of $z_1$, $z_2$, $g_1$ and $G_2$.
To describe this locus, we introduce 
\[
C_1^*\ceq  \{ (s,[v]) \in  {\overline C_1^*}\times \BP^1 \tq s(v)=0,\  \Delta_s(v)=0 \}
\] 
and the obvious cover $\rho_1^* : {\widetilde C_1^*} \to {\overline C_1^*}$. 
Note that this is nothing more than the dual of the cover $\rho_2 : {\widetilde C_2} \to {\overline C_2}$ defined above.
Similarly, if we consider a different affine chart of $\widetilde C_2$, we can draw the same conclusion after extending the definition $z_1$, $z_2$, $g_1$ and $G_2$ accordingly. 
Using \cite{CFMMX}*{Corollary 7.2.6} and the Sebastiani-Thom theorem, it now follows that
\begin{equation}\label{eqn:VtC20}
\RPhi_{\widetilde f_{2,0}}[-1] \left( \1_{\widetilde C_2}\boxtimes \1_{\overline C_0^*}\right)
=
 \1_{C_0}\boxtimes \1_{\widetilde C_1^*}[-4].
\end{equation}
This computes \eqref{eqn:VtC} in this case.

Using \eqref{eqn:VtC20} we now have
\begin{equation}\label{eqn:Ft2}
\begin{array}{rcl}
&&\hskip-8mm \RPhi_{\KPair{}{}}[-1] \left( \IC(\1_{C_2})\boxtimes \1_{\ov C_0^*}[4]\right) \\
&=& \RPhi_{\KPair{}{}}[-1] \left((\rho_{2})_*(\1_{\widetilde C_2}[3]) \boxtimes \1_{\ov C_0^*}[4]\right) \\
&=& (\rho_2\times\id_{\ov C_0^*})_* \RPhi_{\widetilde f_{2,0}}[-1] \left( \1_{\widetilde C_2}[3]\boxtimes \1_{\overline C_0^*}[4]\right) \\
&=& (\rho_2\times\id_{\ov C_0^*})_* \RPhi_{\widetilde f_{2,0}}[-1] \left( \1_{\widetilde C_2}\boxtimes \1_{\overline C_0^*}\right)[7] \\
&=& (\id_{C_0}\times \rho_1^*)_*  \left(
 \1_{C_0}\boxtimes \1_{\widetilde C_1^*}\right)[7-4] \\ 
&=& \1_{C_0}\boxtimes (\rho_1^*)_* \1_{\widetilde C_1^*}[3] \\  
&=& \1_{C_0}\boxtimes \IC(\1_{C_1^*}) \\  
\end{array}
\end{equation}
Therefore,
\[
\Ft \IC(\1_{C_2}) = \IC(\1_{C_1^*}).
\]

Using \eqref{eqn:Ft2} we also find $\Evs_{C_0} \IC(\1_{C_2}$ easily too:
\[
\begin{array}{rcl}
&&\hskip-8mm
\Evs_{C_0}\IC(\1_{C_2}) \\
&=& 
\RPhi_{\KPair{}{}}[-1] \left( \IC(\1_{C_2})\boxtimes \1_{\ov C_0^*}[4]\right)\vert_{\Lambda_0^\mathrm{reg}}[-4] \\
&=& \1_{C_0}\boxtimes \IC(\1_{C_1^*})\vert_{\Lambda_0^\mathrm{reg}}[-4] \\  
&=& 0 ,
\end{array}
\]
since $\Lambda_0^\mathrm{reg} = C_0\times C_0^*$ and $\IC(\1_{C_1^*})\vert_{C_0^*} = 0$.

\subsubsection{$\Evs_{C_1}\IC(\1_{C_2})$}\label{sssec:Evs12}

As above, consider the affine part of $\widetilde C_2$ given by
\[
(u_2,u_3,u_4)\mapsto ( [1:u_2],(y-u_2x )^2(u_3y-u_4x) ).
\]
Set $r = (y-u_2x )^2(u_3y-u_4x)$. 
Now choose local coordinates for ${\overline C_1^*}$ by writing $s =  (y-v_2x )^2(v_3y-v_4x)$.
Then
\[
\begin{array}{rcl}
&&\hskip-8mm 3 {\tilde f}_{2,1}([1:u_2],r,s) \\
&=& \KPair{r}{s} \\
&=& \KPair{(y-u_2x )^2(u_3y-u_4x)}{(y-v_2x )^2(v_3y-v_4x)} \\
&=&  3z_1^2 (u_3v_3+u_4v_4)-2z_1z_2(v_4-v_2v_3)\\
&=& 3g_1\left(z_1-\frac{g_2}{3g_1}z_2 \right)^2-\frac{g_2^2}{3g_1}z_2^2,
\end{array}
\]
where $z_1 = 1+u_2v_2$, $z_2 = u_4-u_2u_3$, $g_1=u_3v_3+u_4v_4$, and $g_2=v_4-v_2v_3$. 
Observe that on $(\rho_2\times\id_{\ov C_1^*})^{-1}\Lambda_1^\mathrm{reg}$, $z_1=0$ and $z_2=0$ while $g_1\ne 0$ and $g_2\ne0$. 
Therefore, by \cite{CFMMX}*{Proposition 7.2.5},
\begin{equation}
\begin{array}{rcl}
\RPhi_{\tilde f_{2,1}}[-1] \left( \1_{\widetilde C_2}\boxtimes \1_{\overline C_1^*}\right)\vert_{(\rho_2\times\id_{\ov C_1^*})^{-1}\Lambda_1^\mathrm{reg}}
=
\1_{(\rho_2\times\id_{\ov C_1^*})^{-1}\Lambda_1^\mathrm{reg}}[-2],
\end{array}
\end{equation}
This completes the calculation of \eqref{eqn:VtC} in this case.

From this we easily find $\Evs_{C_1} \IC(\1_{C_2})$ as follows.
\[
\begin{array}{rcl}
&&\hskip-8mm 
\Evs_{C_1} \IC(\1_{C_2}) \\
&=& \left( \RPhi_{\KPair{}{}}[-1] \IC(\1_{C_2})\boxtimes\1_{\ov C_1^*}[3]\right)\vert_{\Lambda_1^\mathrm{reg}} [-4]\\
&=& \left( \RPhi_{\KPair{}{}}[-1] \left( (\rho_2)_*\1_{\wt C_2}[3] \boxtimes\1_{\ov C_1^*}[3]\right) \right)\vert_{\Lambda_1^\mathrm{reg}}[-4] \\
&=& \left( \RPhi_{\KPair{}{}}[-1] \left( (\rho_2)_*\1_{\wt C_2} \boxtimes\1_{\ov C_1^*}\right) \right)\vert_{\Lambda_1^\mathrm{reg}}[2] \\
&=& (\rho_2\times\id_{\ov C_1^*})_* \left( \RPhi_{\KPair{}{}}[-1] \left( \1_{\wt C_2}\boxtimes\1_{\ov C_1^*}\right) \right)\vert_{\Lambda_1^\mathrm{reg}}[2] \\
&=& (\rho_2\times\id_{\ov C_1^*})_* \left( \RPhi_{\tilde f_{2,1}}[-1] \left( \1_{\widetilde C_2}\boxtimes \1_{\overline C_1^*}\right) \right)\vert_{\Lambda_1^\mathrm{reg}}[2]\\
&=& (\rho_2\times\id_{\ov C_1^*})_* \1_{(\rho_2\times\id_{\ov C_1^*})^{-1}\Lambda_1^\mathrm{reg}}[2-2]\\
&=& \1_{\Lambda_1^\mathrm{reg}},
\end{array}
\]
since $\rho_2$ is an isomorphism over $C_1$.

\subsubsection{$\Evs_{C_2}\IC(\1_{C_2})$}\label{sssec:Evs22}

As above, consider the affine part of $\widetilde C_2$ given by
\[
(u_2,u_3,u_4)\mapsto ( [1:u_2],(y-u_2x )^2(u_3y-u_4x) ).
\]
Set $r =  (y-u_2x )^2(u_3y-u_4x)$. 
Now choose local coordinates for ${\overline C_2^*}$ by writing $s =  (y-v_2x )^3$.
Then
\[
\begin{array}{rcl}
&&\hskip-8mm {\tilde f}_{2,2}([1:u_2],r,s) \\
&=& \KPair{r}{s} \\
&=& \KPair{(y-u_2x )^2(u_3y-u_4x)}{(y-v_2x )^3} \\
&=& \lambda(1+u_2v_2)^2 (u_3+u_4v_2).
\end{array}
\]
Set $z_1 = 1+u_2v_2$ and $g_1 = u_3+u_4v_2$. 
The locus of these two polynomials is 
\[
{\widetilde \Lambda_2} 
\ceq 
\left\{ ([u],r,s,[v])\in {\overline C_2}\times \mathbb{P}^1 \times (\mathbb{P}^1)^*\times {\overline C_2^*} \tq 
\begin{array}{c} u\vert r,\  v\vert s \\
 v\perp u  \end{array} \right\}.
\]
Then 
\begin{equation}\label{eqn:VtC22}
\RPhi_{{\tilde f}_{2,2}} \1_{{\widetilde C_2}\times {\overline C_2}}
=
\mathcal{T}_{{\widetilde \Lambda_2}}
\end{equation}
where $\mathcal{T}_{{\widetilde \Lambda_2}}$ is the local system defined by the double cover $\sqrt{g}$.
Upon restriction of $(\rho_2\times\id_{C_2^*})^{-1}(\Lambda_2^\mathrm{reg})$ we get 
\begin{equation}
\left( \RPhi_{{\tilde f}_{2,2}}[-1] \1_{{\widetilde C_2}\times {\overline C_2}}\right)\vert_{(\rho_2\times\id_{C_2^*})^{-1}(\Lambda_2^\mathrm{reg})}
=
\mathcal{T}_{(\rho_2\times\id_{C_2^*})^{-1}(\Lambda_2^\mathrm{reg})}[-1].
\end{equation}
This computes \eqref{eqn:VtC} in this case.

Now we find $\Evs_{C_2}\IC(\1_{C_2})$ easily:
\[
\begin{array}{rcl}
&&\hskip-8mm 
\Evs_{C_2}\IC(\1_{C_2}) \\
&=& \left(\RPhi_{\KPair{}{}}[-1] \IC(\1_{C_2})\boxtimes\1_{C_2^*}[2]\right)\vert_{\Lambda_2^\mathrm{reg}}[-4] \\
&=& \left(\RPhi_{\KPair{}{}}[-1] (\rho_2)_*\1_{\wt C_2}[3] \boxtimes\1_{C_2^*}[2]\right)\vert_{\Lambda_2^\mathrm{reg}}[-4] \\
&=& (\rho_2 \times \1_{C_2^*})_* \left(\RPhi_{{\tilde f}_{2,2}}[-1] \1_{\wt C_2}[3] \boxtimes\1_{C_2^*}[2]\right)\vert_{\Lambda_2^\mathrm{reg}}[-4] \\
&=& (\rho_2 \times \1_{C_2^*})_* \left(\RPhi_{{\tilde f}_{2,2}}[-1] \1_{\wt C_2} \boxtimes\1_{C_2^*}\right)\vert_{\Lambda_2^\mathrm{reg}}[1] \\
&=& (\rho_2 \times \1_{C_2^*})_* \mathcal{T}_{(\rho_2\times\id_{C_2^*})^{-1}(\Lambda_2^\mathrm{reg})}[1-1] \\
&=& \mathcal{T}_{\Lambda_2^\mathrm{reg}},
\end{array}
\]
since $\rho_2$ is an isomorphism over $C_2$.

\subsection{$\IC(\1_{C_3})$}\label{sssec:C3}
 
Because the variety $P_3[x,y]$ is smooth and $C_3$ is dense in $P_3[x,y]$, we have 
\[ 
\IC(\1_{C_3}) = \1_{P_3[x,y]}[4] .
\]
 
\subsubsection{$\Ev_{C_0}\IC(\1_{C_3})$ }\label{sssec:Evs03}

The only singularity of the killing form on $P_3[x,y]\times \overline{C_0^\ast}$ is at the origin so
\begin{equation}\label{VtC30}
\RPhi_{\KPair{}{}} \left( \1_{\overline C_3}\boxtimes \1_{\overline C_0^*}\right)
=
\1_{\overline C_0}\boxtimes \1_{\overline C_3^*}
=
\1_{C_0}\boxtimes \1_{C_3^*}
\end{equation}
This computes \eqref{eqn:VtC} in this case. 

Note that \eqref{VtC30} may be written in the form
\begin{equation}\label{eqn:Ft3}
\RPhi_{\KPair{}{}} \left( \IC(\1_{C_3}) \boxtimes \1_{\overline C_0^*}\right)
=
\1_{\overline C_0}\boxtimes \IC(\1_{C_3^*})
\end{equation}
from which it follows that
\begin{equation}\label{Ft3}
\Ft \IC(\1_{C_3}) = \IC(\1_{C_3^*}).
\end{equation}

Now, consider the restriction of \eqref{VtC30} to $\Lambda_{0} = C_0\times C_0^*$.
Since
\[
\Lambda_{0}\cap (C_0\times C_3^*) = \emptyset,
\]
it follows that
\begin{equation}\label{VC03}
\left( \RPhi_{\KPair{}{}} \left( \1_{\overline C_3}\boxtimes \1_{\overline C_0^*}\right) \right)\vert_{\Lambda_0}
=
\left( \1_{C_0}\boxtimes \1_{C_3^*}\right)\vert_{\Lambda_0}
=
0
\end{equation}
and therefore that
\[ 
\Ev_{C_0}\IC(\1_{C_3}) 
= 
0 .
\]

\subsubsection{$\Ev_{C_1}\IC(\1_{C_3})$}\label{sssec:Evs13}

Again, observe that the only singularity of the killing form on $P_3[x,y]\times \overline{C_1^\ast}$ is at the origin so
\begin{equation}\label{VtC31}
\RPhi_{\KPair{}{}} \left( \1_{\overline C_3}\boxtimes \1_{\overline C_1^*}\right)
=
\1_{\overline C_0}\boxtimes \1_{\overline C_3^*}
=
\1_{C_0}\boxtimes \1_{C_3^*}
\end{equation}
This computes \eqref{eqn:VtC} in this case. 

Now, consider the restriction of \eqref{VtC31} to $\Lambda_{1} \subset C_1\times C_1^*$.  Since
\[
(C_1\times C_1^*) \cap (C_0\times C_3^*) = \emptyset,
\]
it follows that
\begin{equation}\label{VC31}
\left( \RPhi_{\KPair{}{}} \left( \1_{\overline C_3}\boxtimes \1_{\overline C_1^*}\right) \right)\vert_{\Lambda_1}
=
\left( \1_{C_0}\boxtimes \1_{C_3^*}\right)\vert_{\Lambda_1}
=
0
\end{equation}
and therefore that
\[ 
\Ev_{C_1}\IC(\1_{C_3}) 
= 
0 .
\]

\subsubsection{$\Ev_{C_2}\IC(\1_{C_3})$}\label{sssec:Evs23}

Arguing as in the two cases above,
\begin{equation}\label{VtC32}
\RPhi_{\KPair{}{}} \left( \1_{\overline C_3}\boxtimes \1_{\overline C_2^*}\right)
=
\1_{\overline C_0}\boxtimes \1_{\overline C_3^*}
=
\1_{C_0}\boxtimes \1_{C_3^*}.
\end{equation}
This computes \eqref{eqn:VtC} in this case.

The restriction of \eqref{VtC32} to $\Lambda_{2} \subset C_2\times C_2^*$ is
\begin{equation}\label{VC32}
\left( \RPhi_{\KPair{}{}} \left( \1_{\overline C_3}\boxtimes \1_{\overline C_2^*}\right) \right)\vert_{\Lambda_2}
=
\left( \1_{C_0}\boxtimes \1_{C_3^*}\right)\vert_{\Lambda_2}
=
0
\end{equation}
and therefore
\[ 
\Ev_{C_2}\IC(\1_{C_3}) 
= 
0 .
\]

\subsubsection{$\Ev_{C_3}\IC(\1_{C_3})$}\label{sssec:Evs33}

Since $\KPair{}{}$ vanishes on  ${\overline C_3}\times {\overline C_3^*}$ we have
\begin{equation}\label{VtC33}
\RPhi_{\KPair{}{}} \left( \1_{\overline C_3}\boxtimes \1_{\overline C_3^*}\right)
=
\RPhi_{0} \left( \1_{\overline C_3}\boxtimes \1_{\overline C_3^*}\right)
=
\1_{\overline C_3}\boxtimes \1_{\overline C_3^*}.
\end{equation}
This computes \eqref{eqn:VtC} in this case. 

This time, the restriction of \eqref{VtC33} to $\Lambda_{3} = C_3\times C_3^*$ is
\begin{equation}\label{VC33}
\left( \RPhi_{\KPair{}{}} \left( \1_{\overline C_3}\boxtimes \1_{\overline C_3^*}\right) \right)\vert_{\Lambda_3}
=
\left( \1_{\overline C_3}\boxtimes \1_{\overline C_3^*}\right)\vert_{\Lambda_3}
=
\1_{\overline C_3}\boxtimes \1_{\overline C_3^*}.
\end{equation}
Therefore
\[ 
\Ev_{C_3}\IC(\1_{C_3}) 
= 
\1_{\Lambda_3}.
\]

\subsection{$\IC(\mathcal{R}_{C_3})$}\label{sssec:covC3}
 
Consider 
\begin{equation}\label{eqn:sssec:covC3}
\widetilde C_3 = \{ ([u],r) \in \mathbb{P}^1\times P_3[x,y]  \tq u\vert r \}
\end{equation}
and equip it with an action of $\GL_2$ coming from existing action on $P_3[x,y]$ and the right action of $\GL_2$ on $\mathbb{P}^1$ as in Sections~\ref{sssec:covC1} and \ref{sssec:covC2}.
Then $\widetilde C_3$ is smooth and is a rank $3$ vector bundle over $\mathbb{P}^1$.
 
For use below, we give an alternate description of this cover.
Consider 
\begin{equation}\label{eq: p3}
\begin{array}{rcl}
\BA^2_\times \times \BA^3 &\rightarrow& \ov C_3 \\
((u_1,u_2),(u_3,u_4,u_5)) & \rightarrow& (u_1y-u_2x)(u_3y^2 + u_4 yx + u_5 x^2).
\end{array}
\end{equation}
Let $\mathbb{G}_m$ act on $\BA^2_\times \times \BA^3$ by 
$$a.((u_1,u_2),(u_3,u_4,u_5))\mapsto ((au_1,au_2),(a^{-1}u_3,a^{-1}u_4,a^{-1}u_5)).$$
Then 
\[
\begin{array}{rcl}
(\BA^2_\times \times \BA^3)/\mathbb{G}_m &\to& {\widetilde C_3} \\
\left[ (u_1,u_2),(u_3,u_4,u_5) \right] &\mapsto& \left( [u_1:u_2], (u_1y-u_2x)(u_3y^2 + u_4 yx + u_5 x^2) \right)
\end{array}
\]
is an isomorphism and $\rho_3: \widetilde C_3\rightarrow \ov C_3$ is the map induced by $\BA^2_\times \times \BA^3\rightarrow \ov C_3$.

The map $\rho_3: \widetilde C_3\rightarrow \ov C_3$ is proper, as it is a closed immersion composed with a base change of a proper map.
Recall that $\ov C_3 = P_3[x,y]$.
The map $\rho_3$ is also non-Galois, finite, a $3:1$ cover over $C_3 \cup C_2 \cup C_1$, which becomes $2:1$ over $C_2$ and $1:1$ over $C_1$.
The fibre of $\rho_3$ is isomorphic to $\mathbb{P}^1$ over $C_0$.
 We note that the fibre product of $\widetilde C_3$ over $\ov C_2$ has two irreducible components intersecting over $\ov C_1$.
 One of the two components is precisely $\widetilde C_2$ defined above, the other is isomorphic to it.
 The fiber product of $\widetilde C_3$ over $\ov C_1$ is isomorphic to $\widetilde C_1$.
The map $\rho_3$ is small map and
 \begin{equation}\label{eqn:pi3}
 (\rho_3)_*(\1_{\widetilde C_3}[4]) = \IC(\1_{C_3}) \oplus \IC(\mathcal{R}_{C_3}) 
 \end{equation}
 where $\mathcal{R}_{C_3}$ is the rank $2$ sheaf on $C_3$ associated to the two dimensional irreducible representation $\varrho$ of $S_3$.

\subsubsection{$\Ft \IC(\mathcal{R}_{C_3})$ and $\Ev_{C_0}\IC(\mathcal{R}_{C_3})$}\label{sssec:Evs03R}

From Section~\ref{sssec:covC3} recall the isomorphism
\[
\begin{array}{rcl}
(\BA^2_\times \times \BA^3)/\mathbb{G}_m &\to& {\widetilde C_3} \\
\left[ (u_1,u_2),(u_3,u_4,u_5) \right] &\mapsto& \left( [u_1:u_2] , (u_1y-u_2x)(u_3y^2 + u_4 y x  + u_5 x^2)\right)
\end{array}
\]
We choose an affine part of ${\widetilde C_3}$ given locally by
\[
\begin{array}{rcl}
\mathbb{A}^4 &\hookrightarrow& {\widetilde C_3} \\
(u_2,u_3,u_4,u_5) &\mapsto& \left( [1:u_2] , (y-u_2x)(u_3y^2 + u_4 yx + u_5 x^2)\right).
\end{array}
\]
Then 
\[
{\tilde f}_{3,0} : {\widetilde C_3}\times {\overline C_0^*} \to \mathbb{A}^1
\]
is given locally by
\[
\begin{array}{rcl}
&&\hskip-8mm {\tilde f}_{3,0}\vert_{\BA^4 \times C_0^*} (u_2,u_3,u_4,u_5, s(x,y))\\
&=& \KPair{ (y-u_2x)(u_3y^2 + u_4 yx + u_5 x^2)}{s_0 y^3 - 3s_1y^2 x - 3s_2y x^2 - s_3 x^3} \\
&=& u_3 s_0 - (u_4-u_2u_3) s_1 - (u_5-u_2u_4) s_2 + u_2u_5 s_3\\
&=& u_3 (s_0 +u_2s_1) + u_4(u_2s_2-s_1)+u_5(u_2s_3-s_2).
\end{array} 
\]
Now consider the vanishing cycles appearing in \eqref{eqn:VtC}. 
Set 
\[ \begin{array}{lcl}
z_1(u_2,u_3,u_4,u_5, s) = u_3, &\quad& g_1(u_2,u_3,u_4,u_5, s) = s_0 +u_2s_1,\\
z_2(u_2,u_3,u_4,u_5, s) = u_4, && g_2(u_2,u_3,u_4,u_5, s) = u_2s_2-s_1,\\
z_3(u_2,u_3,u_4,u_5, s) = u_5,&& g_3(u_2,u_3,u_4,u_5, s) = u_2s_3-s_2;
\end{array}
\]
Then, 
\[
\RPhi_{{\tilde f}_{3,0}} \1_{{\widetilde C_3}\times {\overline C_0^*}}\\
= 
\RPhi_{z_1g_1+z_2g_2+z_3g_3} \1_{{\widetilde C_3}\times {\overline C_0^*}}.
\]
Thus, \eqref{eqn:VtC} is the skycraper sheaf supported by the solution to the equations $z_1=z_2=z_3=0$ and $g_1=g_2=g_3=0$, which is to say, $u_3=u_4=u_5=0$ and $s_0 +u_2s_1=u_2s_2-s_1= u_2s_3-s_2 =0$; this is precisely 
$
{\overline C_0}\times {\widetilde C_2^*} 
$
where 
\[
{\widetilde C_2^*}\ceq \{ (s,[v] ) \in  {\overline C_2^*}\times \BP^1 \tq g(v) =0 \}. 
\]
Define $\rho_2^* : {\widetilde C_2^*} \to {\overline C_2^*}$ by $\rho_2^*([v],s) = s$. 
Then
\begin{equation}\label{eqn:VtC30}
\RPhi_{{\tilde f}_{3,0}}[-1] \left( \1_{\widetilde C_3}\boxtimes \1_{\overline C_0^*}\right)
= 
\1_{\overline C_0}\boxtimes \1_{\widetilde C_2^*}[-6] .
\end{equation}
This computes \eqref{eqn:VtC} in this case.

Now, 
\[
\begin{array}{rcl}
&& \hskip-8mm 
(\rho_3\times \id_{\overline C_0^*})_* \RPhi_{{\tilde f}_{3,0}}[-1] \left( \1_{\widetilde C_3}[4]\boxtimes \1_{\overline C_0^*}[4]\right) \\
&=& 
(\id_{C_0} \times \rho_2^*)_* \left( \1_{\overline C_0}[4]\boxtimes \1_{\widetilde C_2^*}[4] \right) [-6]\\
&=&
\1_{C_0}\boxtimes (\rho_2^*)_*\1_{\widetilde C_2^*}[2]\\
&=&
\1_{C_0}\boxtimes \left( \IC(\1_{C_2^*})\oplus \IC(\1_{C_3^*}) \right) \\
&=&
\left( \1_{C_0}\boxtimes \IC(\1_{C_2^*})\right) \oplus \left( \1_{C_0}\boxtimes \IC(\1_{C_3^*}) \right).
\end{array} 
\]
On the other hand,
\[
\begin{array}{rcl}
&& \hskip-8mm (\rho_3\times \id_{\overline C_0^*})_*\RPhi_{{\tilde f}_{3,0}}[-1] \left( \1_{\widetilde C_3}[4]\boxtimes \1_{\overline C_0^*}[4]\right) \\
&=& \RPhi_{\KPair{}{}} \left( (\rho_3)_* \1_{\widetilde C_3}[4] \boxtimes \1_{\overline C_0^*}[4]\right) \\
&=& 
\RPhi_{\KPair{}{}}[-1] \left( \left( \IC(\1_{C_3}) \oplus \IC(\mathcal{R}_{C_3}) \right) \boxtimes \1_{\overline C_0^*}[4]\right) \\
&=& 
\RPhi_{\KPair{}{}}[-1] \left(  \IC(\1_{C_3})\boxtimes \1_{\overline C_0^*}[4]\right) \oplus \RPhi_{\KPair{}{}}[-1]  \left( \IC(\mathcal{R}_{C_3})  \boxtimes \1_{\overline C_0^*}[4]\right)  \\
&=& 
\left( \1_{C_0} \boxtimes \IC(\1_{C_3^*}) \right) \oplus \RPhi_{\KPair{}{}} [-1] \left( \IC(\mathcal{R}_{C_3})  \boxtimes \1_{\overline C_0^*}[4]\right)  \\
\end{array}
\]
Therefore,
\begin{equation}\label{eqn:Ft3'}
\RPhi_{\KPair{}{}}[-1] \left(\IC(\mathcal{R}_{C_3})\otimes\1_{\ov C_0^*}[4]\right)
=
\1_{C_0}\boxtimes  \IC(\1_{C_2^*}).
\end{equation}
If follows that
\[
\Ft \IC(\mathcal{R}_{C_3}) = \IC(\1_{C_2^*}).
\]

Returning to \eqref{eqn:VtC30}, since 
\[
(C_0\times {\widetilde C_2^*}) \cap \Lambda_{C_0}^{\rho_3} = \emptyset
\]
it follows now that
\begin{equation}
\left( \RPhi_{{\tilde f}_{3,0}} \1_{{\widetilde C_3}\times {\overline C_0^*}}\right)\vert_{\Lambda_{0}^{\rho_3}}
= 
0
\end{equation}
so
\[
\Evs_{C_0} \IC(\mathcal{R}_{C_3}) =0.
\]

\subsubsection{$\Evs_{C_1}\IC(\mathcal{R}_{C_3})$}\label{sssec:Evs13R}

As above, for $([u],r)\in {\widetilde C_3}$ 
we take  $u=[1:u_2]$ and $r=(y-u_2x)(u_3y^2 +u_4yx+u_5x^2)$.
For $s\in {\overline C_1^*}$ we take $s = (y-v_2x )^2(v_3y-v_4x)$.
Then
\[
\begin{array}{cl}
&\hskip-7mm {\tilde f}_{3,1}([u],r,s) \\
=&\hskip-5pt \KPair{r}{s} \\
=&\hskip-5pt \KPair{(y-u_2x )(u_3y^2 +u_4yx+u_5x^2)}{(y-v_2x )^2(v_3y-v_4x)} \\
=&\hskip-5pt\frac{1}{3}(3u_3 v_3+(u_2u_3-u_4)(2v_2v_3+v_4)+(u_5-u_2u_4)(v_2^2v_3+2v_2v_4)+3u_2u_5v_2^2v_4)\\
=&\hskip-5pt z_1 g_1 + z_2 g_2 + z_3 g_3 
\end{array}
\]
where 
$z_1=u_2v_2+1$,
$z_2=u_4+2u_2u_3$ and
$z_3=2u_5+u_2u_4$, while
\begin{align*}
    g_1&=\frac{1}{2}(2u_3v_3+2u_5v_2v_4-u_4(v_2v_3+v_4)),\\
    g_2&=-\frac{1}{6}(v_2v_3-v_4),\\
    g_3&=\frac{1}{6}v_2(v_2v_3-v_4).
\end{align*}
In fact $g_1$ is in the ideal generated by $z_1,z_2,z_3$, as
$$2g_1=z_1(2u_3v_3-u_4v_4)-z_2v_2v_3+z_3v_2v_4.$$ 
It follows that
\begin{equation}\label{eqn:VtC31}
\RPhi_{{\tilde f}_{3,1}}[-1] \left( \1_{\widetilde C_3}\boxtimes \1_{\overline C_1^*}\right)
= 
\1_{\widetilde{\Lambda_1}}[-6]
\end{equation}
where
\[
\widetilde{\Lambda_1} = 
\left\{ ([u],r,s,[v])\in \BP^1\times {\overline C_1}\times {\overline C_2^*} \times \BP^1 \tq 
\begin{array}{rcl}
u\vert r,\ v\vert s \\
v\perp u
\end{array}
\right\}
\]
Now compare $\widetilde{\Lambda_1}$ with
\[
(\rho_3\times\id_{C_1^*})^{-1} \Lambda_1^\mathrm{reg}
\subseteq 
\left\{
([u],r,s, [v])\in \BP^1\times {\overline C_1}\times {\overline C_1^*} \times \BP^1 \tq 
\begin{array}{rcl}
u\vert r,\ v\vert s \\
v\perp u
\end{array}
\right\}.
\]
Since ${\overline C_2^*} \cap C_1^* =\emptyset$, it follows that
\[
\left( \RPhi_{{\tilde f}_{3,1}}[-1] \left( \1_{\widetilde C_3}\boxtimes \1_{\overline C_1^*}\right) \right)\vert_{(\rho_3\times\id_{C_1^*})^{-1} \Lambda_1^\mathrm{reg}}
=
0.
\]
It follows that
\[
\Evs_{C_1} \IC(\mathcal{R}_{C_3}) =0.
\]

\subsubsection{$\Evs_{C_2}\IC(\mathcal{R}_{C_3})$}\label{sssec:Evs23R}

As above, for $(r,[u])\in {\widetilde C_3}$ 
we take $r=(y-u_2x)(u_3y^2 +u_4yx+u_5x^2)$ and $u=[1:u_2]$.
For $s\in {\overline C_2^*}$ we take $s = (y-v_2x )^3$.
Then
\[
\begin{array}{rcl}
&&\hskip-8mm {\tilde f}_{3,2}(r,[u],s) \\
&=& \KPair{(y-u_2x )(u_3y^2 +u_4yx+u_5x^2)}{(y-v_2x )^3} \\
&=&  (1+u_2v_2)(u_3 -u_4v_2+u_5v_2^2)
\end{array}
\]
Set
\begin{align*}
    z_1&=u_2v_2+1,\\
    z_2&=u_3 -u_4v_2+u_5v_2^2.
\end{align*}
Then
\begin{equation}\label{eqn:VtC32}
\RPhi_{{\tilde f}_{3,2}}[-1] \left( \1_{\widetilde C_3}\boxtimes \1_{\overline C_2^*}\right)
= 
\1_{\widetilde{\Lambda_2}}[-2]
\end{equation}
where
\[
\widetilde{\Lambda_2} = \left\{ (r,[u],[v],s)\in {\overline C_2}\times \mathbb{P}^1\times (\mathbb{P}^1)^*\times {\overline C_2^*} \tq 
\begin{array}{c} u^2\vert r,\ v\vert s \\ v\perp u \end{array} \right\}.
\]
Note that
\[
\widetilde{\Lambda_2} = (\rho_3\times\id_{\ov C_3^*})^{-1}({\ov \Lambda_2})
\]
Upon restriction to generic elements, we have
\begin{equation}\label{eqn:VC32}
\left(\RPhi_{{\tilde f}_{3,2}}[-1] \left( \1_{\widetilde C_3}\boxtimes \1_{\overline C_2^*}\right)\right)\vert_{(\rho_3\times\id_{C_3^*})^{-1} \Lambda_2^\mathrm{reg}}
= 
\1_{(\rho_3\times\id_{C_3^*})^{-1} \Lambda_2^\mathrm{reg}}.
\end{equation}
Since the cover $\wt \Lambda_2\to \Lambda_2$ is one-to-one over $\Lambda_2^{\mathrm{reg}}$, it follows that
\[
\Evs_{C_2}(\IC(\mathcal{R}_{C_3}))
=
\1_{\Lambda_2^\mathrm{reg}}.
\]

\subsubsection{$\Evs_{C_3}\IC(\mathcal{R}_{C_3})$}\label{sssec:Evs33R}

Since ${\tilde f}_{3,3} : {\widetilde C_3} \times {\overline C_3^*} \to S$ is trivial,
\begin{equation}\label{eqn:VtC33}
\begin{array}{rcl}
&&\hskip-8mm \RPhi_{{\tilde f}_{3,3}}[-1] \left( \1_{\widetilde C_3}\boxtimes \1_{\overline C_3^*}\right) \\
&=& \RPhi_{0}[-1] \left( \1_{\widetilde C_3}\boxtimes \1_{\overline C_3^*}\right) \\
&=& \1_{\widetilde C_3}\boxtimes \1_{\overline C_3^*}
\end{array}
\end{equation}
From \eqref{eqn:VtC33} we have
\[
\begin{array}{rcl}
&&\hskip-8mm (\rho_3\times \id)_* \RPhi_{{\tilde f}_{3,3}}[-1] \left( \1_{\widetilde C_3}[4]\boxtimes \1_{\overline C_3^*}[0]\right)[-4] \\
&=& (\rho_3\times \id)_* \left( \1_{\widetilde C_3}[4]\boxtimes \1_{\overline C_3^*}\right)[-4] \\
&=& \left( (\rho_3)_*\1_{\widetilde C_3}[4]\right) \boxtimes \1_{\overline C_3^*}[-4] \\
&=& \left( \IC(\1_{C_3}) \oplus \IC(\mathcal{R}_{C_3}) \right) \boxtimes \1_{\overline C_3^*} [-4]\\
&=& \left( \IC(\1_{C_3})\boxtimes \1_{\overline C_3^*} \right)[-4] \oplus \left( \IC(\mathcal{R}_{C_3}) \boxtimes \1_{\overline C_3^*} \right)[-4].  
\end{array}
\]
This determines \eqref{eqn:pVtC} in this case.

Finally, upon restriction to generic elements, it now follows that
\[
\begin{array}{rcl}
&& \hskip-8mm \Evs_{C_3} (\rho_3)_* \1_{\widetilde C_3}[4] \\
&=& (\rho_3\times \id)_* \RPhi_{{\tilde f}_{3,3}}[-1] \left( \1_{\widetilde C_3}[4]\boxtimes \1_{\overline C_3^*}[0]\right)\vert_{\Lambda_3^\mathrm{reg}}[-4] \\
&=& \left( \left( \IC(\1_{C_3})\boxtimes \1_{\overline C_3^*} \right)[-4] \oplus \left( \IC(\mathcal{R}_{C_3}) \boxtimes \1_{\overline C_3^*} \right)[-4]\right)\vert_{\Lambda_3^\mathrm{reg}}\\
&=& \1_{\Lambda_3^\mathrm{reg}} \oplus \mathcal{R}_{\Lambda_3^\mathrm{reg}}.
\end{array}
\]
On the other hand, using \eqref{eqn:pi3},
\[
\begin{array}{rcl}
&& \hskip-8mm \Evs_{C_3} (\rho_3)_* \1_{\widetilde C_3}[4] \\
&=& \Evs_{C_3} \left( \IC(\1_{C_3}) \oplus \IC(\mathcal{R}_{C_3}) \right)\\
&=& \left( \Evs_{C_3} \IC(\1_{C_3}) \right) \oplus \left( \Evs_{C_3} \IC(\mathcal{R}_{C_3}) \right) \\
&=& \1_{\Lambda_3^\mathrm{reg}} \oplus \Evs_{C_3} \IC(\mathcal{R}_{C_3}).
\end{array}
\]
Therefore,
\[ 
\Evs_{C_3}\IC(\mathcal{R}_{C_3})   =  \mathcal{R}_{\Lambda_3^\mathrm{reg}}.
\]

 \subsection{$\IC(\mathcal{E}_{C_3})$}\label{sssec:covC3''}

Consider the cover
\[
{\widetilde C_3''} 
\ceq 
\{ ([u],[u'],[u''],r) \in \mathbb{P}^1\times \mathbb{P}^1\times \mathbb{P}^1\times P_3[x,y]  \tq u u' u'' \vert r \}.
\] 
We are thus defining $\widetilde C_3''$ to be the pull back of the bundle $\mathcal{O}(-1)$ over $\mathbb{P}(\Sym^3)$ through the map $\mathbb{P}(\Sym^1)^3 \rightarrow \mathbb{P}(\Sym^3)$.
It follows that $\widetilde C_3''$ is smooth, being a vector bundle.
Define $\rho''_3: \widetilde C_3''\rightarrow \ov C_3$ by $([u],[u'],[u''],r) \mapsto r$. This map is proper, as it is a closed immersion composed with a base change of a proper map.
 The map $\rho''_3$ is a finite Galois cover, with Galois group $S_3$, over $C_3 \cup C_2 \cup C_1$.
 The map degenerates to $3:1$ over $C_2$ and $1:1$ over $C_1$.
 The fiber of the map is isomorphic to $(\mathbb{P}^1)^3$ over $C_0$.
Note that $\rho_3''$ is not semi-small.
%
%
We note that the fibre product of $\widetilde C_3''$  over $\ov C_2$ has three irreducible components intersecting over $\ov C_1$. 
 For each of these irreducible components the map to $\ov C_2$ factors through $\widetilde C_2$ above.
 \begin{equation}\label{eqn:pi''3}
 \begin{array}{rcl}
&&\hskip-8mm 
(\rho''_3)_*(\1_{\widetilde C_3''}[4]) \\
&=& \IC(\1_{C_3})\oplus 2\IC(\mathcal{R}_{C_3})\oplus \IC(\mathcal{E}_{C_3})  \\
&&\ \oplus\ 3 \IC(\1_{C_0}) \oplus\  \IC(\1_{C_0})[2] \oplus \IC(\1_{C_0})[-2] 
\end{array}
 \end{equation}

We parameterize $\wt C_3''$ as follows. Define $\mathbb{A}\times \mathbb{A}^2_{\times}\times \mathbb{A}^2_{\times}\times \mathbb{A}^2_{\times}\to \overline{C}_3$ by 
$$(\lambda,(u_1,u_2),(u_3,u_4),(u_5,u_6))\mapsto \lambda (u_1y-u_2x)(u_3y-u_4x)(u_5y-u_6x).$$
Let $\mathbb{G}_m^3$ acts on $\mathbb{A}\times \mathbb{A}^2_{\times}\times \mathbb{A}^2_{\times}\times \mathbb{A}^2_{\times} $ by 
\begin{align*}&(a,b,c).(\lambda,(u_1,u_2),(u_3,u_4),(u_5,u_6))\\
&\quad =(a^{-1}b^{-1}c^{-1}\lambda,(au_1,au_2),(bu_3,bu_4),(cu_5,cu_6)).
\end{align*}
Then the map $\mathbb{A}\times \mathbb{A}^2_{\times}\times \mathbb{A}^2_{\times}\times \mathbb{A}^2_{\times}\to \overline{C_3} $ is $\mathbb{G}_m^3$-invariant and defines an isomorphism
$$(\mathbb{A}\times \mathbb{A}^2_{\times}\times \mathbb{A}^2_{\times}\times \mathbb{A}^2_{\times})/(\mathbb{G}_m^3)\to \widetilde{C_3} $$
$$[\lambda,(u_1,u_2),(u_3,u_4),(u_5,u_6)]\mapsto ([u_1,u_2],[u_3,u_4],[u_5,u_6], r(x,y)),$$
with 
$$r(x,y)=\lambda (u_1y-u_2x)(u_3y-u_4x)(u_5y-u_6x).$$

We will also make use another cover of ${\ov C_3}$, defined here.
Set
\begin{align*}
 E &= \wpair{([u_1:u_2],[u_3:u_4],r)\in \mathbb{P}^1\times \mathbb{P}^1\times \ov C_3: \begin{pmatrix} u_1 & u_2 \end{pmatrix} {\rm Hess}(r)_{(u_3,u_4)}  = 0 } \\
  &=\wpair{([u],[u'],r)\in \mathbb{P}^1\times \mathbb{P}^1\times \ov C_3: g_1(u,u',r)=g_2(u,u',r)=0},
\end{align*}
where \begin{align*}
g_1([u],[u'],r)&=u_1u_3r_0+(u_2u_3+u_1u_4)r_1-u_2u_4r_2,\\
g_2([u],[u'],r)&=u_1u_3r_1-(u_2u_3+u_1u_4)r_2+u_2u_4r_3,
\end{align*}
for $[u]=[u_1,u_2],[u']=[u_3,u_4],r=(r_0,r_1,r_2,r_3)$. 
Let $\rho_E:E\rightarrow \ov C_3$ is the projection map. 
By considering the first model we can recognize that
 $\rho_E$ is $2:1$ over $C_1$ and $1:1$ over $C_2$ as $[u],[u']$ must be the two lines dividing $\Delta_r$. 
We also have $\rho_E$ is $(\mathbb{P}^1\oplus \mathbb{P}^1):1$ over $C_1$ as at least one of $[u],[u']$ must be the line dividing $r$, note these copies of $\mathbb{P}^1$ intersect at that point.
Finally  $\rho_E$ is $\mathbb{P}^1\times \mathbb{P}^1:1$ over $C_0$. It follows that
\begin{equation}\label{eqn:pushforwardofrhoE}
(\rho_E)_* \1_{E}[4] = \IC(\1_{C_3}) \oplus \IC(\mathcal{E}_{C_3}) \oplus 2\IC(\1_{C_1}) \oplus \IC(\1_{C_0}). 
\end{equation}

\subsubsection{$\Ft \IC(\mathcal{E}_{C_3})$ and $\Evs_{C_0}\IC(\mathcal{E}_{C_3})$}\label{sssec:Evs03E}

Take a point $[\lambda,(u_1,u_2),(u_3,u_4),(u_5,u_6)]\in \wt C_3 $ and its image $r(x,y)=\lambda (u_1y-u_2x)(u_3y-u_4x)(u_5y-u_6x)$ in $\ov C_3$. For $s(x,y)=s_0y^3-3s_1y^2x-3s_2yx^2-s_3x^3\in \ov{C_0^*}$, we have 
\[
\begin{array}{rcl}
&&\hskip-8mm {\tilde f}_{3'',0}([\lambda,(u_1,u_2),(u_3,u_4),(u_5,u_6),s) \\
&=& \lambda( u_1u_3u_5s_0+(u_1u_3u_6+u_2u_3u_5+u_1u_4u_5)s_1\\
&& -(u_2u_4u_5+u_2u_3u_6+u_1u_4u_6)s_2+u_2u_4u_6s_3 ).
\end{array}
\]

Set $z=\lambda$ and $g([\lambda,(u_1,u_2),(u_3,u_4),(u_5,u_6)],s )=u_5g_1+u_6g_1 $ with
 \begin{align*}
g_1([\lambda, (u_1,u_2),(u_3,u_4),(u_5,u_6)],s)&=u_1u_3s_0+(u_2u_3+u_1u_4)s_1-u_2u_4s_2,\\
g_2([\lambda, (u_1,u_2),(u_3,u_4),(u_5,u_6)],s)&=u_1u_3s_1-(u_2u_3+u_1u_4)s_2+u_2u_4s_3.
\end{align*}
Note that $ {\tilde f}_{3'',0}([\lambda,(u_1,u_2),(u_3,u_4),(u_5,u_6),s)=zg.$
Then the singular locus of  ${\tilde f}_{3'',0}$ is $C_0 \times X$
where 
\[
X
= \{([u],[u'],[u''],s)\in \ov{C_0^*}\times (\BP^1)^3  \tq g(u,u',u'',s)=0 \}
\]
and
\begin{equation}\label{eqn:VtC3''3}
\RPhi_{{\tilde f}_{3'',0}}[-1] \left( \1_{\wt C_3''} \boxtimes \1_{\ov C_0^*} \right) \\
= 
\1_{C_0\times X}[-2].
\end{equation}
so
\begin{equation}
\RPhi_{{\tilde f}_{3'',0}}[-1] \left( \1_{\wt C_3''}[4] \boxtimes \1_{\ov C_0^*}[4] \right) \\
= 
\1_{C_0\times X}[6].
\end{equation}

Now
\[
\begin{array}{rcl}
&& \hskip-8mm
(\rho_3''\times \id)_* \RPhi_{{\tilde f}_{3'',0}}[-1] \left( \1_{\wt C_3''}[4] \boxtimes \1_{\ov C_0^*}[4] \right)\\
&=&
(\rho_3''\times \id)_* \1_{C_0 \times X}[6] \\
&=&
\1_{C_0}\boxtimes (\rho'_3)_* \1_X[6], 
\end{array}
\]
for $\rho'_3: X \to {\ov C_0^*}$ defined by $\rho'_3([u],[u'],[u''],s) = s$. 
Consider the stratification 
\[
\begin{tikzcd}
U \arrow{r} \arrow{d}{\iso} & X \arrow{d}{\rho'_3} &\arrow{l} F \arrow{d}{\BP^1}\\
V \arrow{r} \arrow{dr}[swap]{\BP^1\times \BP^1} &Y \arrow{d}{\pi} &\arrow{l} E \arrow{dl}{\rho_E} \\
& {\ov C_0^*} 
\end{tikzcd}
\]
with $Y=\BP^1\times\BP^1\times {\ov C_0^*}$ and
\[
F = 
\left\{ 
([u],[u'],[u''],s)\in X 
\tq 
\begin{array}{c}
g_1(u,u',u'',s)=0  \mathrm{ \ and\ } \\
g_2(u,u',u'',s) =0
\end{array}
\right\}
\]
and
\[
U = 
\left\{ 
(s, [u],[u'],[u''])\in X 
\tq
\begin{array}{c}
g_1(s,u,u',u'')\ne 0 \mathrm{ \ or\ }\\ 
g_2(s,u,u',u'')\ne 0
\end{array}
\right\}
\]
If $([u],[u'],[u''],s)\in U$ then we can solve $[u'']=[u_5:u_6]$ (as a point in $\mathbb{P}^1$) uniquely from $g=0$ and $g_1\ne 0$ or $g_2\ne 0$.
Abusing notation slightly, set 
\begin{align*}
g_1((u_1,u_2),(u_3,u_4),s)&=u_1u_3s_0+(u_2u_3+u_1u_4)s_1-u_2u_4s_2,\\
g_2((u_1,u_2),(u_3,u_4),s)&=u_1u_3s_1-(u_2u_3+u_1u_4)s_2+u_2u_4s_3.
\end{align*}
and define
\[
E \ceq
\left\{ 
([u],[u'],s)\in \BP^1\times\BP^1 \times {\ov C_0^*} 
\tq 
\begin{array}{c}
g_1(u,u',s)=0 \\
g_2(u,u',s) =0
\end{array}
\right\}
\]
Factor $\rho_3' = \rho_E \circ\rho_E'$ where $\rho_E' :X \to \BP^1\times\BP^1\times {\ov C_0^*}$ is defined by $\rho_E'(u,u',u'',s) = (u,u',s)$ and where $\rho_E :E\to {\ov C_0^*}$ is defined by $\rho_E(u,u',s)=s$.
Then $(\rho_E')\vert_U : U \to \BP^1\times\BP^1\times {\ov C_0^*}$ is an isomorphism while the fibre of $(\rho_E')\vert_F : F \to E$ is $\BP^1$. 
By the decomposition theorem,
\[
(\rho_E')_* \1_{X}[6]
=
\1_{Y}[6] \oplus \1_{E}[4].
\]
Therefore,
\[
(\rho'_3)_* \1_{X}[4]
=
(\pi)_*\1_{Y}[6] \oplus (\rho_E)_* \1_{E}[4].
\]
Now
\[
(\pi)_*\1_{Y}[6] = H^\bullet(\BP^1\times \BP^1)\otimes \1_{\ov C_0^*}[6]
=
\1_{\ov C_0^*}[6] \oplus 2\1_{\ov C_0^*}[4]\oplus \1_{\ov C_0^*}[2],
\]
while
\[
(\rho_E)_* \1_{E}[4] = \IC(\1_{C_0^*}) \oplus \IC(\mathcal{E}_{C_0^*}) \oplus 2\IC(\1_{C_2^*}) \oplus \IC(\1_{C_3^*}). 
\]
Thus,
\[
\begin{array}{rcl}
&&\hskip-8mm (\rho'_3)_* \1_{X}[6] \\
&=& 
\left( H^\bullet(\BP^1\times \BP^1)\otimes \1_{\ov C_0^*}[6]\right)\\
&& \oplus\ \IC(\1_{C_0^*}) \oplus \IC(\mathcal{E}_{C_0^*}) \oplus\ 2\IC(\1_{C_2^*}) \oplus \IC(\1_{C_3^*}) \\
&=& \IC(\1_{C_0^*})[-2]\oplus 2 \IC(\1_{C_0^*})[0] \oplus \IC(\1_{C_0^*})[2]\\
&& \oplus\ \IC(\1_{C_0^*}) \oplus \IC(\mathcal{E}_{C_0^*}) \oplus\ 2\IC(\1_{C_2^*}) \oplus \IC(\1_{C_3^*}) \\
&=& \IC(\1_{C_0^*})[-2]\oplus 3 \IC(\1_{C_0^*})[0] \oplus \IC(\1_{C_0^*})[2] \\
&& \oplus\  \IC(\mathcal{E}_{C_0^*}) \oplus\ 2\IC(\1_{C_2^*}) \oplus \IC(\1_{C_3^*}) 
\end{array}
\]
So,
\begin{equation}\label{eqn:LFt3}
\begin{array}{rcl}
&&\hskip-8mm (\rho_3''\times \id)_* \RPhi_{{\tilde f}_{3'',0}} [-1]\left( \1_{\wt C_3''}[4] \boxtimes \1_{\ov C_0^*}[4] \right) \\
&=&
\1_{C_0}\boxtimes (\rho'_3)_*\1_{\mathcal{S}}[6]\\
&=& \left( \1_{C_0}\boxtimes \IC(\1_{C_0^*})\right) [-2]\oplus 3  \left( \1_{C_0}\boxtimes \IC(\1_{C_0^*})\right) [0] \\
&& \oplus\  \left( \1_{C_0}\boxtimes \IC(\1_{C_0^*})\right) [2] \oplus\   \left( \1_{C_0}\boxtimes \IC(\mathcal{E}_{C_0^*})\right)[0] \\
&& \oplus\ 2 \left( \1_{C_0}\boxtimes \IC(\1_{C_2^*})\right)[0]  \oplus  \left( \1_{C_0}\boxtimes \IC(\1_{C_3^*})\right)[0].
\end{array}
\end{equation}
On the other hand,
\[
\begin{array}{rcl}
&&\hskip-8mm (\rho_3''\times \id)_* \RPhi_{{\tilde f}_{3'',0}}[-1] \left( \1_{\wt C_3''}[4] \boxtimes \1_{\ov C_0^*}[4] \right) \\
&=&
\RPhi_{\KPair{}{}}[-1] \left( \left( (\rho_3'')_* \1_{\wt C_3''}[4] \right) \boxtimes \1_{\ov C_0^*}[4] \right) \\
&=& \RPhi_{\KPair{}{}}[-1] \left( \left( \IC(\1_{C_3})\boxtimes \1_{\ov C_0^*}[4]\right) \oplus 2\left( \IC(\mathcal{R}_{C_3})\boxtimes \1_{\ov C_0^*}[4]\right) \right) \\
&& \oplus\ \RPhi_{\KPair{}{}}[-1] \left( \left(\IC(\mathcal{E}_{C_3}) \boxtimes \1_{\ov C_0^*}[4]\right) \oplus\  3 \left( \IC(\1_{C_0})\boxtimes \1_{\ov C_0^*}[4]\right) \right) \\
&& \oplus\  \RPhi_{\KPair{}{}}[-1] \left( \left( \IC(\1_{C_0})[2] \boxtimes \1_{\ov C_0^*}[4]\right) \oplus \left( \IC(\1_{C_0})[-2] \boxtimes \1_{\ov C_0^*}[4]\right)\right)\\
\end{array}
\]
Now use \eqref{eqn:Ft0}, \eqref{eqn:Ft3} and \eqref{eqn:Ft3'} to conclude
\begin{equation}\label{eqn:RFt3}
\begin{array}{rcl}
&&\hskip-8mm
\RPhi_{\KPair{}{}}[-1] \left( \left( (\rho_3'')_* \1_{\wt C_3''}[4] \right) \boxtimes \1_{\ov C_0^*}[4] \right) \\
&=&  \left( \1_{C_0}\boxtimes \IC(\1_{C_3^*}) \right) \oplus 2\left( \1_{C_0}\boxtimes  \IC(\1_{C_2^*})\right) \\
&& \oplus\ \RPhi_{\KPair{}{}}[-1]  \left(\IC(\mathcal{E}_{C_3}) \boxtimes \1_{\ov C_0^*}[4]\right) \oplus\ 3 \left( \1_{C_0}\boxtimes \IC(\1_{C_0^*}) \right) \\
&& \oplus\  \left(  \1_{C_0}\boxtimes \IC(\1_{C_0^*})\right)[2] \oplus \left(  \1_{C_0}\boxtimes \IC(\1_{C_0^*})\right)[-2]\\
\end{array}
\end{equation}
Comparing \eqref{eqn:RFt3} with \eqref{eqn:LFt3} gives
\begin{equation}\label{eqn:Ft3''}
\RPhi_{\KPair{}{}}[-1]  \left(\IC(\mathcal{E}_{C_3}) \boxtimes \1_{\ov C_0^*}[4]\right) 
=
\1_{C_0}\boxtimes \IC(\mathcal{E}_{C_0^*})
\end{equation}
It follows that
\[
\Ft \IC(\mathcal{E}_{C_3})  = \IC(\mathcal{E}_{C_0^*}) .
\]

Returning to \eqref{eqn:RFt3} we now see
\[
\begin{array}{rcl}
&&\hskip-8mm 
\Evs_{C_0}\left( (\rho_3'')_* \1_{\wt C_3''}[4] \right) \\
&=&
\RPhi_{\KPair{}{}}[-1] \left( \left( (\rho_3'')_* \1_{\wt C_3''}[4] \right) \boxtimes \1_{\ov C_0^*}[4] \right)\vert_{\Lambda_0^\mathrm{reg}}[-4] \\ 
&=&   \left( \1_{C_0}\boxtimes \IC(\1_{C_3^*}) \right)\vert_{\Lambda_0^\mathrm{reg}}[-4]  \oplus 2\left( \1_{C_0}\boxtimes  \IC(\1_{C_2^*}) \right)\vert_{\Lambda_0^\mathrm{reg}}[-4]  \\
&& \oplus\ \left( \1_{C_0}\boxtimes \IC(\mathcal{E}_{C_0^*})\right)\vert_{\Lambda_0^\mathrm{reg}}[-4]  \oplus\ 3 \left( \1_{C_0}\boxtimes \IC(\1_{C_0^*}) \right)\vert_{\Lambda_0^\mathrm{reg}}[-4]  \\
&& \oplus\  \left(  \1_{C_0}\boxtimes \IC(\1_{C_0^*})\right)\vert_{\Lambda_0^\mathrm{reg}}[-2] \oplus \left(  \1_{C_0}\boxtimes \IC(\1_{C_0^*})\right)\vert_{\Lambda_0^\mathrm{reg}}[-6] \\
&=& \mathcal{E}_{\Lambda_0^\mathrm{reg}}\oplus 3\1_{\Lambda_0^\mathrm{reg}} \oplus \1_{\Lambda_0^\mathrm{reg}}[2] \oplus \1_{\Lambda_0^\mathrm{reg}}[-2].
\end{array}
\]
On the other hand,
\[
\begin{array}{rcl}
&&\hskip-8mm 
\Evs_{C_0}\left( (\rho_3'')_* \1_{\wt C_3''}[4] \right) \\
&=& 
\Evs_{C_0}\left( 
\IC(\1_{C_3})\oplus 2\IC(\mathcal{R}_{C_3})\oplus \IC(\mathcal{E}_{C_3}) \right) \\
&&\ \oplus\ \Evs_{C_0}\left(  3 \IC(\1_{C_0}) \oplus\  \IC(\1_{C_0})[2] \oplus \IC(\1_{C_0})[-2]\right) \\
&=& 
\Evs_{C_0}\IC(\mathcal{E}_{C_3}) \oplus 3 \1_{\Lambda_0^\mathrm{reg}} \oplus  \1_{\Lambda_0^\mathrm{reg}}[2] \oplus \1_{\Lambda_0^\mathrm{reg}}[-2].
\end{array}
\]
Therefore,
\[
\Evs_{C_0}\IC(\mathcal{E}_{C_3}) 
=
\mathcal{E}_{\Lambda_0^\mathrm{reg}}.
\]

\subsubsection{$\Evs_{C_1}\IC(\mathcal{E}_{C_3})$}\label{sssec:Evs13E}

Take a point $[\lambda,(u_1,u_2),(u_3,u_4),(u_5,u_6)]\in \wt C_3 $ and its image $r(x,y)=\lambda (u_1y-u_2x)(u_3y-u_4x)(u_5y-u_6x)$ in $\ov C_3$. For $s\in \ov{C_1^*}$, we write $s(x,y)=s_0y^3-3s_1y^2x-3s_2yx^2-s_3x^3$, we get 
\[
\begin{array}{rcl}
&&\hskip-8mm {\widetilde f}_{3'',1}([\lambda,(u_1,u_2),(u_3,u_4),(u_5,u_6),s) \\
&=& \lambda( u_1u_3u_5s_0+(u_1u_3u_6+u_2u_3u_5+u_1u_4u_5)s_1\\
&& -(u_2u_4u_5+u_2u_3u_6+u_1u_4u_6)s_2+u_2u_4u_6s_3 ).
\end{array}
\]

For simplicity, consider the affine part $u_1=u_3=1$. Then we have 
$$\widetilde f_{3'',1}=\lambda(s_0+(u_2+u_4+u_6)s_1-(u_2u_4+u_2u_6+u_4u_6)s_2+u_2u_4u_6s_3).$$
Denote $z_2=u_4-u_2,z_3=u_6-u_2$. Note that $z_2,z_3$ are the equations which define $(\rho_{3}'')^{-1}(C_1\times C_1^*)$ inside $\widetilde C_3''\times C_1^*$.
A simple calculation shows that
\begin{align*}
\wt f_{3'',1}&=\lambda(s_0+3u_2s_1-3u_2^2s_2+u_2^3s_3) +\lambda z_2(s_1-2u_2s_2+u_2^2s_3)\\
& +\lambda z_3 (s_1-2u_2s_2+u_2^2s_3)+\lambda z_2z_3(s_2-u_2s_3).
\end{align*}
The condition $[r,s]=0$ is given by an equation $z_1$ which is equivalent to that $s$ vanishes of order $2$ at $(u_2,-1)$. We actually have $z_1^2|(s_0+3u_2s_1-3u_2^2s_2+u_2^3s_3)$ and $z_1|(s_1-2u_2s_2+u_2^2s_3)$. In fact, if we take $s(x,y)=(v_1y-v_2x)^2(v_3y-v_4x)$, then we can take $z_1=v_1+v_2u_2$. We then have
$$s_0+3u_2s_1-3u_2^2s_2+u_2^3s_3=-s(u_2,-1)=z_1^2(v_3+v_4u_2),$$
and 
$$s_1-2u_2s_2+u_2^2s_3=-\frac{1}{3}\left.\frac{\partial s}{\partial x}\right|_{x=u_2,y=-1} =\frac{1}{3}z_1(2v_2v_3+v_1v_4+3v_2u_2v_4). $$
Denote by $g_1=\lambda(v_3+v_4u_2), g_{12}=\frac{\lambda}{3}( 2v_2v_3+v_1v_4+3v_2u_2v_4)$ and $g_{23}=\lambda (s_2-u_2s_3)=\frac{\lambda}{3}(2v_1v_2v_4+v_2^2v_3+3u_2v_2^2v_4)$. We get that 
$$\wt f_{3'',1}=g_1z_1^2+g_{12}z_1(z_2+z_3)+g_{23}z_2z_3 .$$
Let $I$ be the ideal generated by $z_1,z_2,z_3$ in $\mathbb{C}[z_1,z_2,z_3]$. Then modulo $I$, we have 
\[
g_{12}=-\frac{2}{3}u_2^{-1}v_1g_1
\qquad\text{and}\qquad
g_{23}=\frac{1}{3}u_2^{-2}v_1^2g_1.
\]
Thus we get
$$\wt f_{3'',1}=g_1(z_1^2-\frac{2}{3}u_2^{-1}v_1z_1(z_2+z_3)+\frac{1}{3}u_2^{-2}v_1^2z_2z_3)\mod I^3.$$

By completing squares, we get that
\begin{align*}\wt f_{3'',1}&=g_1(z_1-\frac{1}{3}u_2^{-1}v_1(z_2+z_3))^2-\frac{1}{9}g_1u_2^{-2}v_2^2(z_2-\frac{1}{2}z_3)^2\\
&-\frac{1}{12}g_1u_2^{-2}v_1^2z_3^2  \mod I^3.
\end{align*}
Note that $g_1\notin I$ if and only if $s\in C_2^*$. In particular, if $s\in C_1^*$, we have $g_1\notin I$. It follows that

\[
\begin{array}{rcl}
\left( \RPhi_{{\tilde f}_{3'',1}} [-1] \1_{{\wt C_3''}\times {C_1^*}} \right)
\vert_{(\rho_3''\times\1_{C_1^*})^{-1}\Lambda_1^\mathrm{reg}} 
&=& 
\mathcal{T}_{(\rho_3''\times\1_{C_1^*})^{-1}\Lambda_1^\mathrm{reg}} [-3],
\end{array}
\]
and consequently we have,
\[
\begin{array}{rcl}
&&\hskip-8mm
\Evs_{C_1} \left( (\rho_3'')_* \1_{\wt C_3''}[4]\right) \\
&=&
\left(\RPhi_{\KPair{}{}}[-1] (\rho_3'')_* \1_{\wt C_3''}[4] \boxtimes \1_{C_1^*}[3] \right)\vert_{\Lambda_1^\mathrm{reg}}[-4] \\
&=&
\left(  (\rho_3''\times\id_{C_1^*})_* \RPhi_{{\tilde f}_{3,1}}[-1] \1_{\wt C_3''}[4] \boxtimes \1_{C_1^*}[3] \right)\vert_{\Lambda_1^\mathrm{reg}} [-4]\\
&=&
\left(  (\rho_3''\times\id_{C_1^*})_* \RPhi_{{\tilde f}_{3,1}}[-1] \1_{\wt C_3''} \boxtimes \1_{C_1^*} \right)\vert_{\Lambda_1^\mathrm{reg}} [3] \\
&=& 
(\rho_3''\times\id_{C_1^*})_* \left( \mathcal{T}_{(\rho_3''\times\1_{C_1^*})^{-1}\Lambda_1^\mathrm{reg}}\right) [3-3] \\
&=&
\mathcal{T}_{\Lambda_1^\mathrm{reg}}.
\end{array}
\]
On the other hand,
\[
\begin{array}{rcl}
&&\hskip-8mm
\Evs_{C_1} \left( (\rho_3'')_* \1_{\wt C_3''}[4]\right) \\
&=&
\Evs_{C_1} \left( \IC(\1_{C_3})\oplus 2\IC(\mathcal{R}_{C_3})\oplus \IC(\mathcal{E}_{C_3}) \right) \\
&&\ \oplus\ \Evs_{C_1} \left( 3 \IC(\1_{C_0}) \oplus\  \IC(\1_{C_0})[2] \oplus \IC(\1_{C_0})[-2]\right) \\
&=&
\Evs_{C_1} \IC(\mathcal{E}_{C_3}).
\end{array}
\]
Therefore,
\[
\Evs_{C_1} \IC(\mathcal{E}_{C_3})
=
\mathcal{T}_{\Lambda_1^\mathrm{reg}}.
\]

\subsubsection{$\Evs_{C_2}\IC(\mathcal{E}_{C_3})$}\label{sssec:Evs23E}

To compute $\Ev_{C_2}\IC(\mathcal{E}_{C_3})$, we use the cover $\rho_E:E\rightarrow \ov C_3$ defined at the bottom of Section~\ref{sssec:covC3''}.

Denoting $ ([u],[u'],r)\in E$ and $s= (v_1y-v_2x)^3\in C_2^* $ we consider the function 
\begin{align*}
\wt f_{E,2}([u],[u'],r,s )&=\KPair{r}{s}\\
&=r_0v_1^3+3v_1^2v_2r_1-3v_1v_2^2r_2+v_2^3r_3.
\end{align*}
For simplicity, we restrict to the affine cover when $v_1\ne 0$ and thus we can assume that $v_1=1$. Moreover, we work on the affine part when $u_1=u_3=1$.  From $g_1=g_2=0$, we can get that 
\begin{align*}
r_1&=(u_2+u_4)r_2-u_2u_4r_3,\\
r_0&=-(u_2u_4+u_2^2+u_4^2)r_2+(u_2+u_4)u_2u_4r_3.
\end{align*}
Thus we get that 
\begin{align*}
\wt f_{E,2}([u],[u'],r,s)&=-(u_2u_4+u_2^2+u_4^2 )r_2+(u_2+u_4)u_2u_4r_3\\
&+3(u_2+u_4)r_2v_2-3u_2u_4v_2r-3v_2^2 r_2+v_2^3r_3.
\end{align*}
Set $z_1=u_2-v_2, z_2=u_2-v_2$. We can check that 
$$\wt f_{E,2}=(z_1^2+z_1z_2+z_2^2)(r_3v_2-r_2)+r_3(z_1^2z_2+z_1z_2^2).$$
Let $I$ be the ideal generated by $z_1,z_2$ in $\mathbb{C}[z_1,z_2]$. We have
$$\wt f_{E,2} =(z_1^2+z_1z_2+z_2^2)(r_3v_2-r_2)\mod I^3.$$
We then get that 
$$\RPhi_{\wt f_{E,2}}[-1](\1_{E\times  C_2^*})=\1_{(\rho_E\times\id_{C_2^*})^{-1}\Lambda_2^\mathrm{reg}} [-2].$$
Therefore,
\[
\begin{array}{rcl}
&&\hskip-8mm 
\Evs_{C_2}(\rho_E)_*\1_E[4] \\
&=&
\left( (\rho_E\times\id_{C_2^*})_*(\RPhi_{\wt f_{E,2}}[-1](\1_{E}[4]\boxtimes\1_{C_2^*}[2])\right)\vert_{\Lambda_2^\mathrm{reg}}[-4] \\
&=&
\left( (\rho_E\times\id_{C_2^*})_*(\RPhi_{\wt f_{E,2}}[-1](\1_{E}\boxtimes\1_{C_2^*})\right)\vert_{\Lambda_2^\mathrm{reg}}[2] \\
&=&
(\rho_E\times\id_{C_2^*})_* \1_{(\rho_E\times\id_{C_2^*})^{-1}\Lambda_2^\mathrm{reg}} [2-2] \\
&=& \1_{\Lambda_2^{\mathrm{reg}}},
\end{array}
\]
since $\rho_E$ is one-to-one on $C_2$.
On the other hand,
\[
\begin{array}{rcl}
&&\hskip-8mm 
\Evs_{C_2}(\rho_E)_*\1_E [4] \\
&=& 
\Ev_{C_2} \left( \IC(\1_{C_3}) \oplus \IC(\mathcal{E}_{C_3}) \oplus 2\IC(\1_{C_1}) \oplus \IC(\1_{C_0})\right) \\
&=&
\Ev_{C_2} \IC(\mathcal{E}_{C_3}) .
\end{array}
\]
Therefore,
\[
\Evs_{C_2} \IC(\mathcal{E}_{C_3})
=
\1_{\Lambda_2^{\mathrm{reg}}}.
\]

\subsubsection{$\Evs_{C_3}\IC(\mathcal{E}_{C_3})$}\label{sssec:Evs33E}

Since ${\tilde f}_{3,3} : {\wt C_3''}\times{\ov C_3^*} \to S$ is trivial, we have
\[
\RPhi_{{\tilde f}_{3,3}}[-1] \1_{{\wt C_3''}\times{\ov C_3^*}}
=
\RPhi_{0}[-1] \1_{{\wt C_3''}\times{\ov C_3^*}}
=
\1_{{\wt C_3''}\times{\ov C_3^*}}.
\]
This determines \eqref{eqn:VtC} in this case.

Therefore,
\[
\begin{array}{rcl}
&& \hskip-8mm (\rho_3''\times \id)_* \left(\RPhi_{{\tilde f}_{3,3}}[-1] \1_{{\wt C_3''}\times{\ov C_3^*}}\right) \\
&=& (\rho_3''\times \id)_* \left( \1_{\wt C_3''}\boxtimes \1_{\ov C_3^*}\right) \\
&=& \left( (\rho_3'')_* \1_{\wt C_3''}\right) \boxtimes \1_{\ov C_3^*} \\
&=& \left( \IC(\1_{C_3})\boxtimes \1_{\ov C_3^*}\right)[-4]  \oplus  2 \left( \IC(\mathcal{R}_{C_3})\boxtimes \1_{\ov C_3^*}\right)[-4] \\
&& \oplus\  \left( \IC(\mathcal{E}_{C_3})\boxtimes \1_{\ov C_3^*}\right)[-4]     \oplus\ 3  \left(\IC(\1_{C_0})\boxtimes \1_{\ov C_3^*}\right)[-4]  \\
&& \oplus\   \left( \IC(\1_{C_0})[2] \boxtimes \1_{\ov C_3^*}\right)[-4]  \oplus  \left( \IC(\1_{C_0})[-2]\boxtimes \1_{\ov C_3^*}\right)[-4] ,
\end{array}
\]
using the expression for $(\rho_3'')_* \1_{\wt C_3''}[4]$ from \eqref{eqn:pi''3}.
Now restrict to ${\ov \Lambda_3} = {\ov C_3} \times {\ov C_3^*} = P_3[x,y]\times C_3^*$ to find
\[
\begin{array}{rcl}
&& \hskip-8mm (\rho_3''\times \id)_* \left(\RPhi_{{\tilde f}_{3,3}}[-1] \1_{{\wt C_3''}\times{\ov C_3^*}}\right)\vert_{\ov \Lambda_3} \\
&=& \left( \IC(\1_{C_3})\boxtimes \1_{\ov C_3^*}\right)\vert_{\ov \Lambda_3}[-4]  \oplus  2 \left( \IC(\mathcal{R}_{C_3})\boxtimes \1_{\ov C_3^*}\right) \vert_{\ov \Lambda_3}[-4] \\
&& \oplus\  \left( \IC(\mathcal{E}_{C_3})\boxtimes \1_{\ov C_3^*}\right)\vert_{\ov \Lambda_3}[-4] .
\end{array}
\]
Finally, restrict to $\Lambda_3^\mathrm{reg} = C_3 \times C_3^*$:
\[
 (\rho_3''\times \id)_* \left(\RPhi_{{\tilde f}_{3,3}} \1_{{\wt C_3''}\times{\ov C_3^*}}\right)[4]\vert_{\Lambda_3^\mathrm{reg}} \\
= \1_{\Lambda_3^\mathrm{reg}}[4]  \oplus  2 \mathcal{R}_{\Lambda_3^\mathrm{reg}}[4] \oplus  \mathcal{E}_{\Lambda_3^\mathrm{reg}}[4].
\]
Therefore,
\[
\Evs_{C_3} \left( (\rho_3'')_* \1_{\wt C_3''}[4] \right) = 
\1_{\Lambda_3^\mathrm{reg}}  \oplus  2 \mathcal{R}_{\Lambda_3^\mathrm{reg}} \oplus  \mathcal{E}_{\Lambda_3^\mathrm{reg}}.
\]

On the other hand,
\[
\begin{array}{rcl}
&&\hskip-8mm \Evs_{C_3} \left( (\rho_3'')_* \1_{\wt C_3''}[4] \right) \\
&=& \Evs_{C_3} \IC(\1_{C_3})\oplus 2 \Evs_{C_3} \IC(\mathcal{R}_{C_3})\oplus \Evs_{C_3} \IC(\mathcal{E}_{C_3})  \\
&&\hskip2mm\ \oplus\ 3 \Evs_{C_3} \IC(\1_{C_0}) \oplus\  \Evs_{C_3} \IC(\1_{C_0})[2] \oplus \Evs_{C_3} \IC(\1_{C_0})[-2] \\
&=& \1_{\Lambda_3^\mathrm{reg}}\oplus 2 \mathcal{R}_{\Lambda_3^\mathrm{reg}}\oplus \Evs_{C_3} \IC(\mathcal{E}_{C_3}) .
\end{array}
\]
Therefore,
\[
\Evs_{C_3} \IC(\mathcal{E}_{C_3})
=
\mathcal{E}_{\Lambda_3^\mathrm{reg}}.
\]

\subsection{Stalks and the geometric multiplicity matrix}\label{ssec:stalks}

\begin{table}
\caption{Restrictions of Simple Objects to Boundary}
\begin{center}
$
 \begin{array}{|c | l  l  l  l |}
\hline
 \mathcal{P} &  \mathcal{P}\vert_{C_0} &\mathcal{P}\vert _{C_1} & \mathcal{P}\vert_{C_2} & \mathcal{P}\vert_{C_3} \\
 \hline
 \IC(\1_{C_0})& \1_{C_0}[0] & 0 & 0 & 0 \\
 \IC(\1_{C_1})& \1_{C_0}[2] & \1_{C_1}[2] & 0 & 0 \\
 \IC(\1_{C_2})& \1_{C_0}[1]\oplus \1_{C_0}[3] & \1_{C_1}[3] & \1_{C_2}[3] & 0 \\
 \IC(\1_{C_3})& \1_{C_0}[4] & \1_{C_1}[4] & \1_{C_2}[4] & \1_{C_3}[4] \\
 \IC(\mathcal{R}_{C_3}) & \1_{C_0}[2] & 0 & \1_{C_2}[4] & \mathcal{R}_{C_3}[4] \\
 \IC(\mathcal{E}_{C_3}) & 0 & 0 & 0 & \mathcal{E}_{C_3}[4]\\
 \hline
 \end{array}
 $
\end{center}
\label{table:stalks}
\end{table}%

\begin{table}
\caption{The geometric multiplicity matrix for equivariant perverse sheaves with infinitesimal parameter $\lambda_\mathrm{sub}$. Here we use the notation $\mathcal{L}_C^\sharp$ for the constructible sheaf complex defined by $\IC(\mathcal{L}_C) = \mathcal{L}_C^\sharp[\dim C]$ and $\mathcal{L}_C^! \ceq j_!\mathcal{L}_C$ where $j : C \to V$ is inclusion.}
\begin{center}
$
 \begin{array}{|c | ccccc |c |}
\hline
 &  \1^!_{C_0} & \1^!_{C_1} & \1^!_{C_2} & \1^!_{C_3} & \mathcal{R}^!_{C_3} & \mathcal{E}^!_{C_3} \\
 \hline
\1^\sharp_{C_0} & 1 & 0 & 0 & 0 & 0 & 0 \\
\1^\sharp_{C_1} & 1 & 1 & 0 & 0 & 0 & 0 \\
\1^\sharp_{C_2} & 2 & 1 & 1 & 0 & 0 & 0 \\
\1^\sharp_{C_3} & 1 & 1 & 1 & 1& 0 & 0 \\
\mathcal{R}^\sharp_{C_3} & 1 & 0 & 1 & 0 & 1 & 0 \\
\hline
\mathcal{E}^\sharp_{C_3} & 0 & 0 & 0 & 0 & 0 & 1 \\
 \hline
 \end{array}
$
\end{center}
\label{table:geomult}
\end{table}%

\begin{theorem}\label{thm:stalks}
The stalks of the simple objects in $\Perv_{\GL_2}(P_3[x,y])$ are given by Table~\ref{table:stalks}.
The geometric multiplicity matrix is given by Table~\ref{table:geomult}.
 \end{theorem}

\begin{proof}
In Sections~\ref{sssec:covC1}, \ref{sssec:covC2}, \ref{sssec:covC3} and \ref{sssec:covC3''}, we found:
\[
\begin{array}{rclr}

(\rho_1)_*(\1_{\widetilde C_1}[2])&=&\IC(\1_{C_0})\oplus \IC(\1_{C_1}) & \\
(\rho_2)_*(\1_{\widetilde C_2}[3])&=&\IC(\1_{C_2}),\\
(\rho_3)_*(\1_{\widetilde C_3}[4])&=&\IC(\1_{C_3})\oplus \IC(\mathcal{R}_{C_3}), & \\
(\rho''_3)_*(\1_{\widetilde C_3''}[4]) &=& \IC(\1_{C_3})\oplus 2\IC(\mathcal{R}_{C_3})\oplus \IC(\mathcal{E}_{C_3})   & \\
&& \oplus \ 3 \IC(\1_{C_0}) \oplus\  \IC(\1_{C_0})[2] \oplus \IC(\1_{C_0})[-2]. & 
\end{array}
\]
Using this, we compute the geometric multiplicity matrix as follows.
The fact that the only non-trivial sheaves are on the open orbit $C_3$ simplifies this as computing the multiplicity of $\1_{C_i}$ in $\mathcal{P}|_{C_i}$ thus reduces to computing the rank of the stalks.

We may compute the ranks of the following stalks geometrically using proper base change and our computations of the stalks above.
\[
\begin{array}{|c | c  c  c| }
\hline
 \mathcal{P} & {\rm rank}(\mathcal{P}|_{C_2}) &{\rm rank}(\mathcal{P}|_{C_1})  & {\rm rank}(\mathcal{P}|_{C_0})  \\
 \hline
 (\rho_1)_*(\1_{\widetilde C_1}[2]) & 0 & 1 & 2 \\
 (\rho_2)_*(\1_{\widetilde C_2}[3]) & 1 & 1 & 2\\
 (\rho_3)_*(\1_{\widetilde C_3}[4]) & 2 & 1 & 2 \\
 (\rho_3'')_*(\1_{\widetilde C_3}[4]) & 3 & 1 & 8\\
 \hline
 \end{array}
 \]
 We may then solve for the ranks of each of the $\IC$ sheaves, in the order below using what we know from the decomposition of the above sheaves. Note that $\IC(\1_{C_0}) = \1_{C_0}$ and $\IC(\1_{C_3}) = \1_{\ov C_3}[4]$.
 
 \[
\begin{array}{|c | c  c  c |}
\hline
 \mathcal{P} & {\rm rank}(\mathcal{P}|_{C_2}) &{\rm rank}(\mathcal{P}|_{C_1})  & {\rm rank}(\mathcal{P}|_{C_0})  \\
 \hline
 \IC(\1_{C_0})& 0 & 0 & 1 \\
 \IC(\1_{C_1})& 0 & 1 & 1\\
 \IC(\1_{C_2})& 1 & 1 & 2 \\
 \IC(\1_{C_3})& 1 & 1 & 1\\
 \IC(\mathcal{R}_{C_3}) & 1 & 0 & 1\\
 \IC(\mathcal{E}_{C_3}) & 0 & 0 & 0 \\
 \hline
 \end{array}
 \]
 To understand how to find the above notice that we may compute the row for $ \IC(\mathcal{E}_{C_3}) $ by looking at $ \IC(\mathcal{E}_{C_3})$ which we know equals
 \[ (\pi_3'')_*(\1_{\widetilde C_3''}[4]) - \IC(\1_{C_3})- 2\IC(\mathcal{R}_{C_3})- 3 \IC(\1_{C_0}) -  \IC(\1_{C_0})[2] - \IC(\1_{C_0})[-2] \]
 in the Grothendieck group.
 \end{proof}

\subsection{Normalised microlocal vanishing cycles functor}\label{ssec:NEvs}

\begin{theorem}\label{thm:Evs}
On simple objects, the microlocal vanishing cycles functor $\Evs$ is given by:
\[
\begin{array}{rcl}
\Perv_{\GL_2}(P_3[x,y]) &\mathop{\longrightarrow}\limits^{\Evs} & \Loc_{\GL_2}(\Lambda^\mathrm{reg})\\
\IC(\1_{C_{0}}) &\mapsto& \1_{\Lambda^\mathrm{reg}_0}  \\
\IC(\1_{C_{1}}) &\mapsto& \mathcal{T}_{\Lambda^\mathrm{reg}_1} \oplus \mathcal{R}_{\Lambda^\mathrm{reg}_0} \\
\IC(\1_{C_{2}}) &\mapsto& {\mathcal{T}}_{\Lambda^\mathrm{reg}_2} \oplus \1_{\Lambda^\mathrm{reg}_1}   \\
\IC(\1_{C_{3}}) &\mapsto& \1_{\Lambda^\mathrm{reg}_3}  \\
\IC(\mathcal{R}_{C_{3}}) &\mapsto& {\mathcal{R}}_{\Lambda^\mathrm{reg}_3} \oplus \1_{\Lambda^\mathrm{reg}_2} \\
\IC(\mathcal{E}_{C_{3}}) &\mapsto& {\mathcal{R}}_{\Lambda^\mathrm{reg}_3}  \oplus \mathcal{T}_{\Lambda^\mathrm{reg}_2} \oplus \1_{\Lambda^\mathrm{reg}_1} \oplus {\mathcal{E}}_{\Lambda^\mathrm{reg}_0},
\end{array}
\]
where $\1_{\Lambda_{i}^\mathrm{reg}}$, $\mathcal{T}_{\Lambda_{1}^\mathrm{reg}}$, $\mathcal{T}_{\Lambda_{2}^\mathrm{reg}}$, $\mathcal{R}_{\Lambda_3^\mathrm{reg}}$ and $\mathcal{E}_{\Lambda_3^\mathrm{reg}}$ appear in Section~\ref{ssec:conormal}.
This information is also summarised in Table~\ref{table:Evs}.
\end{theorem}

\begin{proof}
\begin{enumerate}
\item
$\Evs_{C_0} \IC(\1_{C_0}) = \1_{\Lambda_0^\mathrm{reg}}$ by Section~\ref{ssec:covC0}
\item $\Evs_{C_i} \IC(\1_{C_0}) = 0$ for $i=1$, $2$ and $3$ by \cite{CFMMX}*{Proposition 7.5.1}
\item $\Evs_{C_0} \IC(\1_{C_1}) = \mathcal{R}_{\Lambda_0^\mathrm{reg}}$ by Section~\ref{sssec:Evs01}
\item $\Evs_{C_1} \IC(\1_{C_1}) = \mathcal{T}_{\Lambda_1^\mathrm{reg}}$ by Section~\ref{sssec:Evs11}
\item $\Evs_{C_i} \IC(\1_{C_1}) = 0$ for $i=2$ and $3$ by \cite{CFMMX}*{Proposition 7.5.1}
\item $\Evs_{C_0} \IC(\1_{C_2}) = 0$ by Section~\ref{sssec:Evs02}
\item $\Evs_{C_1} \IC(\1_{C_2}) = 1_{\Lambda_1^\mathrm{reg}}$ by Section~\ref{sssec:Evs12}
\item $\Evs_{C_2} \IC(\1_{C_2}) = \mathcal{L}_{\Lambda_2^\mathrm{reg}}$ by Section~\ref{sssec:Evs22}
\item $\Evs_{C_3} \IC(\1_{C_2}) = 0$ by \cite{CFMMX}*{Proposition 7.5.1}
\item $\Evs_{C_0} \IC(\1_{C_3}) = 0$ by Section~\ref{sssec:Evs03}
\item $\Evs_{C_1} \IC(\1_{C_3}) = 0$ by Section~\ref{sssec:Evs13}
\item $\Evs_{C_2} \IC(\1_{C_3}) = 0$ by Section~\ref{sssec:Evs23}
\item $\Evs_{C_3} \IC(\1_{C_3}) = \1_{\Lambda_3^\mathrm{reg}}$ by Section~\ref{sssec:Evs33}
\item $\Evs_{C_0} \IC(\mathcal{R}_{C_3}) = 0$ by Section~\ref{sssec:Evs03R}
\item $\Evs_{C_1} \IC(\mathcal{R}_{C_3}) = 0$ by Section~\ref{sssec:Evs13R}
\item $\Evs_{C_2} \IC(\mathcal{R}_{C_3}) = \1_{\Lambda_2^\mathrm{reg}}$ by Section~\ref{sssec:Evs23R}
\item $\Evs_{C_3} \IC(\mathcal{R}_{C_3}) = \mathcal{R}_{\Lambda_3^\mathrm{reg}}$ by Section~\ref{sssec:Evs33R}
\item $\Evs_{C_0} \IC(\mathcal{E}_{C_3}) = 0$ by Section~\ref{sssec:Evs03E}
\item $\Evs_{C_1} \IC(\mathcal{E}_{C_3}) = 0$ by Section~\ref{sssec:Evs13E}
\item $\Evs_{C_2} \IC(\mathcal{E}_{C_3}) = 0$ by Section~\ref{sssec:Evs23E}
\item $\Evs_{C_3} \IC(\mathcal{E}_{C_3}) = \mathcal{E}_{\Lambda_3^\mathrm{reg}}$ by Section~\ref{sssec:Evs33E}\qedhere
\end{enumerate}
\end{proof}

\begin{table}
\caption{Values of $\Evs$ on Simple Objects}
\begin{center}
$
\begin{array}{| c || cccc |}
\hline
\Perv_{\GL_2}(P_3[x,y]) & \Loc_{\GL_2}(\Lambda_0^\mathrm{reg}) &  \Loc_{\GL_2}(\Lambda_1^\mathrm{reg}) &  \Loc_{\GL_2}(\Lambda_2^\mathrm{reg}) &  \Loc_{\GL_2}(\Lambda_3^\mathrm{reg})  \\
\hline\hline
\IC(\1_{C_0}) & \1_{\Lambda_0^\mathrm{reg}} & 0 & 0 & 0 \\
\IC(\1_{C_1})  & \mathcal{R}_{\Lambda_0^\mathrm{reg}} & \mathcal{T}_{\Lambda_1^\mathrm{reg}} & 0 & 0 \\
\IC(\1_{C_2}) & 0 & \1_{\Lambda_1^\mathrm{reg}} & \mathcal{T}_{\Lambda_2^\mathrm{reg}} & 0 \\
\IC(\1_{C_3}) & 0 & 0 & 0 & \1_{\Lambda_3^\mathrm{reg}} \\
\IC(\mathcal{R}_{C_3}) &  0 & 0 & \1_{\Lambda_2^\mathrm{reg}} & \mathcal{R}_{\Lambda_3^\mathrm{reg}} \\
\IC(\mathcal{E}_{C_3}) & \mathcal{E}_{\Lambda_0^\mathrm{reg}} & \mathcal{T}_{\Lambda_1^\mathrm{reg}} & \1_{\Lambda_2^\mathrm{reg}} & \mathcal{E}_{\Lambda_3^\mathrm{reg}} \\
\hline
\end{array}
$
\end{center}
\label{table:Evs}
\end{table}%

From \cite{CFMMX}*{Section 7.10}, we recall the normalised microlocal vanishing cycles functor
\[
\NEvs_C : \Perv_{H}(V) \to \Loc_{H}(T^*_{C}\Lambda^\mathrm{reg}),
\]
defined by
\[
\NEvs_C \mathcal{F} \ceq  \operatorname{\mathcal{H}\mathrm{om}}
\left( 
\Evs_C \IC(C),
\Evs_C \mathcal{F}
\right)
= \left(\Evs_C \IC(C)\right)^\vee \otimes \Evs_C \mathcal{F}.
\]

\begin{theorem}\label{thm:NEvs}
On simple objects, the normalised microlocal vanishing cycles functor $\NEvs$ is given by:
\[
\begin{array}{rcl}
\Perv_{\GL_2}(P_3[x,y]) &\mathop{\longrightarrow}\limits^{\NEvs} & \Loc_{\GL_2}(\Lambda^\mathrm{reg})\\
\IC(\1_{C_{0}}) &\mapsto& \1_{\Lambda^\mathrm{reg}_0}  \\
\IC(\1_{C_{1}}) &\mapsto& \1_{\Lambda^\mathrm{reg}_1} \oplus \mathcal{R}_{\Lambda^\mathrm{reg}_0} \\
\IC(\1_{C_{2}}) &\mapsto& \1_{\Lambda^\mathrm{reg}_2} \oplus \mathcal{T}_{\Lambda^\mathrm{reg}_1}   \\
\IC(\1_{C_{3}}) &\mapsto& \1_{\Lambda^\mathrm{reg}_3}  \\
\IC(\mathcal{R}_{C_{3}}) &\mapsto& {\mathcal{R}}_{\Lambda^\mathrm{reg}_3} \oplus \mathcal{T}_{\Lambda^\mathrm{reg}_2} \\
\IC(\mathcal{E}_{C_{3}}) &\mapsto& {\mathcal{E}}_{\Lambda^\mathrm{reg}_3}  \oplus \1_{\Lambda^\mathrm{reg}_2} \oplus \mathcal{T}_{\Lambda^\mathrm{reg}_1} \oplus {\mathcal{E}}_{\Lambda^\mathrm{reg}_0}
\end{array}
\]
This information is also summarised in Table~\ref{table:NEvs}.
\end{theorem}

\begin{proof}
From Table~\ref{table:Evs} we see that $\Evs_{C_i} \IC(\1_{C_i}) = \1_{\Lambda_i^\mathrm{reg}}$ for $i=0,3$ and $\Evs_{C_i} \IC(\1_{C_i}) = \mathcal{T}_{\Lambda_i^\mathrm{reg}}$ for $i=1,2$.
Table~\ref{table:NEvs} is therefore obtained from Table~\ref{table:Evs} by tensoring the columns $\Loc_{\GL_2}(\Lambda_1^\mathrm{reg})$ and $\Loc_{\GL_2}(\Lambda_2^\mathrm{reg})$ by $\mathcal{T}_{\Lambda_1^\mathrm{reg}}$ and $\mathcal{T}_{\Lambda_2^\mathrm{reg}}$, respectively.
\end{proof}

\begin{table}
\caption{Values of $\NEvs$ on Simple Objects}
\begin{center}
$
\begin{array}{| c || cccc |}
\hline
\Perv_{\GL_2}(P_3[x,y]) & \Loc_{\GL_2}(\Lambda_0^\mathrm{reg}) &  \Loc_{\GL_2}(\Lambda_1^\mathrm{reg}) &  \Loc_{\GL_2}(\Lambda_2^\mathrm{reg}) &  \Loc_{\GL_2}(\Lambda_3^\mathrm{reg})  \\
\hline\hline
\IC(\1_{C_0}) & \1_{\Lambda_0^\mathrm{reg}} & 0 & 0 & 0 \\
\IC(\1_{C_1})  & \mathcal{R}_{\Lambda_0^\mathrm{reg}} & \1_{\Lambda_1^\mathrm{reg}} & 0 & 0 \\
\IC(\1_{C_2}) & 0 & \mathcal{T}_{\Lambda_1^\mathrm{reg}} & \1_{\Lambda_2^\mathrm{reg}} & 0 \\
\IC(\1_{C_3}) & 0 & 0 & 0 & \1_{\Lambda_3^\mathrm{reg}} \\
\IC(\mathcal{R}_{C_3}) &  0 & 0 &\mathcal{T}_{\Lambda_2^\mathrm{reg}} & \mathcal{R}_{\Lambda_3^\mathrm{reg}} \\
\IC(\mathcal{E}_{C_3}) & \mathcal{E}_{\Lambda_0^\mathrm{reg}} & \1_{\Lambda_1^\mathrm{reg}}  & \mathcal{T}_{\Lambda_2^\mathrm{reg}} & \mathcal{E}_{\Lambda_3^\mathrm{reg}} \\
\hline
\end{array}
$
\end{center}
\label{table:NEvs}
\end{table}%

\subsection{Fourier transform}\label{ssec:Ft}

\begin{theorem}\label{thm:geoFt}
On the simple objects in $\Perv_{\GL_2}(P_3[x,y])$, the Fourier transform is given by
\[
\begin{array}{rcl}
\Ft \IC(\1_{C_0}) 	&=& \IC(\1_{C_0^*})\\
\Ft \IC(\1_{C_1})	&=& \IC(\mathcal{R}_{C_0^*})\\
\Ft \IC(\1_{C_2})	&=& \IC(\1_{C_1^*})\\
\Ft \IC(\1_{C_3})	&=& \IC(\1_{C_3^*})\\
\Ft \IC(\mathcal{R}_{C_3}) 	&=& \IC(\1_{C_2^*})\\
\Ft \IC(\mathcal{E}_{C_3}) 	&=& \IC(\mathcal{E}_{C_0^*}).
\end{array}
\]
\end{theorem}

\begin{proof}
\begin{enumerate}
\item $\Ft \IC(\1_{C_0}) = \IC(\1_{C_0^*})$ by Section~\ref{ssec:covC0}
\item $\Ft \IC(\1_{C_1}) = \IC(\mathcal{R}_{C_0^*})$ by Section~\ref{sssec:Evs01}
\item $\Ft \IC(\1_{C_2}) = \IC(\1_{C_1^*})$ by Section~\ref{sssec:Evs02} 
\item $\Ft \IC(\1_{C_3}) = \IC(\1_{C_3^*})$ by Section~\ref{sssec:Evs03}
\item $\Ft \IC(\mathcal{R}_{C_3}) = \IC(\1_{C_2^*})$ by Section~\ref{sssec:Evs03R} 
\item $\Ft \IC(\mathcal{E}_{C_3}) = \IC(\mathcal{E}_{C_0^*})$ by Section~\ref{sssec:Evs03E} \qedhere
\end{enumerate} 
\end{proof}

\end{document}